\documentclass[12pt]{article}

\usepackage[latin1]{inputenc}
\usepackage{bbm}
\usepackage{stmaryrd}
\usepackage{latexsym}
\usepackage{amsfonts}
\usepackage{amsmath,amssymb,amscd}
\usepackage{yfonts}
\usepackage{dsfont}
\usepackage[dvips]{graphicx}
\usepackage{epsfig}
\usepackage{psfrag}
\usepackage[font={small}]{caption}

\DeclareFontFamily{U}{txsyc}{}
\DeclareFontShape{U}{txsyc}{m}{n}{
   <-> txsyc%
}{}
\DeclareFontShape{U}{txsyc}{bx}{n}{
   <-> txbsyc%
}{}
\DeclareFontShape{U}{txsyc}{l}{n}{<->ssub * txsyc/m/n}{}
\DeclareFontShape{U}{txsyc}{b}{n}{<->ssub * txsyc/bx/n}{}
\DeclareSymbolFont{symbolsC}{U}{txsyc}{m}{n}
\SetSymbolFont{symbolsC}{bold}{U}{txsyc}{bx}{n}
\DeclareFontSubstitution{U}{txsyc}{m}{n}
\DeclareMathSymbol{\df}{\mathrel}{symbolsC}{"42}
\DeclareMathSymbol{\fd}{\mathrel}{symbolsC}{"43}
\DeclareMathSymbol{\lJoin}{\mathrel}{symbolsC}{"58}
\DeclareMathSymbol{\rJoin}{\mathrel}{symbolsC}{"59}

\newcommand{\cC}{{\cal C}}
\newcommand{\cD}{{\cal D}}

\newcommand{\ES}{\mathbb{E}}

\newcommand{\cO}{{\cal O}}

\newcommand{\cS}{{\cal S}}

\newcommand{\ER}{\mathbb{R}}

\newcommand{\PE}{\mathbb{P}}

\newcommand{\EE}{\mathbb{E}}

\newcommand{\LL}{\mathbb{L}}
\newcommand{\NN}{\mathbb{N}}
\newcommand{\PP}{\mathbb{P}}
\newcommand{\QQ}{\mathbb{Q}}
\newcommand{\RR}{\mathbb{R}}

\newcommand{\EEE}{\mathrm{\bf E}}
\newcommand{\PPP}{\mathrm{\bf P}}

\newcommand{\iy}{\infty}
\newcommand{\lt}{\left}
\newcommand{\me}{\medskip}

\newcommand{\pa}{\partial}
\newcommand{\ri}{\rightarrow}
\newcommand{\rt}{\right}
\newcommand{\sm}{\smallskip}
\newcommand{\tr}{\triangle}

\newcommand{\wi}{\widetilde}
\newcommand{\wit}{\widehat}

\newcommand{\fo}{\forall\ }

\newcommand{\Id}{\mathrm{Id}}

\newcommand{\lve}{\lt\vert}
\newcommand{\lVe}{\lt\Vert}

\newcommand{\rve}{\rt\vert}
\newcommand{\rVe}{\rt\Vert}

\newcommand{\st}{\,:\,}

\newcommand{\trace}{\mathrm{tr}}

\newcommand{\un}{\mathds{1}}

\newcommand{\bq}{\begin{eqnarray*}}
\newcommand{\bqn}[1]{\begin{eqnarray}\label{#1}}
\newcommand{\eq}{\end{eqnarray*}}
\newcommand{\eqn}{\end{eqnarray}}
\newcommand{\wwtbp}{\par\hfill $\blacksquare$\par\me\noindent}
\newcommand{\thistitlepagestyle}{}
\newcommand{\lin}{\llbracket}
\newcommand{\rin}{\rrbracket}

\newcommand{\ttsim}{\raise.17ex\hbox{$\scriptstyle\mathtt{\sim}$}}

\newtheorem{pro}{Proposition} 
\newtheorem{cor}[pro]{Corollary}
\newtheorem{lem}[pro]{Lemma}
\newtheorem{theo}[pro]{Theorem}
\renewcommand{\thepro}{\arabic{pro}}

\newenvironment{rem}
{\par\me\refstepcounter{pro}\noindent{\bf Remark \thepro\ }}
{\par\hfill $\square$\par\me\noindent}
\newcommand{\proof}{\par\me\noindent\textbf{Proof}\par\sm\noindent}
\newcommand{\prooff}[1]{\par\me\noindent\textbf{#1}\par\sm\noindent}

\usepackage{color}

\title{A stochastic model for speculative bubbles}
\author{Sébastien Gadat, Laurent Miclo and Fabien Panloup
}
\date{\box1
}

\begin{document}

\setbox1=\vbox{
\large
\begin{center}
Institut de Mathématiques de Toulouse, UMR 5219\\
Université de Toulouse and CNRS, France\\
\end{center}
} 
\setbox3=\vbox{
\hbox{\{gadat, miclo, panloup\}@math.univ-toulouse.fr\\}
\vskip1mm
\hbox{Institut de Mathématiques de Toulouse\\}
\hbox{Université Paul Sabatier\\}
\hbox{118, route de Narbonne\\} 
\hbox{31062 Toulouse Cedex 9, France\\}
}
\setbox5=\vbox{
\box3
}

\maketitle
\thistitlepagestyle
\abstract{This paper aims to provide a simple modelling of speculative  bubbles and derive some quantitative properties of its dynamical evolution.
Starting from a description of individual speculative behaviours, we build and study a second order Markov process, which after simple transformations can be viewed as a  \textit{turning} two-dimensional Gaussian process.   
Then, our main problem is to obtain some bounds for the \textit{persistence rate} relative to the return time to a given price.
In our main results, we prove with both spectral and probabilistic methods that this rate is almost proportional to the turning frequency $\omega$ of the model and provide some explicit bounds.
In the continuity of this result, we build some estimators of $\omega$ and of the pseudo-period of the prices. At last, we end the paper by a proof of the quasi-stationary distribution of the process, as well as the existence of its persistence rate.
}
\bigskip

{\small
\textbf{Keywords}: Speculative bubble; Persistence rate; Gaussian Process; Diffusion Bridge; Statistics of processes.
 
\par
\vskip.3cm
\textbf{MSC2010:} Primary: 60J70, 35H10, 60G15, 35P15. 

}\par

\section{Introduction}

The evolution of prices in markets such as real estate is a popular subject of investigation.
The purpose of this paper is to propose a stochastic model which, at the same
time, is simple enough to be studied mathematically
and accounts for periodicity phenomena induced by speculation.
\smallskip

One commonly talks of financial bubble when, due to speculation of traders or owners, an asset price exceeds an asset fundamental value.
These owners then expect to resell the asset at an even higher price in
the future. There exist a lot of famous historical examples such as the Dutch Tulip Mania (1634-1637), the Mississippi bubble
(1718-1720) or the "Roaring '20s that preceded the 1929" crash. We refer to \cite{Garber} for a general remainder on historical bubbles.
Through some more recent events, one can observe that this phenomena is certainly actual. Think for instance to the Internet bubble which bursted in March 2000 after having led astronomical heights and lost more than 75\% of its value, and to the housing bubble encountered in the United States (2000-2010) 
 or in European countries (Spain, Ireland, France \ldots) (see $e.g.$ \cite{Shiller} or \cite{Friggit}).
 \smallskip 
 
A huge litterature exists on speculative bubbles and it seems nearly impossible to quote all the related numerous previous works.
We point out that our goal here is not to detail a general model flexible enough to take into account several complex economic realities.
However,  our approach is to propose a very simple tractable model from a mathematical point of view: the natural equilibrium price is assumed to be $0$ all along our temporal evolution, we do not consider any inflation nor credit crunch \cite{Fahri} and there is no regulating effect of any federal bank \cite{Bernanke}. At last, the processes introduced in the paper will be supposed time-homogeneous. 

According to Shiller \cite{Shiller} (see also \cite{kiselev}), the mechanism of creation of speculative bubbles is the following: 
\textit{"If asset prices start to rise strongly, the success of some investors
attracts public attention that fuels the spread of the enthusiasm for the market: 
(often, less sophisticated) investors enter the market and bid up prices. This "irrational
exuberance" heightens expectations of further price increases, as investors extrapolate recent price action far into the future. The markets meteoric rise is typically justified in
the popular culture by some superficially plausible "new era" theory that validates the
abandonment of traditional valuation metrics. But the bubble carries the seeds of its own
destruction; if prices begin to sag, pessimism can take hold, causing some investors to exit
the market. Downward price motion begets expectations of further downward motion, and
so on, until the bottom is eventually reached."}

In the previous citation, two phenomenas are exhibited: on the one hand, the investors have a tendency to follow the forecasting rule which consists in deciding that the price will increase if it has (strongly) increased in the past. On the other hand, the actions of the investors have certainly a self-reinforcing effect. 
In this paper, we assume  more or less that the dynamics of the market is dictated by these two phenomenas. However, we (necessarily) assume that there is also a general mean-reverting force and there exists randomness in the decisions of the investors.
Then, our model is obtained as the limit of the mean dynamics of all the investors   when the number of these investors tends to infinity (see next paragraph for more details).

Let us also precise that our setting corresponds to the so-called rational bubbles under symmetric information paradigm described in  \cite{blanchard, tirole1} for instance. In our framework, we are interested in the periodic pattern commonly encountered in such speculative markets, which is of primarily interest as pointed by \cite{Evans}. We establish that these periodic phenomena are related to a persistence problem which is also an important field of interest from an economic point of view \cite{blanchard,boucher2}. Note that our model is also simple enough to imagine statistical inference procedures for the estimation of several key parameters. Hence, even if our work comes from a probabilistic motivation, it also opens the way of statistical procedures to test bubble formation. This last statistical point is shortly discussed in the end of our paper and seems challenging for future works (some numerical results show that the standard likelihood estimation does not seem well suited to approach the unknown parameters in such a model).

At last, it is generally empirically observed that the bubbles bursts are fasten than bubbles formations. Our model can be generalized to more complex settings where such burst's and formation's timing could be different using a mixture of memory weights $\Gamma_{k,b}$  with $k \geq 2$ (see next paragraph for more details).

\subsection{Modeling of speculation}\label{mos}

Let us designate by $X\df(X_t)_{t\geq 0}$ the temporal evolution of the relative price of a commodity with respect to
another one. For instance it can be the difference between the price of the mean square meter of real estate in a particular town
and the price of the ounce of gold or the mean salary of a month of work.
Let the units be chosen
 so that, in the mean over a long time period, this relative price is zero.
We assume that three mechanisms are at work for the evolution of $X$:\\
- Economic reality plays the role of a restoring force, trying to draw $X$ back toward zero.
At least as a first approximation, it is natural to assume that this force is linear, whose rate will be denoted $a>0$.\\
- Speculation is reinforcing a tendency observed for some times in the past.
We make the hypothesis that the weight of past
 influences is decreasing exponentially fast in time, with rate $b>0$.
 The typical length the observation time window will be given by $b^{-1}$.\\
- Uncertainty is modeled by a Brownian motion of volatility $c>0$, which is a traditional assumption for randomness
coming from a lot of small unpredictable and independent perturbations, due to the functional central limit theorem.
\par
Putting together these three leverages, we end up with a law of evolution of $X$ described by the
stochastic differential equation
\bqn{eds}\label{eq:model}
\fo t\geq 0,\qquad dX_t&=&-aX_tdt+\lt(b \int_0^t\exp(b(s-t))\, dX_s\rt)dt +cdB_t\eqn
assuming for instance that initially, $X_0=0$.
Because of the presence of the Brownian motion $(B_t)_{t\geq 0}$ in the r.h.s., the trajectories of $X$
are not differentiable with respect to the time parameter $t\geq 0$.
But for the purpose of a heuristic interpretation, let us pretend they are, so we can consider $X'_t\df\frac{dX_t}{dt}$.
Assume furthermore that the ``origin" of time was chosen so that before it, $X$ was zero, namely, in the above economic interpretation,
the two commodities had their prices tied up at their relative equilibrium point before time 0.  This enables us to define $X'_t=X_t=0$ for any $t\leq 0$.
The middle term of the r.h.s.\ of (\ref{eds}) can then be rewritten as
\bqn{sigma}
\nonumber b \int_0^t\exp(b(s-t))\, dX_s&=&\int_0^{+\iy} X'_{t-s}\, b\exp(-bs)ds\\
&=&\EEE[X'_{t-\sigma}]\eqn
where $\sigma$ is distributed as an exponential variable of parameter $b$ and where $\EEE$ stands for the expectation
with respect to $\sigma$ (i.e.\ not with respect to the randomness underlying $X$, the corresponding expectation
will be denoted $\EE$).
Thus $X$ has a drift taking into account its past tendencies $X'$, but very old ones are almost forgotten, due to the exponential weight.
Indeed, since $\EEE[\sigma]=b$,  tendencies older than a time of order $b$ don't contribute much.
\par\me
The equation (\ref{eds}) can be seen as the limit evolution of the means of (relative) prices predicted by a large number $N\in\NN$ of speculative agents.
Assume that each agent $n\in\lin N\rin\df\{1, ..., N\}$ has his own idea of the evolution of the prices, designated by $X(n)\df (X_t(n))_{t\geq 0}$.
The mean process $\bar X\df (\bar X_t)_{t\geq 0}$ is defined by 
\bq
\fo t\geq 0,\qquad \bar X_t&\df& \frac1{N}\sum_{n\in\lin N\rin} X_t(n).\eq
For simplicity, we assume as above that all these processes were also defined for negative times and that
\bq
\fo t\leq 0,\, \fo n\in\lin N\rin, \qquad X_t(n)\ =\ \bar X_t\ =\ 0.\eq
At any time $t\geq 0$, each agent $n\in\lin N\rin$ has access to the whole past history $(\bar X_s)_{s\leq t}$ of the mean prices
(say, which is published a particular institute or website). But to handle this wealth of information, agent $n$
has chosen, once for all, a time window length $ \Upsilon(n)>0$
 and he computes the ratio $(\bar X_t-\bar X_{t-\Upsilon(n)})/\Upsilon(n)$
in order to decide what is the present tendency of the prices.
Then he interferes that this tendency contributes to
 the infinitesimal evolution of his estimate of prices $dX_t(n)$ via the term $(\bar X_t-\bar X_{t-\Upsilon(n)})/\Upsilon(n)\,dt$,
 speculating that what has increased (respectively decreased) will keep on increasing (resp.\ decreasing). Nevertheless, as everyone, he also undergoes the strength of the economic reality with rate $a>0$,
which adds a term $-a X_t(n)dt$ to his previsions. Furthermore, he cannot escape vagaries of life, good or bad, 
which disturb his evaluations with the infinitesimal increment $c\sqrt{N} dB_t(n)$, where $B(n)\df (B_t(n))_{t\geq 0}$ is a standard Brownian motion. The factor $\sqrt{N}$ may seem strange at first view,
but it accounts for the fact that the consequences of random events are amplified by a large population.
Alternately, it could be argued that $\sqrt{N}d B_t(n)$ decompose into $\sum_{m\in\lin N\rin} dB_t(n,m)$,
where  $(B_t(n,m))_{t\geq 0}$, for $n,m\in\lin N\rin$, are independent Brownian motions standing respectively for
the random perturbations induced by  $m$ on $n$ (including a self-influence $(B_t(n,n))_{t\geq 0}$).
It follows that 
\bq
\fo t\geq 0,\qquad dX_t(n)&=&-aX_t(n)dt+\frac{\bar X_t-\bar X_{t-\Upsilon(n)}}{\Upsilon(n)}dt +c\sqrt{N}dB_t(n)\eq
and we deduce that
\bq
\fo t\geq 0,\qquad d\bar X_t&=&-a\bar X_tdt+\lt(\frac1{N}\sum_{n\in\lin N\rin}\frac{\bar X_t-\bar X_{t-\Upsilon(n)}}{\Upsilon(n)}\rt)dt +c\frac1{\sqrt{N}}\sum_{n\in\lin N\rin}dB_t(n)\eq
Let us assume that all the $\Upsilon(n)$, for $n\in\lin N\rin$, and all the $B(m)$, for $m\in\lin N\rin$ are independent.
A first consequence is that the process $\bar B=(\bar B_t)_{t\geq 0}$ defined by
\bq
\fo t\geq 0,\qquad \bar B_t&\df& \frac1{\sqrt{N}}\sum_{n\in\lin N\rin}B_t(n)\eq
is a standard Brownian motion.
\\
Next, under the hypothesis that all the $\Upsilon(n)$, $n\in\lin N\rin$ have the same law as a random variable $\Upsilon$, we get by the law of 
large numbers, that almost surely,
\bq
\lim_{N\ri\iy}
\frac1{N}\sum_{n\in\lin N\rin}\frac{\bar X_t-\bar X_{t-\Upsilon(n)}}{\Upsilon(n)}&=&\EEE\lt[\frac{\bar X_t-\bar X_{t-\Upsilon}}{\Upsilon}
\rt]\eq
where $\EEE$ stands for the expectation with respect to $\Upsilon$ only.
Thus letting $N$ go to infinity, $\bar X$ ends up satisfying the same evolution equation as $X$, if the law of $\Upsilon$ is such that
\bqn{ipp}
\fo t\geq 0,\qquad
\EEE\lt[\frac{X_t- X_{t-\Upsilon}}{\Upsilon}
\rt]&=&b \int_0^t\exp(b(s-t))\, dX_s\eqn
almost surely with respect the trajectory $(X_s)_{s\in\RR}$.\\
Contrary to the first guess which could be made, $\Upsilon$ should not be distributed according to an exponential law of parameter $b$:
\begin{lem}
For any continuous semi-martingale $X=(X_t)_{t\in\RR}$ with $X_t=0$ for $t\leq 0$, (\ref{ipp}) is satisfied if 
$\Upsilon$ is distributed as a gamma law $\Gamma_{2,b}$ of shape 2 and scale $b$, namely if 
\bq
\fo t\geq 0, \qquad \PPP[\Upsilon\in dt]\ =\ \Gamma_{2,b}(dt)\ \df\ b^2t\exp(-bt)\, dt\eq
\end{lem}
\proof
By continuity of $X$, it is sufficient to check the almost sure equality of (\ref{ipp}) for any fixed $t\geq 0$.
Then denote $\wi X_s= X_t-X_{t-s}$, for $s\geq 0$, so that
\bq
\EEE\lt[\frac{X_t- X_{t-\Upsilon}}{\Upsilon}
\rt]&=&b^2\int_0^{+\iy} \frac{X_t- X_{t-s}}{s}s\exp(-bs)\, ds\\
&=&b^2\int_0^{+\iy}  \wi X_s \exp(-bs)\, ds\eq
The fact that $X$ is a semi-martingale enables to integrate by parts and we find
\bq
b^2\int_0^{+\iy}  \wi X_s \exp(-bs)\, ds&=&-b\Big[\wi X_s \exp(-bs)\Big]_0^{+\iy}+b\int_0^{+\iy} \exp(-bs) \, d\wi X_s\\
&=&b\int_{-\iy}^t \exp(-b(t-s)) \, d X_{s}\\
&=&b\int_0^{t} \exp(-b(t-s)) \, d X_{s}\eq
\wwtbp
\begin{rem}
Conversely, for the process $X$ defined by (\ref{eds}) and $X_t=0$ for $t\leq 0$,
the validity of (\ref{ipp}) implies (under an integrability assumption) that the law of $\Upsilon$ is the gamma distribution $\Gamma_{2,b}$. Actually,
denote by $G$ the distribution of $\Upsilon$ and assume that  $\int_0^{+\infty}s^{-1} G(ds)<+\infty$.
 Then, by the previous result, (\ref{ipp}) reads
\begin{equation}\label{equalgam}
\forall \,t\ge0, \quad \int_0^{+\infty} \frac {{X}_t-{X}_{t-s}}{s} G(ds)= \int_0^{+\infty} \frac {{X}_t-{X}_{t-s}}{s} \Gamma_{2,b}(ds)\quad a.s.
\end{equation}
Now, let $t>0$. Since the above equality holds almost surely, it follows  from  Girsanov Theorem (see $e.g.$ \cite{MR1725357}, Chapter 8), that we can replace $(X_s)_{s\in[0,t]}$ by $c$ times a Brownian motion (and next by linearity take $c=1$). Then, the main argument is the support Theorem (see $e.g.$ \cite{StroockVaradhan_support}), which yields in particular that for every positive $t$ and $\varepsilon$, for every ${\cal C}^1$-function $\varphi:(-\infty,t]\rightarrow\ER$ such that $\varphi(u)=0$ on $\ER_{-}$,   
\begin{equation}\label{equalgam2}
\PE(\sup_{s\in[0,t]}|X_s-\varphi(s)|\le\varepsilon)>0.
\end{equation}
Let $\varphi$ be such a function. By \eqref{equalgam} and \eqref{equalgam2}, we obtain that for every positive $\varepsilon$,
\bq\left|\int_0^{+\infty}\frac{\varphi(t)-\varphi({t-s})}{s} (G(ds)-\Gamma_{2,b}(ds))\right|&\le 2\varepsilon\int_0^{+\infty} \frac{1}{s} (G(ds)+\Gamma_{2,b})(ds)
\eq
and it follows that for every ${\cal C}^1$-function with $\varphi(u)=0$ on $\ER_{-}$, 
\begin{equation}\label{depjp}
\int_0^{+\infty}\frac{\varphi(t)-\varphi({t-s})}{s} G(ds)=\int_0^{+\infty}\frac{\varphi(t)-\varphi({t-s})}{s}\Gamma_{2,b}(ds),
\end{equation}
the result being available for all positive $t$.
Denoting $r=\varphi(t)$ and $h(s)=\varphi(t)-\varphi({t-s})$ for all $s\in [0,t]$, we get that for all $r\in\RR$ 
and all $\cC^1$ function $h\,:\, [0,t]\ri\RR$ with $h(0)=0$,
\bq
r\int_t^{+\iy} \frac 1s( G-\Gamma_{2,b})(ds)+\int_0^t h(s) \frac{G-\Gamma_{2,b}}{s}\,ds&=&0\eq
namely
\bq
\int_t^{+\iy} \frac 1s( G-\Gamma_{2,b})(ds)&=&0\eq
and $G$ and $\Gamma_{2,b}$ coincide on $(0,t]$.
Since this is true for all $t>0$, we get that $G$ and $\Gamma_{2,b}$ coincide on $(0,+\iy)$.
Because they are both probability measures, they cannot differ only on $\{0\}$, so
$G=\Gamma_{2,b}$.
\par
\noindent In fact this proof can be extended to any continuous semi-martingale whose martingale part 
is  non-degenerate.
\end{rem}
The law $\Gamma_{2,b}$ has the same rate $b$ of exponential decrease of the queues at infinity as the exponential distribution $\Gamma_{1,b}$
of $\sigma$ in (\ref{sigma}). The most notable difference between these two distributions is their behavior near zero: it is much 
less probable to sample a small values under $\Gamma_{2,b}$ than under $\Gamma_{1,b}$. Furthermore,  $\Gamma_{2,b}$
is a little more concentrated around its mean $2/b$ than $\Gamma_{1,b}$ around its mean $1/b$, their respective relative standard deviations
being 1/2 and 1.
These features are compatible with the previous modeling: the chance is small that an agent looks shortly in the past to get an idea
of the present tendency of $X$ and the dispersion of the lengths of the windows used by the agents may not be very important.
These behaviors would be amplified, if instead of $\Gamma_{2,b}$, we had chosen a gamma distribution $\Gamma_{k,b}$ of shape $k$ and scale $b$, with $k\in\NN\setminus\{1,2\}$, for the law of $\Upsilon$. The limit  evolution in this situation is dictated by the stochastic differential equation in $X^{[k]}$ given by
\bq
\fo t\geq 0,\qquad dX^{[k]}_t&=&-aX^{[k]}_tdt+\lt(b (k-2)!\int_0^tg_{b,k}(t-s)\, dX^{[k]}_s\rt)dt +cdB_t\eq
(starting again from $X_0^{[k]}=0$), where $g_{b,k}$ is the function defined by
\bq
g_{b,k}\st \RR_+\ni s&\mapsto& \exp(-bs)\sum_{l\in \lin 0, k-2\rin}\frac{(bs)^{l}}{l!}\eq
(curiously, the r.h.s.\ coincides with  the probability that a Poisson random variable of parameter $bs$ belongs to $ \lin 0, k-2\rin$).
\\
For $k=2$, we recover (\ref{eds}) and $X^{[2]}=X$.
The stochastic process $X$ is clearly not Markovian, but we will see in the sequel that it is a Markov process of order 2:
it is sufficient to add another real component to $X$ to get a Markov process.
It can be shown more generally  that $X^{[k]}$ is a Markov process of order $k$: $k-1$ real components must be added
to make it a Markov process. While this observation provides opportunities of better modelings, the investigation
of $X^{[k]}$ for $k>2$ (as well as the extension to non-integer values of $k$) is deferred to a future paper.
Here we will concentrate on the properties of $X$, but before presenting the results obtained,
let us give some simulations of $X$ in Figure \ref{fig:trajectories}.

\begin{figure}[h]
 \centering
\includegraphics[height=5cm]{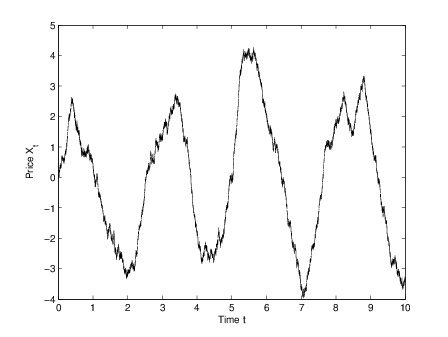}
\includegraphics[height=5cm]{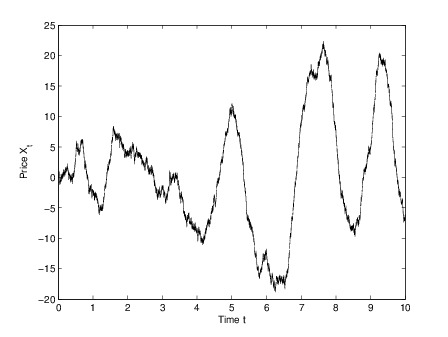}
\includegraphics[height=5cm]{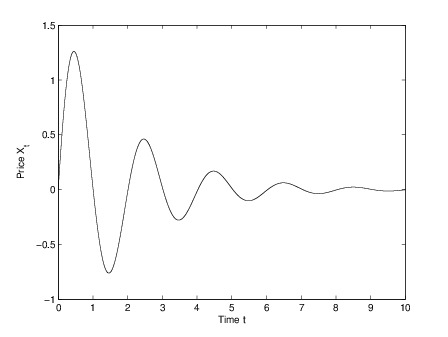}
\caption{\label{fig:trajectories} Several trajectories for various parameters (top left: $a=1, b=5, c=1$, top right: $a=1,b=10,c=5$, bottom: $a=1,b=10,c=0$).}
\end{figure}
A periodic structure appears, as that observed in practice in the forming of speculative bubbles. The process $X$ shows
some regularity in returning to its equilibrium position, trend which seems to be only slightly perturbed by the noise.
The variety of the trajectories is apparently less rich than that experienced by traditional Ornstein-Ulhenbeck processes,
suggesting  a concentration of the trajectory laws around some periodic patterns. Figure \ref{fig:return} shows the density of the return time of the process $(X)_{t \geq 0}$ to its equilibrium price $0$. These results have been obtained using a large number of Monte-Carlo simulations. One may remark in Figure \ref{fig:return} that the tail of the return time to equilibrium state is much smaller for our bubble process than the one of the O-U. process with the same invariant measure on the $X$ coordinate and with the same amount of injected randomness (namely through a standard Brownian motion). The purpose of this paper is to quantify these behaviors.

\begin{figure}[h]
 \centering
\includegraphics[height=5cm]{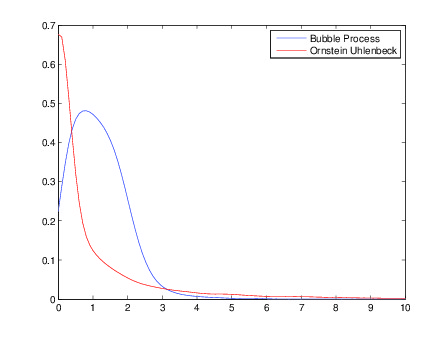}
\caption{\label{fig:return} Density function of the return time of the Ornstein-Uhlenbeck process, as well as the one of the bubble process. The O-U process is set to have the same invariant measure as the bubble process and  the same amount of injected randomness.}
\end{figure}

\subsection{Results}\label{r}

As already mentioned, the process $X$ whose evolution is driven by (\ref{eds}) is not Markovian.
Nevertheless, it is not so far away from being Markovian: consider the process $Y\df(Y_t)_{t\geq 0}$ defined by
\bq
\fo t\geq 0,\qquad Y_t&\df&b \int_0^t\exp(b(s-t))\, dX_s-b X_t\eq
The process $Z\df (Z_t)_{t\geq 0}\df((X_t,Y_t)^*)_{t\geq 0}$ (where ${}^*$  stands for the transpose operation) is then Markovian and its evolution
is dictated by the simple  2-dimensional stochastic differential equation
\bqn{eds2}
\fo t\geq 0,\qquad dZ_t&=&AZ_t\,dt+C\,dB_t\eqn
starting from $Z_0=0$ and where 
\bqn{AC}
A\ \df\ \lt( \begin{array}{cc}
b-a &1\\
-b^2& -b\end{array}\rt)
&\hbox{ and }&
C\ \df\  \lt( \begin{array}{c}
c\\
0\end{array}\rt)
\eqn
The linearity of (\ref{eds2}) and the fact that the initial condition is deterministic imply that at any time $t\geq 0$
the distribution of $Z_t$ is Gaussian. As it will be checked in next section, this distribution converges  for large time $t\geq 0$
toward $\mu$, a normal distribution of mean 0 and whose variance matrix $\Sigma$ is positive definite.
Since the Markov process $Z$ is  Feller, $\mu$ is an invariant probability measure for $Z$. It is in fact the only one,
because the generator $L$ associated to the evolution equation (\ref{eds2}) and given by
\bqn{L}
L&\df& ((b-a)x+y)\pa_x -(b^2x+by)\pa_y +\frac{c}{2}\pa_x^2\eqn
is hypoelliptic (also implying that $\Sigma$ is positive definite). \\
The study of the convergence to equilibrium of $Z$
begins with the spectral resolution of $A$. Three situations occur:
\\
$\bullet$ If $a>4b$, $A$ admits two real eigenvalues, $\lambda_{\pm}\df (-a\pm\sqrt{a^2-4ab})/2$.\\
$\bullet$ If $a=4b$, $A$ is similar to the  $2\times 2$ Jordan matrix associated to the eigenvalue $-a/2$.\\
$\bullet$ If $a<4b$, $A$ admits two conjugate complex eigenvalues, $\lambda_{\pm}\df (-a\pm i\sqrt{4ab-a^2})/2$.
\\
But in all cases, let $l<0$ be the largest  real part of the  eigenvalues, namely
\bqn{l}
l&\df&
\frac{-a+\sqrt{(a^2-4ab)_+}}{2}\eqn
This quantity is the exponential rate of convergence of $\mu_t$, the law of $Z_t$, toward $\mu$,
in the $\LL^2$ sense: for $t>0$, measure the discrepancy between $\mu_t$ and $\mu$ through
\bqn{J}
J(\mu_t,\mu)&\df& \sqrt{\int \lt( \frac{d \mu_t}{d\mu}-1\rt)^2\, d\mu}\eqn
Since $X$ was our primary object of interest, let us also denote by $\nu$ and $\nu_t$ the 
first marginal distributions of $\mu$ and $\mu_t$ respectively.
\begin{pro}\label{Relambda}
We have
\bq
\lim_{t\ri+\iy} \frac{1}{t}\ln(J(\mu_t,\mu))\ =\ 2l\ =\ \lim_{t\ri+\iy} \frac{1}{t}\ln(J(\nu_t,\nu))\eq
\end{pro}
These convergences can be extended to other measures of discrepancy, such as the square root of the relative entropy,
 or to initial distributions $\mu_0$ of $Z_0$ more general than the Dirac measure at 0, at least under the assumption that
 $J(\mu_0,\mu)<+\iy$.
 Thus if we look at $X_{r/\lve l\rve}$ for large $r>0$, it has almost forgotten that it started from 0 and its law is close
 to the Gaussian distribution $\nu$, up to an error $\exp(-(1+\circ(1))r)$.
 \par
Nevertheless, the periodicity features we are looking for appear only for $a<4b$, as it can be guessed from 
the existence of non-real eigenvalues, which suggests $2\pi/\omega$ as period, where
\bqn{w}
\omega&\df& \sqrt{ab-\frac{a^2}{4}}\eqn
In the regime where $b\gg a$, we have $\omega \gg 2\lve l\rve$: a lot of periods has to alternate before stationarity is approached.
This phenomenon is often encountered in the study  of ergodic Markov processes which are far from being reversible, e.g.\ a diffusion on a circle with a strong constant drift  (for instance turning clockwise).
\par
In order to quantify this behavior, we are interested in the return time $\tau$ to zero for $X$, which is of primary interest in the economic 
interpretation given at the beginning of the introduction (on the contrary to the relaxation time to equilibrium,
which seems very far away in the future):
\bqn{tau}
\tau&\df& \inf \{t\geq 0\st X_t=0\}\eqn
Of course it is no longer relevant to assume that $Z_0=0$ and instead we assume that $(X_0,Y_0)=(x_0,y_0)\in \RR_+^*\times \RR$.
In practice, $\tau$ appears through a temporal shift: we are at time $s>0$ which is such that $X_s>0$ and we are wondering
when in the future $X$ will return to its equilibrium position 0. 
Up to the knowledge of $(X_s,Y_s)$, the time left before this return has the same law as $\tau$ if $(x_0,y_0)$ is initialized with
the value $(X_s,Y_s)$. 
\\
The next result shows that up to universal factors, the exponential rate of concentration of $\tau$ is given by $1/\omega$, confirming that
when $b\gg a$, the return to zero happens much before the process reaches equilibrium.
\begin{theo}\label{T1}
For any $0<a<4b$ , $c>0$, $x_0>0$ and $y_0\in\RR$, we have 
\bq
 \ \PP_{(x_0,y_0)}[\tau>t]\ \leq \ 2\exp\lt(-\frac{\ln(2)}{\pi}{\omega}t\rt).
\eq
Furthermore, if {$(1+\frac{1}{\sqrt{2}})a\le b$}, there exists a quantity $\epsilon(x_0,y_0) >0$ (which in addition to $x_0$ and $y_0$,  depends on the parameters $a,b,c$)
such that
\bq
\fo t\geq 0,\qquad
\PP_{(x_0,y_0)}[\tau>t]
&\ge &\epsilon(x_0,y_0)\exp\lt(-4{\omega}t\rt).\
\eq
More generally, to any initial distribution $m_0$ on $D\df\{(x,y)\in\RR^2\st x>0\}$, we can associate a quantity $\epsilon(m_0)$
such that
\bq
\fo t\geq 0,\qquad
\PP_{m_0}[\tau>t]
&\ge &\epsilon(m_0)\exp\lt(-4{\omega}t\rt).\
\eq
\end{theo}
Surprisingly, we found the lower bound more difficult to obtain than the upper bound, while
in reversible situations it is often the opposite which is experienced. \par
\begin{rem}\label{persistence}
In the first appendix it will be shown that  a quasi-stationary probability $\nu^D$ and a corresponding rate $\lambda_0(D)>0$ can be associated
to $D$: the support of $\nu^D$ is the closure of $D$ and 
under $\PP_{\nu^D}$, $\tau$ is distributed as an exponential random variable of parameter $\lambda_0(D)$:
\bqn{persistence1}
\fo t\geq 0,\qquad 
\PP_{\nu^D}[\tau\geq t]&=&\exp(-\lambda_0(D)t)\eqn
In the sequel, the quantity $\lambda_0(D)$ will be called the persistence rate of $D$.
It can be seen as  the smallest eigenvalue (in modulus) of the underlying Markov generator 
with a Dirichlet condition on the boundary of the domain $D$, when it is interpreted as acting on $\LL^2(\mu^D)$,
where $\mu^D$ is the restriction of $\mu$ on $D$.
 The above theorem then provides
lower and upper bounds on $\lambda_0(D)$, essentially proportional  to $\omega$: at least for $0<(1+\frac{1}{\sqrt{2}})a\le b$,
\bqn{persistence2}
\frac{\ln(2)}{\pi}\omega\ \leq \ \lambda_0(D)\ \leq \ 4\omega\eqn
According to figure \ref{fig:return}, starting from other initial distributions on $D$, the law of $\tau$ will no longer be exponential,
nevertheless
we believe that for any $(x_0,y_0)\in D$, the following limit takes place
\bq
\lim_{t\ri+\iy}\frac1t\ln(\PP_{x_0,y_0}[\tau>t])&=&-\lambda_0(D)\eq
The difficulty in obtaining this convergence stems from the non-reversibility of the process under consideration. 
In the literature, it is the reversible and elliptic situations which are the most thoroughly investigated.
For a general reference on quasi-stationarity, see e.g.\ the book \cite{MR2986807} of Collet, Mart{\'{\i}}nez and San
              Mart{\'{\i}}n, as well as the bibliography therein.
              \end{rem}
\par\sm
The previous result provides a good picture for large values of $\tau$, but is there a precursor sign that $\tau$ will be much shorter than expected? Indeed we cannot miss it, because in this situation of a precocious return to zero, the system has a strong tendency  to first explode!
To give a rigorous meaning of this statement, we need to introduce the bridges associated to $Z$.
For $z,z'\in\RR^2$ and $T>0$, denote by $\PP^{(T)}_{z,z'}$ the law of the process $Z$ evolving according to (\ref{eds2}),
conditioned by the event $\{Z_0=z,\ Z_T=z'\}$. Note that there is no difficulty to condition by this negligible set, because
the process $Z$ starting from $z$ is Gaussian and the law of $Z_T$ is non-degenerate. \\

For fixed $z,z'\in\RR^2$ and $T>0$ small, we are interested in the behavior of
$\xi^{(T)}\df(\xi_t^{(T)})_{t\in[0,1]}$,  the process defined by
\bq
\fo t\in[0,1],\qquad 
\xi_t^{(T)}&\df& TZ_{Tt}\eq
Let us define the trajectory $\varphi_{z,z'}\st [0,1]\ri \RR^2$ by
\bqn{limit_bridge}
\fo t\in[0,1],\qquad \varphi_{z,z'}(t)&\df&
\lt( \begin{array}{c}
\frac{6 \omega}{b^2} t (1-t) (y-y')\\
0\end{array}\rt)
 \eqn
 where $z=(x,y)$ and $z'=(x',y')$.
\par
\begin{theo}\label{T2}
For fixed $z,z'\in\RR^2$, as $T$ goes to $0_+$, $\xi^{(T)}$ converges in probability (under  $\PP^{(T)}_{z,z'}$) toward the deterministic trajectory
$\varphi_{z,z'}$, with respect to the uniform norm on $\cC([0,1],\RR^2)$.
\end{theo}
In particular, if $z,z'\in\RR^2$ are such that $\Re(z)>0$, $\Re(z')\leq 0$ and $\Im(z)\not=\Im(z')$, the bridge $(Z_t)_{t\in[0,T]}$
relying $z$ to $z'$ for small $T>0$ explodes as $1/T$.
From the definition of $\varphi_{z,z'}$ given in \eqref{limit_bridge}, we can see that the explosion is in the $x$-direction, toward $+\iy$ or $-\iy$, depending on the sign of $\Im(z)-\Im(z')$,
as it is illustrated by the pictures of Figure \ref{fig:bridge}.

\begin{figure}[h]
 \centering
\includegraphics[height=4cm]{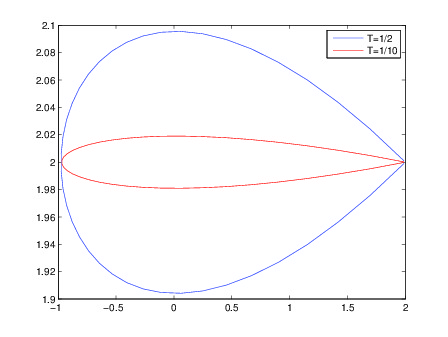}
\includegraphics[height=4cm]{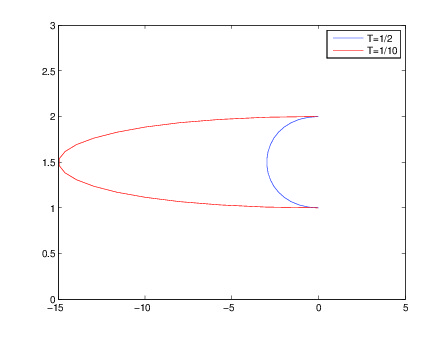}
\includegraphics[height=4cm]{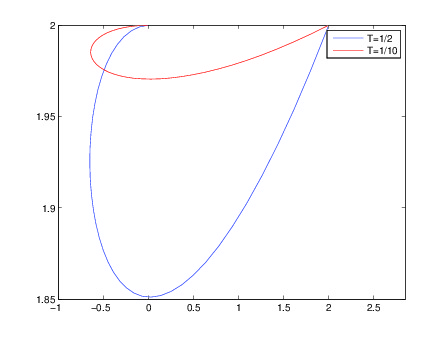}
\caption{\label{fig:bridge}  Expected trajectories of the hypo-elliptic bridge from $z$ to $z'$ within time $1/2$ and within small time $1/10$.
Left: Non explosion when $z'=z$. Middle: explosion when $\Re(z)= \Re(z')$ and $\Im(z)\not=\Im(z')$. Right: Non explosion when $\Im(z)=\Im(z')$ and $\Re(z)\not= \Re(z')$.}
\end{figure}

\begin{rem} Note that the sharp behavior of the bridge when $T\rightarrow 0$ leads in Section \ref{Eobast} to a probabilistic proof of a lower-bound for $\PE(\tau>t)$ (see Proposition \ref{lowerboundprob}). The interest of this alternative proof is  that the approach is maybe more intuitive. However, we have not been able to provide some explicit constants following this method.
\end{rem}

\par\me
The paper is constructed on the following plan. In next section we present the preliminaries on $Z$, especially its Gaussian
features which enable to obtain Proposition \ref{Relambda}. We will also see how to parametrize the process $Z$ under a simpler form.
The exit time  defined in (\ref{tau}) is investigated in Section \ref{Dee}, where Theorem \ref{T1} is obtained.
Section \ref{Eobast} is devoted to the study of bridges and to the proof of Theorem \ref{T2}. At last, we shortly discuss in Section \ref{sec:stat} on a numerical and statistical estimation problems related to the estimation of $\omega$.

\section{Preliminaries and simplifications}\label{pas}
This section contains some basic results about the Ornstein-Uhlenbeck diffusion $Z$  described by (\ref{eds2}) and whose coefficients
are given by (\ref{AC}).
\subsection{Gaussian computations}\label{Gc}
Our main goal here is to prove Proposition \ref{Relambda}. 
\par\sm
We begin by checking that the process $Z$ is Gaussian. 
Indeed, considering the process $\wi Z$
defined by
\bq
\fo t\geq 0,\qquad
\wi Z_t &\df& \exp(-At)Z_t\eq
we get that
\bq
\fo t\geq 0,\qquad d\wi Z_t\ =\ \exp(-At)(-AZ_t \,dt+dZ_t)\ =\ \exp(-At)C\, dB_t\eq
It follows that
\bqn{Zt}
\fo t\geq 0,\qquad Z_t&=& \exp(A t)Z_0+\int_0^t\exp(A(t-s)) C\, dB_s\\
\nonumber &=&\int_0^t\exp(A(t-s)) C\, dB_s
\eqn
since we assumed that $Z_0=0$.
It appears on this expression that for any $t\geq 0$, the law of $Z_t$ is a Gaussian distribution
of mean 0 and variance matrix $\Sigma_t$ given by
\bqn{Sigmat}
\nonumber\Sigma_t&\df&\int_0^t\exp(A(t-s)) C C^* \exp(A^*(t-s))\, ds\\
&=&
\int_0^t\exp(As) C C^* \exp(A^*s)\, ds\eqn
For $a,b>0$, the eigenvalues of $A$ have negative real parts, so that the above rhs\ converges
as $t$ goes to infinity toward a symmetric positive definite matrix $\Sigma$. As announced in the introduction,
the Gaussian distribution $\mu$ of mean 0 and variance $\Sigma$ is then an invariant measure for the
evolution (\ref{eds2}). It  is a consequence of the fact that the underlying semi-group is Fellerian 
(i.e.\ it preserves the space of bounded continuous functions), as it can be seen from (\ref{Zt}),
where $Z_t$ depends continuously on $Z_0$, for any fixed $t\geq 0$.
Note furthermore that the above computations show  that for any initial law of $Z_0$,
 the law of $Z_t$ converges toward $\mu$ for large $t$, because $ \exp(A t)Z_0$
 converges almost surely toward 0. It follows that $\mu$ is the unique invariant measure associated to
 (\ref{eds2}).
 To obtain more explicit expressions for the above variances, we need the spectral decomposition of $A$.
 The characteristic polynomial of $A$ being $X^2+aX+ab$, we immediately obtain the results presented in the beginning of Subsection \ref{r}
about the eigenvalues of $A$. Let us treat in detail the case $a<4b$, which is the most interesting for us: there are
two conjugate eigenvalues, $\lambda_{\pm}=l\pm \omega i$, where $l=-a/2$ (see (\ref{l}))  and $\omega$ is defined in
 (\ref{w}).
 \begin{lem}\label{angles}
 If $a<4b$, there exist two angles $\alpha \in (\pi/2,3\pi/2)$ and $\beta\in [0,2\pi)$ such that 
 for any $t\geq 0$,
 \bq
 \Sigma_t&=& R_0-\exp(-at) R_t\eq
 where
 \bq R_t&\df& \frac{c^2}{4ab-a^2}
\lt( \begin{array}{cc}
 \frac{b^2}{a}\lt(2+\sqrt{\frac{a}{b}}\cos(2\beta- \alpha-2\omega t)\rt)&
 \frac{b}{a}\lt(2\cos(\beta)+\cos(\beta- \alpha-2\omega t)\rt)\\
 \frac{b}{a}\lt(2\cos(\beta)+\cos(\beta- \alpha-2\omega t)\rt)& 
 \frac{b^4}{a}\lt(2+\sqrt{\frac{a}{b}}\cos(- \alpha-2\omega t)\rt) \end{array}\rt)
\eq
Passing to the limit as $t\ri+\iy$, we get
$\Sigma=R_0$ and we deduce more precisely that
\bq
\Sigma &=& \frac{c^2}{2a^2} \left( 
 \begin{matrix}
 {a+b}  & - {b^2} \\ - {b^2}  & {b^3}
 \end{matrix}
\right).
\eq
\end{lem}
\proof
From the first line of the matrix $A$, we deduce that an eigenvector associated to $\lambda_\pm$ is $(1,\lambda_\pm +a-b)^*$.
So writing
\bq 
\tr\ \df\ \lt( \begin{array}{cc}
\lambda_- &0\\
0& \lambda_+ \end{array}\rt)
&\hbox{ and }&
M\ \df\ \lt( \begin{array}{cc}
1&1\\
\lambda_- +a-b&\lambda_+ +a-b \end{array}\rt)
\eq
we have that $A=M\tr M^{-1}$, where
\bq
M^{-1}\ =\ \frac1{\lambda_+-\lambda_-}\lt( \begin{array}{cc}
\lambda_+ +a-b&-1\\
-\lambda_- +b-a &1 \end{array}\rt)
\eq
In view of (\ref{Sigmat}), we need to compute for any $s\geq 0$,
\bq
\exp(As) C C^* \exp(A^*s)&=&M\exp(s\tr)M^{-1}CC^*(M^*)^{-1}\exp(s \tr^*)M^*\eq
where ${}^*$ is now the conjugate transpose operation.
A direct computation leads to 
\bq
\frac{\lambda_+-\lambda_-}{c}M\exp(s\tr)M^{-1}C\ =\ \lt(\begin{array}{c}
\sigma_1\\
\sigma_2\end{array}\rt)\ \df\ \lt(\begin{array}{c}
z\exp(s\lambda_-)-\bar{z}\exp(s\lambda_+) \\
\lve z\rve^2(\exp(s\lambda_-)-\exp(s\lambda_+))\end{array}\rt)
\eq
where $z\df \lambda_++a-b=a/2-b+i\omega$ and $\overline{z}:=a-b+\lambda_-=a/2-b-i\omega$.
So we get that 
\bq
\exp(As) C C^* \exp(A^*s)
&=&\frac{c^2}{\lve \lambda_+-\lambda_-\rve^2}
\lt( \begin{array}{cc}
\lve \sigma_1\rve^2&\sigma_1\overline{\sigma_2}\\
\sigma_2 \overline{\sigma_1}&\lve \sigma_2\rve^2 \end{array}\rt)
\\
&=&
\frac{c^2}{4\omega^2}
\lt( \begin{array}{cc}
2\lve z\rve ^2e^{2ls}-2\Re(z^2e^{2\lambda_-s})&2\Re(z)e^{2ls}-2\Re(ze^{2\lambda_-s})\\
2\Re(z)e^{2ls}-2\Re(ze^{2\lambda_-s})&2\lve z\rve ^4(e^{2ls}-\Re(e^{2\lambda_-s}) )\end{array}\rt)
\eq
Integrating this expression with respect to $s$, we obtain, first for any $t\geq 0$,
\bq \Sigma_t&=& 
\frac{c^2}{4\omega^2}
\lt( \begin{array}{cc}
\lve z\rve ^2\frac{e^{2lt}-1}{l}-\Re(z^2\frac{e^{2\lambda_-t}-1}{\lambda_-})&\Re(z)\frac{e^{2lt}-1}{l}-\Re(z\frac{e^{2\lambda_-t}-1}{\lambda_-})\\
\Re(z)\frac{e^{2lt}-1}{l}-\Re(z\frac{e^{2\lambda_-t}-1}{\lambda_-})&\lve z\rve ^4(\frac{e^{2lt}-1}{l}-\Re(\frac{e^{2\lambda_-t}-1}{\lambda_-}) )\end{array}\rt)
\eq
and next, recalling that $\Re(\lambda_-)=\Re(\lambda_+)=l<0$,
\bq \Sigma := \lim_{t \rightarrow + \infty} \Sigma_t&=& 
\frac{c^2}{4\omega^2}
\lt( \begin{array}{cc}
-\lve z\rve ^2\frac{1}{l}+\Re(z^2\frac{1}{\lambda_-})&-\Re(z)\frac{1}{l}+\Re(z\frac{1}{\lambda_-})\\
-\Re(z)\frac{1}{l}+\Re(z\frac{1}{\lambda_-})&-\lve z\rve ^4(\frac{1}{l}-\Re(\frac{1}{\lambda_-}) )\end{array}\rt)
\eq
Thus it appears that
\bq
\fo t\geq 0,\qquad \Sigma_t&=&\Sigma-e^{-at}R_t\eq
where the last term is the sinusoidal matrix   defined by
\bq 
R_t&\df& 
\frac{c^2}{4\omega^2}
\lt( \begin{array}{cc}
-\lve z\rve ^2\frac{1}{l}+\Re(z^2\frac{e^{-2\omega i t}}{\lambda_-})&-\Re(z)\frac{1}{l}+\Re(z\frac{e^{-2\omega i t}}{\lambda_-})\\
-\Re(z)\frac{1}{l}+\Re(z\frac{e^{-2\omega i t}}{\lambda_-})&-\lve z\rve ^4(\frac{1}{l}-\Re(\frac{e^{-2\omega i t}}{\lambda_-}) )\end{array}\rt)
\eq
Note that $\Sigma=R_0$ (this can also be deduced from $\Sigma_0=0$). To recover the matrices given in the statement of the lemma,
we remark that $\lve \lambda_-\rve^2 =ab$ and $\lve z\rve^2=b^2$, so there exist angles $\alpha,\beta\in [0,2\pi)$ such that
\bq
\lambda_-\ =\ \sqrt{ab}\exp(i\alpha)&\hbox{ and }& z\ =\ b\exp(i\beta)\eq
Since $\Re(\lambda_-)<0$, we have $\alpha \in (\pi/2, 3\pi/2)$ and the first announced results follow at once.
Concerning the more explicit computation of $\Sigma$, just take into account that
\bq
\cos(\alpha)\ =\ -\frac{\sqrt{a}}{2\sqrt{b}}&\hbox{ and }&\sin(\alpha)\ =\ -\frac{\omega}{\sqrt{ab}}\\
\cos(\beta)\ =\ \frac{a-2b}{2b}&\hbox{ and }&\sin(\beta)\ =\ \frac{\omega}{b}\eq
and expand the matrix
\bq
R_0&=&
 \frac{c^2b}{(4b-a)a^2}
\lt( \begin{array}{cc}
 b\lt(2+\sqrt{\frac{a}{b}}\cos(2\beta- \alpha)\rt)&
\lt(2\cos(\beta)+\cos(\beta- \alpha)\rt)\\
\lt(2\cos(\beta)+\cos(\beta- \alpha)\rt)& 
 {b^3}\lt(2+\sqrt{\frac{a}{b}}\cos( \alpha)\rt) \end{array}\rt)
\eq
\wwtbp
\begin{rem}\label{R0}
In particular, it appears that for large time $t\geq 0$, $X_t$ converges in law toward 
the centered Gaussian distribution $\nu$ of variance $c^2(b+a)/(2a^2)$.
\end{rem}
\par
We will need another basic ingredient,  valid in any dimension, about the functional $J$
defined in (\ref{J}).
\begin{lem}\label{muwimu}
Let $\mu$ and $\wi\mu$ be two Gaussian distributions in $\RR^d$, $d\geq 1$, of mean 0 and respective variance matrices 
$\Sigma$ and $\wi\Sigma$, assumed to be positive definite.
If $\wi\Sigma^{-1}-\Sigma^{-1}/2$ is  positive definite, we have
\bq
J(\wi\mu,\mu)&=& \sqrt{\frac{1}{\sqrt{\det(\Id-S^2)}}-1}\eq
where $S\df \Sigma^{-1}\wi\Sigma-\Id$ and $J(\wi\mu,\mu)=+\iy$ otherwise.
\end{lem}
\proof
From the above assumptions, we have
\bq
\fo x\in\RR^d,\qquad \frac{d\wi\mu}{d\mu}(x)&=&\sqrt{\frac{\det (\Sigma)}{\det (\wi\Sigma)}}\exp\lt(-x^*\frac{(\wi\Sigma^{-1}-\Sigma^{-1})}{2}x\rt)\eq
Thus the function $ \frac{d\wi\mu}{d\mu}$ belongs to $\LL^2(\mu)$ (property itself equivalent to the finiteness of $J(\wi\mu,\mu)$), if and only if the symmetric matrix
$\wi\Sigma^{-1}-\Sigma^{-1}+\Sigma^{-1}/2$ is positive definite. In this case, we have
\bq
\int \lt( \frac{d\wi\mu}{d\mu}\rt)^2\,d\mu&=&\frac{\det (\Sigma)}{\det (\wi\Sigma)}\sqrt{\frac{\det((2\wi\Sigma^{-1}-\Sigma^{-1})^{-1})}{\det(\Sigma)}}\\
&=&\frac{\sqrt{\det (\Sigma)}}{\det (\wi\Sigma)}\sqrt{\det(\wi\Sigma)\det((2\Id-\Sigma^{-1}\wi\Sigma)^{-1})}\\
&=&\sqrt{\frac{\det (\Sigma)}{\det (\wi\Sigma)}}\sqrt{\det((\Id-S)^{-1})}\\
&=&\frac{1}{\sqrt{\det(\Id+S)\det(\Id-S)}}
\eq
where we used that $\wi\Sigma=\Sigma( \Id+S)$. It remains to note that 
\bq
J^2(\wi\mu,\mu)\ =\ \int \lt[ \lt( \frac{d\wi\mu}{d\mu}\rt)^2-2 \frac{d\wi\mu}{d\mu}+1 \rt]\,d\mu\ =\ 
 \int \lt( \frac{d\wi\mu}{d\mu}\rt)^2\,d\mu-1
\eq\wwtbp
\par Still in the case $4b>a$, we can now proceed to the
\prooff{Proof of Proposition \ref{Relambda}}
In view of Lemma \ref{angles}, we want to apply Lemma \ref{muwimu}
with $\Sigma=R_0$ and $\wi\Sigma= R_0-\exp(-at)R_t$, for $t\geq 0$.
This amounts to take $S\df -\exp(-at)R_0^{-1}R_t$, matrix converging to zero exponentially fast as $t$ goes to $+\iy$.
It follows that for $t$ large enough, $\wi\Sigma^{-1}-\Sigma^{-1}/2$ is  positive definite
and we get 
\bq
J(\mu_t,\mu)&=& \sqrt{\frac{1}{\sqrt{\det(\Id-S^2)}}-1}\eq
Taking into account that the matrices $R_t$ are bounded uniformly over $t\in\RR_+$, an expansion for large $t$ gives
\bq
\frac{1}{\sqrt{\det(\Id-S^2)}}&=&
\frac{1}{\sqrt{1-\trace(S^2)+\cO(\lVe S^2\rVe_{\mathrm{HS}}^2)}}
\\&=&
1+\frac12\trace(R_0^{-1}R_tR_0^{-1}R_t)\exp(-2at)+\cO(\exp(-4at))\eq
(where $\lVe\cdot \rVe_{\mathrm{HS}}$ stands for the Hilbert-Schmidt norm, i.e.\  the square root of the sum of the squares of the
entries of the matrix).
We will be able to conclude to
\bqn{a2}
\lim_{t\ri+\iy} \frac{1}{t}\ln(J(\mu_t,\mu))&=&-a\eqn
if we can show that
\bqn{liminf0}
\liminf_{t\ri+\iy} \trace(R_0^{-1}R_tR_0^{-1}R_t)&>&0\eqn
(since it is clear that $\limsup_{t\ri+\iy} \trace(R_0^{-1}R_tR_0^{-1}R_t)<+\iy$).
Taking advantage of the fact that $R_0$ is a symmetric and positive definite matrix, we consider for $t\geq 0$,
$\wit R_t\df R_0^{-1/2}R_tR_0^{-1/2}$, which is also a symmetric matrix.
Since $\trace(R_0^{-1}R_tR_0^{-1}R_t)=\trace(\wit R_t^2)=\lVe\wit R_t\rVe_{\mathrm{HS}}^2$,
this quantity  is nonnegative and can only vanish if $\wit R_t$,
or equivalently $R_t$, is the null matrix. This never happens, because the first entry of $R_t$, namely
$ (cb/a)^2
(2+\sqrt{{a}/{b}}\cos(2\beta- \alpha-2\omega t))/(4b-a)$, is positive.
The continuity and the periodicity of the mapping $\RR_+\ni t\mapsto R_t$ enables to check the validity  of (\ref{liminf0}) and next of (\ref{a2}).\par\sm
The corresponding result for the first marginal  $\nu_t$ (the law of $X_t$) is obtained in the same way.
Indeed, from Lemma \ref{angles}, for any $t\geq 0$, $\nu_t$ is the real Gaussian law of mean 0 and variance
$r_0-\exp(-at)r_t$, where
\bq
\fo t\geq 0, \qquad r_t&\df& 
\frac{(cb)^2}{4a^2b-a^3}
\lt(2+\sqrt{\frac{a}{b}}\cos(2\beta- \alpha-2\omega t)\rt)\eq
with the angles $\alpha \in (\pi/2,3\pi/2)$ and $\beta\in [0,2\pi)$ described in the proof of Lemma \ref{angles}.
In particular $\nu\df\lim_{t\ri+\iy}\nu_t$ is  the real Gaussian law of mean 0 and variance
$r_0$. Lemma \ref{muwimu} applied with $d=1$ leads at once to 
\bq
\lim_{t\ri+\iy} \frac{1}{t}\ln(J(\nu_t,\nu))&=&-a\eq
\wwtbp
\par
The remaining situations $a=4b$ and $a>4b$ can be treated in the same way. In view of the previous arguments,
it is sufficient to check that it is possible to write
\bq \fo t\geq 0,\qquad \Sigma_t&=&\Sigma-\exp(-2 lt)R_t\eq
where $l$ is defined in (\ref{l}) and where the family $(R_t)_{t\geq 0}$
is such that \bq
\lim_{t\ri +\iy} \frac{1}{t}\ln(\lVe R_t\rVe
)&=&0\eq
 for any chosen norm $\lVe\cdot \rVe$ on the space of $2\times 2$ real matrices, due to their mutual equivalence.
 The obtention of the family $(R_t)_{t\geq 0}$ also relies on the spectral decomposition of $A$, with $R_t$
 converging for large times $t$ if $a>4b$ and exploding like $t^2$ if $a=4b$.\par\me
\subsection{Simplifications with the view to Theorem \ref{T1}}\label{sitvoT}
Let us begin this subsection by emphasizing two important properties of  Theorem \ref{T1}:\\
\noindent $\bullet$ The result does not depend on the variance coefficient $c$.\\
\noindent $\bullet$ The exponent is proportional to $\omega=\sqrt{ab-a^2/4}$ which denotes the mean angular speed of the deterministic
system $\dot{z}=Az$.\\
These properties can be understood through some linear and scaling transformations of the process $(Z_t)_{t\ge0}$. 
More precisely, these transformations will be used in the sequel to reduce the problem to the study of a process with 
mean constant angular speed and a  normalized diffusion component.
\smallskip


We choose to first give the idea in  a general case and then, apply it to our model.\\

Let $A\in{GL}_{2}(\ER)$ with complex eigenvalues given by $\lambda_{\pm}={-\rho}{\pm}i{\omega}$ where $\rho\in\mathbb{\ER}$
and $\omega\in\ER_+^*$. Let us consider the two-dimensional Gaussian differential system given by
\begin{equation}\label{eq:gene-syst}
d \zeta_t=A  \zeta_t dt+\Sigma dB_t
\end{equation}
where $\Sigma\in\mathbb{M}_2(\ER)$ and $(B_t)_{t\ge0}$ is a standard two dimensional Brownian motion.
For such a process, the precise transformation is given in Proposition \ref{genetransf}. This proposition is based on the following lemma. 

\noindent 
\begin{lem} Let $A\in{GL}_{2}(\ER)$ with complex eigenvalues given by $\lambda_{\pm}={-\rho}{\pm}i{\omega}$ where $\rho\in\mathbb{\ER}$
and $\omega\in\ER_+^*$. There exists $P\in{GL}_{2}(\ER)$ such that 
$$A=P(-\rho I_2+\omega J_2)P^{-1}$$
where 
\begin{equation}
I_2:=\begin{pmatrix}1&0\\
0&1\end{pmatrix} \qquad \text{and} \qquad 
J_2:=\begin{pmatrix}0&-1\\
1&0\end{pmatrix}.
\end{equation}
Furthermore, for every $v\in\ER^2\backslash\{0\}$, $P:=P_v$ given by $P_v=(v,\frac{D+\rho I_2}{\omega} v)$ is an admissible choice. 
\end{lem}
\begin{proof}
Set $B=\frac{{D-\rho I_2}}{\omega}$. The eigenvalues of $B$ are $\pm i$ so that $B^2=-I_2$. For any $v\in\ER^2$, set $P_v=(v,Bv)$.
The matrix $P_v$ is clearly invertible and using that $B^2v=-v$, one obtains that $B=P_v J_2 P_{v}^{-1}$.
The result follows. \wwtbp
\end{proof}

\begin{pro}\label{genetransf} Let $(\zeta_t)_{t\ge0}$ be a solution to \eqref{eq:gene-syst} where $A\in{GL}_{2}(\ER)$ with complex eigenvalues given by $\lambda_{\pm}=-{\rho}{\pm}i{\omega}$ (with $\rho\in\ER$ and $\omega\in\ER_+^*$).
For any $\alpha\in\ER^*$ and $v\in\ER^2\backslash\{0\}$, set $\hat{\zeta}_t={\sqrt{\omega}\alpha} P_v^{-1} \zeta_{\frac{t}{\omega}}$. The process $(\hat{\zeta}_t)_{t\ge0}$ is a solution to 
\begin{equation}
d\hat{\zeta}_t=-\frac{\rho}{\omega} \hat{\zeta}_t+ J_2 \hat{\zeta}_t+\alpha P_{v}^{-1}\Sigma d W_t
\end{equation}
where $(W_t)$ is a standard two-dimensional Brownian motion.\\
\end{pro}
\begin{proof}
First, set  $\tilde{\zeta}_t^v=P_v^{-1} \zeta_t$. Owing to the preceding lemma, $(\tilde{\zeta}^v_t)_{t\ge0}$ is a solution to 
$$
d\tilde{\zeta}^v_t= -\rho \tilde{\zeta}^v_t +\omega J_2 \tilde{\zeta}_t^v+P_{v}^{-1}\Sigma dB_t.$$
For any $\alpha\in\ER^*$, set $\hat{\zeta}_t={\sqrt{\omega}\alpha}\tilde{\zeta}_{\frac{t}{\omega}}^v$. Setting $W_t=\sqrt{\omega} B_{\frac{t}{\omega}}$ (which is a Brownian motion), one checks that 
$$d\hat{\zeta}_t=-\frac{\rho}{\omega} \hat{\zeta}_t+ J_2 \hat{\zeta}_t+\alpha P_{v}^{-1}\Sigma d W_t.$$\wwtbp
\end{proof}
\smallskip
We now apply this proposition to our problem.

\begin{cor}\label{cor:trmod} Let $(Z_t)_{t\ge0}$ be a solution to \eqref{eds2}  and assume that {$a<4b$}. Set $\omega=\sqrt{ab-\frac{a^2}{4}}$. Let $v\in \ER^2\setminus \{0\}$ and set $P_v=(v,B v)$ with $B=\frac{1}{\omega}(A+\frac{a}{2} I_2)$.  Then for any $\alpha\in\RR\setminus\{0\}$, the process $(\hat{Z}_t)_{t\ge0}$ defined by $\hat{Z}_t={\sqrt{\omega}\alpha} P_v^{-1} Z_{\frac{t}{\omega}}$ is a  solution to 
\begin{equation}\label{defzzz}
d\hat{Z}_t=-\frac{a}{2\omega} \hat{Z}_t+ J_2 \hat{Z}_t+\alpha c P_{v}^{-1}\Sigma d W_t\quad \textnormal{with} 
\quad
\Sigma=\begin{pmatrix} 1 & 0\\
0& 0\end{pmatrix},
\end{equation}
where $W$ is a standard two-dimensional Brownian motion.
In particular, if 
 $v=(\frac{1}{b^2}(\frac{a}{2}-b),1)^*$ and $\alpha=\frac{\sqrt{2}\omega}{c b^2}$, then
$(\hat{Z}_t):=(U_t,V_t)$ is a solution to 
\begin{equation}\label{eq:hypocoercif}
\begin{cases}
d U_t=-\frac{a}{2\omega} U_t-V_t dt\\
dV_t= -\frac{a}{2\omega} V_t+U_t+\sqrt{2}dW_t.
\end{cases}
\end{equation}
where $W$ is now a standard one-dimensional Brownian motion.
\end{cor}
\begin{rem} In the second part, of the corollary, one remarks that one chooses $v$ in order that 
the transformed process has only a (normalized) diffusive component on the second coordinate.
%

Furthermore,  if $Z$ has $(x,y)\in\RR^2$ for initial deterministic condition, then $\hat{Z}$ starts from the point $\sqrt{\omega}c^{-1}b^{-2}(\omega y,
b^2x +(b-a/2)y)$. The images of $(1,0)^*$ and $(0,1)^*$ by $P_v^{-1} $ are particularly important for our purposes, since they enable to see
that the half-plane $\{(x,y)\in\RR^2\st x>0\}$ for $Z$ is transformed into the half-plane $\{(u,v)\in\RR^2\st v>
\frac{2b-a}{2\omega}u\}$ for $\hat{Z}$.
Note that in the setting  $a << b$, the latter half-plane  is quite similar to the former one, since $\omega \sim \sqrt{ab} << b$.


\end{rem}
\begin{proof} We recall that $(Z_t)=(X_t,Y_t)$ is a solution
to 
$$dZ_t=A Z_t+  c\Sigma  dB_t\quad\textnormal{where }\quad A=\begin{pmatrix} b-a & 1\\
-b^2& -b\end{pmatrix}\quad\textnormal{and }\quad 
$$
and $\Sigma$ is defined in \eqref{defzzz}.
When $a<4b$, the eigenvalues of  $A$ are given by
$$\lambda_{\pm}=-\frac{a}{2}\pm i\omega.$$
For any $v\in\ER^2$, set $P_v:=(v,Bv)$ with 
$$B=\frac{1}{\omega}{(A+\frac{a}{2} I_2)}.$$
Applying the previous proposition, we deduce that  for any $\alpha\in\ER^*$, $(\hat{Z}_t)_{t\ge0}:= (\alpha\sqrt{\omega} P_v^{-1} Z_{\frac{t}{\omega}})_{t\ge0}$
is a solution to 
$$d\hat{Z}_t=-\frac{a}{2\omega} \hat{Z}_t+ J_2 \hat{Z}_t+\alpha c P_{v}^{-1}\Sigma d\hat{W}_t$$   
where $(\hat{W}_t)$ is a standard two-dimensional Brownian motion. 
\smallskip
For the second part, it remains to choose $v$ and $\alpha$ so that 
\begin{equation}\label{eq:condpv}
\alpha c P_{v}^{-1}\Sigma=\begin{pmatrix} 0&0\\
\sqrt{2}&0\end{pmatrix}.
\end{equation}
If $v=(u_1,u_2)^*$, then 
\begin{equation}\label{matrixpv}
P_v=\begin{pmatrix} u_1 &\frac{1}{\omega} \left((b-\frac{a}{2})u_1+u_2\right)\\
&\\
u_2& -\frac{1}{\omega} \left(b^2 u_1+(b-\frac{a}{2})u_2\right)\end{pmatrix}
\end{equation}
One deduces  that condition \eqref{eq:condpv} (or more precisely the fact that $(P_v^{-1}\Sigma)_{1,1}=0$)  implies  $b^2 u_1+(b-\frac{a}{2}) u_2=0$. Setting $v=(\frac{1}{b^2}(\frac{a}{2}-b),1)^*$, we have 
$$P_v=\begin{pmatrix} \frac{1}{b^2}\left(\frac{a}{2}-b\right) &\frac{\omega}{b^2}\\
 1& 0\end{pmatrix}\quad\textnormal{and}\quad P_v^{-1}=\begin{pmatrix} 0 &1\\
 \frac{b^2}{\omega}& \frac{1}{\omega}\left(\frac{a}{2}-b\right)\end{pmatrix}.
$$
Condition  \eqref{eq:condpv} is then satisfied when $\alpha=\frac{\sqrt{2}\omega}{c b^2}$.  \wwtbp

%
%
\end{proof}

\section{Dirichlet eigenvalues estimates}\label{Dee}
This section is devoted to the proof of Theorem \ref{T1}. More precisely, the aim is to obtain successively upper and lower bounds for
$\PE_{(x_0,y_0)}(\tau>t)$ where $\tau:=\inf\{t\ge0,X_t\le 0\}$. In fact, some of the results will be  stated for exit times of more general domains. For a given (open) domain ${\cal S}$ of $\ER^2$, we will thus denote by 
$$\tau_{\cal S}:=\inf\{t\ge0, (X_t,Y_t)\in {\cal S}^c\}.$$
\subsection{Upper-bound for the exit time of an angular sector ${\cal S}$}
\subsubsection{The case ${\cal S}=\{(x,y),x>0\}$}
In this part, we focus on the particular stopping time $\tau$ of Theorem \ref{T1} which corresponds to the exit time of $D=\{(x,y),x>0\}$. We have the following result:
\begin{pro}\label{upper1erresult} Let $(Z_t)_{t\ge0}$ be a solution to \eqref{eds2} with $a<4b$. Then, for every $(x_0,y_0)\in\ER^2$ such that $x_0>0$,
$$\PE_{(x_0,y_0)}(\tau>t)\le 2\exp\lt(-\frac{\log 2}{\pi}{\omega}t\rt).$$
\end{pro}
\begin{proof} Set $z(t)=(x(t),y(t))=(\ES[X_t],\ES[Y_t])$. The function $(z(t))_{t\ge0}$ being a solution to $\dot{z}=A z$, we deduce in particular that 
$$\forall \, t\ge0, \quad \overset{..}{x}(t)+a\dot{x}+ ab x(t)=0.$$
Since $a<4b$, the roots of the characteristic equation associated with the previous equation are: $\lambda_{\pm}=-\frac{a}{2}\pm \omega$
where $\omega=\sqrt{ab-a^2/4}$. As a consequence, there exists $C>0$ and $\varphi_0\in(-\pi,\pi]$ such that 
$$x(t)=C\cos(\omega t+\varphi_0),\quad t\ge0. $$
Reminding that $x_0>0$, we deduce that $\varphi_0\in(-\frac{\pi}{2},\frac{\pi}{2})$. Thus, at time $T_\omega=\frac{\pi}{\omega}$, 
$$\forall x_0>0,\quad \omega T_\omega+\phi_0\in(\frac{\pi}{2},\frac{3\pi}{2})\;\Longrightarrow\, x(T_\omega)<0.$$
But $x(T_\omega)=\ES[X_{T_\omega}]$ and $X_{T_\omega}$ is Gaussian and has a symmetric distribution. Thus, we deduce from what precedes that 
$$\forall x_0>0,\,y_0\in\ER,\quad \PE_{(x_0,y_0)}(X_{T_\omega}<0)\ge \frac{1}{2}$$
which in turn implies that
\begin{equation}\label{taule12}
\forall x_0>0,\,y_0\in\ER,\quad \PE_{(x_0,y_0)}(\tau\ge T_\omega)\le \frac{1}{2}.
\end{equation}
Thus, we have a upper-bound at time $T_\omega$ which does not depend on the initial value $(x_0,y_0)$. As a consequence, we can  use a Markov argument. More precisely, owing to the Markov property and to \eqref{taule12}, we have for every integer $k\ge 1$:
$$\PE(\tau> kT_\omega |\tau> (k-1)T_\omega)=\frac{\ES[\PE_{(X_{(k-1)T_\omega},Y_{(k-1)T_\omega})}(\tau>T_\omega)1_{\tau> (k-1)T_\omega}]}{\PE(\tau> (k-1)T_\omega)}\le \sup_{x_0>0,y_0\in\ER}\PE_{(x_0,y_0)}(\tau>T_\omega).$$
An iteration of this property yields 
$$\forall n\in\NN,\;\forall (x_0,y_0)\in\ER_+^*\times \ER,\quad\PE_{(x_0,y_0)}(\tau>nT)\le \left(\frac{1}{2}\right)^n.$$
It follows that 
$$\forall t\ge0,\;\forall (x_0,y_0)\in\ER_+^*\times \ER,\quad\PE_{(x_0,y_0)}(\tau>t)\le \left(\frac{1}{2}\right)^{\lfloor \frac{t}{T_\omega}\rfloor}\le 2\exp\lt(-\frac{\log 2}{T_\omega} t\rt).$$
This concludes the proof.
\end{proof}
\subsubsection{Extension to general angular sectors}
We now consider an angular sector ${\cal S}_{\alpha_1,\alpha_2}$ defined as
\begin{equation}\label{st1t2}
{\cal S}_{\alpha_1,\alpha_2}=\{(x,y)\in\ER^2,\,x>0\,,\; \alpha_1 x<{y}<\alpha_2 x\}
\end{equation}
where $\alpha_1,\alpha_2\in\ER$ and $\alpha_1<\alpha_2$. The set ${\cal S}_{\alpha_1,\alpha_2}$ can also be written 
${\cal S}_{\alpha_1,\alpha_2}=\{(r\cos \theta,r\sin\theta), r>0,\theta_1<\theta<\theta_2\}$ with $\theta_1,\theta_2\in[-\pi/2,\pi/2]$. Note that for the sake of simplicity, we only consider angular sectors which are included in $\{(x,y),x>0\}$. The results below can  be extended to any angular sectors  
for which the angular size is lower than $\pi$.
For such domains, we first give a result when the model has a constant (mean) angular speed even if such a result does not apply to  
the solutions of \eqref{eds2} for sake of completeness. This is the purpose of  Lemma \ref{lemma:ass} below. 

Concerning now our initial motivation, we also derive an extension of Proposition \ref{upper1erresult} for any general angular sector, and this result is stated in Proposition \ref{upper2result}.
\begin{lem}\label{lemma:ass}
Let $(Z_t)_{t\ge0}$ be a solution of
$$dZ_t=-\rho Z_t+\omega J_2 Z_t +\Sigma dW_t$$
where $\rho\in\ER$, $\omega\in\ER_+^*$, $\Sigma\in \mathbb{M}_2(\ER)$ and $W$ is a two-dimensional Brownian motion.
Let ${\cal S}_{\alpha_1,\alpha_2}$ be defined by \eqref{st1t2} where $\alpha_1,\alpha_2\in\ER$ and $\alpha_1<\alpha_2$. 
Then, for any $(x_0,y_0)\in {\cal S}_{\alpha_1,\alpha_2}$, 
$$\PE(\tau_{{\cal S}_{\alpha_1,\alpha_2}}\ge t)\le 2\exp\lt(-\frac{\ln(2)}{\theta_2-\theta_1}{\omega}t\rt)$$
with $\theta_1={\rm Arctan}(\alpha_1)$ and $\theta_2={\rm Arctan}(\alpha_2).$
\end{lem}
\begin{proof}
Let $z(t)=(\ES[X_t],\ES[Y_t])$ and define $(u(t))_{t \geq 0} :=(e^{\rho t} z(t))_{t\ge0}$, $u$ is a solution of
$$\dot{u}=\omega J_2 u.$$
We deduce that
$$e^{\rho t}z(t)= (A\cos(\omega t+\varphi),A\omega\sin(\omega t+\varphi))$$
where $A\ge0$ and $\varphi\in[-\pi,\pi)$. This implies that the angular rate of $(z(t))_{t\ge0}$ is constant and is equal to $\omega$.  Thus, it follows that for every starting point $(x,y)\in {\cal S}_{\alpha_1,\alpha_2}$,
$$ z(T_\omega)\in {\cal S}_{\alpha_1,\alpha_2}^c\quad\textnormal{with}\quad T_\omega=\omega (\theta_2-\theta_1).$$

One can then find a line 
passing  through $0$ 
and dividing $\RR^2$
into two half-planes  $D^+$ and $D^-$ such that  ${\cal S}_{\alpha_1,\alpha_2}$ is included in $D^-$ and $z(T_\omega)\in D^+$. Owing to the symmetry of a one-dimensional centered Gaussian distribution, we have
$$\PE((X_{T_\omega},Y_{T_\omega})\in D^-)= \PE((X_{T_\omega},Y_{T_\omega})\in D^+)=\frac{1}{2}.$$
One finally deduces that
for every $(x,y)\in{\cal S}_{\alpha_1,\alpha_2}$,
$$\PE((X_{T_\omega},Y_{T_\omega})\in {\cal S}_{\alpha_1,\alpha_2}^c)\ge\frac{1}{2}.$$
and thus that 
$$\forall\, (x,y)\in{\cal S}_{\alpha_1,\alpha_2},\quad \PE_{(x,y)}(\tau_{{\cal S}_{\alpha_1,\alpha_2}}>T_\omega)\le \frac{1}{2}.$$
The end of the proof is then identical to that of Proposition \ref{upper1erresult}. \wwtbp


\end{proof}
\smallskip
We now consider our initial bubble process $(Z_t)_{t\ge0}$ which is solution of Equation \eqref{eds2}. We have the following result.
 \begin{pro}\label{upper2result}
 Let $(Z_t)_{t\ge0}$ be a solution to \eqref{eds2}. Let ${\cal S}_{\alpha_1,\alpha_2}$ be defined by \eqref{st1t2} where $\alpha_1,\alpha_2\in\ER$ and $\alpha_1<\alpha_2$. 
  Then, for any $(x_0,y_0)\in {\cal S}_{\alpha_1,\alpha_2}$, 
$$\PE(\tau_{{\cal S}_{\alpha_1,\alpha_2}}\ge t)\le2\exp\lt(-\frac{\ln 2}{\tilde{\theta}_1-\tilde{\theta}_2}{\omega}t\rt)$$
with $\tilde\theta_i={\rm Arctan}\lt(\frac{a/2-b-\alpha_i}{\omega}\rt)$, $i=1,2$. 
 \end{pro}
 \begin{rem} Taking $\alpha_1=-\infty$ and $\alpha_2=+\infty$, we retrieve Proposition \ref{upper1erresult} since ${\cal S}_{-\infty,+\infty}$ then corresponds to the half-plane $\{x>0\}$. Note that contrary to Lemma \ref{lemma:ass}, the exponential rate is not directly proportional to  $\omega$. More precisely, due to the non constant angular speed, $\tilde{\theta}_1$ and $\tilde{\theta}_2$   depend on $\omega$. For the particular domain of Proposition \ref{upper1erresult} this dependence does not appear since, even if the the angular rate is not constant, the time to do a \textit{U-turn} is still proportional to $\omega$.
\end{rem}
\begin{proof}
By Corollary \ref{cor:trmod}, for any $v$ of $\ER^2\setminus\{0\}$, $(\tilde{Z}_t)_{t\ge0}=(P_v^{-1} Z_\frac{t}{\omega})_{t\ge0}$ (where $P_v=(v,Bv)$) is a solution of 
$$d\tilde{Z}_t=(-\rho I+ J_2)\tilde{Z}_t+ \tilde{\Sigma} dW_t$$
where $\tilde{\Sigma}$ is a constant real matrix (whose exact expression is not so   important) and where $\rho = a/(2\omega)$. In the new basis $\tilde{\cal B}=(v,Bv)$, 
$${\cal S}_{\alpha_1,\alpha_2}=\{\tilde{z}=(\tilde{x},\tilde{y})_{\tilde{\cal B}}\in\ER^2, \alpha_1 (P_v \tilde{z})_1<(P_v \tilde{z})_2<\alpha_2 (P_v \tilde{z})_1\}.$$ 
Setting $v=(1,\frac{a}{2}-b)$, we deduce from Equation \eqref{matrixpv} that 
$$
P_v=\begin{pmatrix} 1 & 0\\
(\frac{a}{2}-b) & -{\omega}\end{pmatrix}
$$
In such a case
\begin{equation*}
P_v\tilde{z}=\begin{pmatrix} \tilde{x}\\
&\\
(\frac{a}{2}-b)\tilde{x} -{\omega}\tilde{y}\end{pmatrix}
\end{equation*}
so that
$${\cal S}_{\alpha_1,\alpha_2}=\{\tilde{z}=(\tilde{x},\tilde{y})_{\tilde{\cal B}}\in\ER^2, \left(\frac{a}{2}-b-\alpha_2\right)\tilde{x}<\omega \tilde{y}<\left(\frac{a}{2}-b-\alpha_1\right)\tilde{x}\}.$$
Thus, we deduce from Lemma \ref{lemma:ass} that 
$$\PE_{\tilde{x},\tilde{y}}(\tilde{\tau}_{{\cal S}_{\alpha_1,\alpha_2}}\ge t)\le 2\exp\lt(-\frac{\ln 2}{\tilde{\theta}_1-\tilde{\theta}_2} t\rt)$$
where for a given domain $A$, $\tilde{\tau}_A:=\inf\{t\ge0, \tilde{Z}_t\in A^c\}$ and $\tilde{\theta}_i={\rm Arctan}\lt(\frac{a/2-b-\alpha_i}{\omega}\rt)$, $i=1,2$. The result follows. \wwtbp
\end{proof}
\subsection{Lower-bound\label{sec:lbound}}

\subsubsection{General tool}
In this second part, our aim is to obtain the lower-bound part of Theorem \ref{T1}, in particular  we want to derive a upper-bound on:
\bq\bar{\lambda}&\df&\underset{t\rightarrow+\infty}{\limsup}\,-\frac{1}{t}\log(\PE(\tau \ge t)).\eq
The results of this section are based on the following (classical) proposition.
\begin{pro}\label{prodynkin} Let $(X_t)_{t\ge0}$ be a $\ER^d$-valued Markov process with infinitesimal generator $L$ and initial distribution ${m_0}$. Let  ${\cal S}$ be an (open) domain of $\ER^d$ and assume that ${m_0}({\cal S})=1$. Let $\tau:=\inf\{t>0, X_t\in {\cal S}^c\}$. Then,
if there exists a bounded function $f:\ER^d\rightarrow\ER$ and $\lambda\in\ER$ such that 
\begin{equation}\label{eq:dirichletcond}
\begin{cases}& f_{/\partial {\cal S}}=0\quad \textnormal{and} \quad f_{/{\cal S}}>0\\
& \forall x\in {\cal S},\quad Lf(x)\ge -\lambda f(x)
\end{cases}
\end{equation}
then,   $\ES_{m_0}[e^{\lambda\tau}]=+\infty$. As a consequence,
\bq
 \underset{t\rightarrow+\infty}{\limsup}\,-\frac{1}{t}\log(\PE_{m_0}(\tau\ge t))&\le &\lambda.
 \eq
 \end{pro}
 \begin{proof}
 Owing to the Dynkin formula, we have for every $t\ge 0$
 $$e^{\lambda (t\wedge \tau)} f (X_{t\wedge \tau})= f(X_0)+\int_0^{t\wedge \tau}  (\lambda f(X_{s\wedge \tau})+Lf(X_{s\wedge \tau}))e^{\lambda (s\wedge \tau)} ds+M_{t\wedge \tau}$$
 where $(M_t)$ is a local martingale. Owing to a localization argument and to Fatou's Lemma, we deduce that 
$$ \ES_{m_0}[e^{\lambda (t\wedge \tau)} f (X_{t\wedge \tau})]= \int f(x){m_0}(dx)+\ES_{m_0}\lt[\int_0^{t\wedge \tau}  (\lambda f(X_{s\wedge \tau})+Lf(X_{s\wedge \tau}))e^{\lambda (s\wedge \tau)} ds\rt].$$
Suppose that  $\ES_{m_0}[e^{\lambda\tau}]<+\infty$. Then, using the dominated convergence theorem and the fact that 
$f_{/\partial {\cal S}}=0$, we have
 
$$\lim_{t\rightarrow+\infty}\ES_{m_0}[e^{\lambda (t\wedge \tau)} f (X_{t\wedge \tau})]=0$$
whereas the fact that $Lf\ge -\lambda f$ implies that the right-hand term is uniformly lower-bounded by $\int f(x){m_0}(dx)$ which is (strictly) positive. This yields a contradiction.
Thus, $\ES_{m_0}[e^{\lambda\tau}]=+\infty$. The second assertion then classically follows from the equality
$$\ES_{m_0}[e^{\lambda \tau}]=\int_{\RR} \lambda e^{\lambda t}\PE_{m_0}(\tau>t)dt.$$ \wwtbp
\end{proof}

\smallskip
The end of  Section \ref{sec:lbound} is devoted to the construction of a function $f$ satisfying \eqref{eq:dirichletcond}.
In fact, for this part, the degeneracy of the process described by Equation \eqref{eds2} implies a significant amount of difficulties. That is why we propose in the next subsection to first focus on the elliptic case which can be handled  more easily.  Some of the ideas developed in this framework will then be extended to the initial hypoelliptic setting.
\subsubsection{The elliptic case}\label{subsec:ellip}
By Corollary \ref{cor:trmod}, we know that we can reduce the problem to the study of a process $(U_t,V_t)$ solution to \eqref{eq:hypocoercif}. In this part we focus on its elliptic counterpart: we consider a two dimensional process $(\xi_t)_{t\ge0}$ solution of the following stochastic differential equation
\begin{equation}\label{eq:xitt}
d\xi_t= (-\rho \xi_t+J_2\xi_t)dt+\sqrt{2}dW_t
\end{equation}
where $\rho$ is a real number and $W$ is a standard two-dimensional Brownian motion.\\
We define ${\cal S}_0=\{(x,y),x>0\}$.
In order to build a suitable function $f$ satisfying \eqref{eq:dirichletcond} with $D={\cal S}_0$ and relatively to the  infinitesimal  generator $L_\rho$  associated to \eqref{eq:xitt}, we first switch to  polar coordinates: Proposition \ref{prop:appendix1} (stated in the second appendix) shows that  $L_\rho$ is given on ${\cal C}^2(\ER_+^*\times\ER)$ by
\begin{equation}\label{opellippol}
L_\rho =-\rho r\partial_r  +\partial_\theta +\partial ^2_r +\frac{1}{r}\partial_r +\frac{1}{r^2}\partial^2_\theta.
\end{equation}
To  justify the approach developed below, let us
 forget formally the derivatives with respect to $r$ by fixing $r>0$ and by considering unknown angular functions $G_r\st [-\pi/2,\pi/2]\ri\RR$. The problem of finding $G_r$ and $\lambda_r$ such that \eqref{eq:dirichletcond} holds (with $D=(-\pi/2,\pi/2)$) reduces to solve the second order ordinary differential equation 
$$r^{-2} G''_r(\theta)+G'_r(\theta)=-\lambda_r G_r(\theta), \quad -\frac{\pi}{2}\le \theta\le \frac{\pi}{2}$$
with $G_r(\pi/2)=G_r(-\pi/2)=0$. The solutions are  given by exponential functions $G_r(\theta) = \alpha_1 e^{\rho_1 \theta} + \alpha_2 e^{\rho_2 \theta}$ where $(\rho_1,\rho_2)$ are the complex roots of the quadratic characteristic equation $X^2/r^2+X+\lambda_r=0$ associated to the linear ODE given above. One can easily check that $\Re(\rho_1)=\Re(\rho_2)= - r^2/2$ and 
the boundary conditions imply in particular to choose $\lambda_r$ such that $\Im(\rho_1) = - \Im(\rho_2) =1 $ otherwise the function $G_r$ cannot remain positive on $]- \pi/2;\pi/2[$.
This is possible if and only if
$\lambda_r=\frac{1}{r^{2}}+\frac{r^2}{4}$ and the solutions of this simplified spectral problem are  then proportional to the function defined by
\bq
\fo \theta\in[-\pi/2,\pi/2],\qquad G_r(\theta)&\df&e^{-\frac{r^2}{2} \theta}\cos \theta.\eq
The previous construction cannot really be extended to the exact initial  problem $L_{\rho} g= - \lambda g$. Nevertheless, this suggests  
to search potential solutions $g$ to the problem $L_{\rho} g \geq - \lambda g$
under the following form
\begin{equation}\label{eq:gg}
g(r,\theta)=re^{\beta(\theta) r^2}\cos (\theta),\quad\theta\in\lt[-\frac{\pi}{2},\frac{\pi}{2}\rt],\; r\ge0
\end{equation}
where $\beta:[-\pi/2,\pi/2]\rightarrow\ER$ is a ${\cal C}^2$-function which must be non-positive (due to the boundedness condition in Proposition \ref{prodynkin}).
 The action of the generator $L_{\rho}$ described in \eqref{opellippol} is given as follows (the proof is postponed to the second appendix).

\begin{pro}\label{prop:lrhog_elliptic}
For any $g \in \mathcal{C}^2\lt(\ER_+ \times \lt[-\frac{\pi}{2},\frac{\pi}{2}\rt] , \ER\rt)$ given by \eqref{eq:gg}, one has  
$$\forall (r,\theta)\in \ER_+^*\times \lt[-\frac{\pi}{2},\frac{\pi}{2}\rt]\qquad 
L_\rho g (r,\theta)= \lt[ \psi_1(\theta) r^2+\psi_2(\theta) \rt] g(r,\theta)$$
where 
\begin{align*}
&\psi_1(\theta)=-2\rho\beta(\theta)+ \beta'(\theta)+\left(4{\beta^2(\theta)}+(\beta'(\theta))^2\right)\\
&\psi_2(\theta)=-\rho+8\beta(\theta)-(1+2\beta'(\theta))\tan\theta+\beta''(\theta).
\end{align*}
\end{pro} 

In order to apply Proposition \ref{prodynkin}, the problem is now reduced to find a non-positive angular function $\beta$ such that  $\frac{L_\rho g}{g}$ is lower-bounded on $\ER_+^*\times ]-\frac{\pi}{2},\frac{\pi}{2}[$. 
 The previous computation shows that we mainly need to satisfy the following constraint
$$\forall \theta\in\lt]-\frac{\pi}{2},\frac{\pi}{2}\rt[, \quad\psi_1(\theta)\ge0.$$
An admissible value of $\lambda_{\rho}$ for the spectral inequality $L_{\rho} g \geq - \lambda_\rho g$ will  then be obtained by
$$ \lambda_\rho:=\inf_{\theta\in]-\frac{\pi}{2},\frac{\pi}{2}[}\psi_2(\theta)>-\infty.$$
 Note that this implies in particular that
$$\limsup_{\theta\rightarrow\frac{\pi}{2}} 1+2\beta'(\theta)\le 0\quad\textnormal{and}\quad\liminf_{\theta\rightarrow-\frac{\pi}{2}} 1+2\beta'(\theta)\ge 0.$$
A solution of the problem is given in the next proposition.
\begin{pro} (i) Let $\rho\ge 0$ and let $g$ be given by \eqref{eq:gg} with
\begin{equation}
\beta(\theta)=\begin{cases}\frac{1}{4}(1-\sqrt{3})&\textnormal{if $\theta\in[-\frac{\pi}{2},\frac{\pi}{4})$}\\
\frac{1}{4}(\sin(2\theta)-\sqrt{3}) &\textnormal{if $\theta\in[\frac{\pi}{4},\frac{\pi}{2}]$}.
\end{cases}
\end{equation}
Then, for every $r>0$ and $\theta\in [-\pi/2,\pi/2]$ such that $\theta\neq\pi/4$, 
$$L_\rho g(r,\theta)\ge -\lambda_\rho g(r,\theta) \quad \textnormal{with $\lambda_\rho=2\sqrt{3}+\rho$}.$$
(ii) Let $\rho\ge 0$ and consider $(\xi_t)_{t\ge0}$ solution to \eqref{eq:xitt}. Then, for every $(x_0,y_0)\in\ER^2$ such that
$x_0>0$,
$$\limsup_{t\rightarrow+\infty}\,-\frac{1}{t}\log(\PE_{(x_0,y_0)}(\tau \ge t))\le \lambda_\rho$$
where $\tau:=\inf\{t\ge0, (\xi_t)_1<0\}.$
\end{pro}
\begin{rem} Note that by the scaling and linear transformations previously described, this result can be transferred to general
elliptic two-dimensional Ornstein-Uhlenbeck evolutions whose linear drift is given via a matrix admitting complex conjugate eigenvalues (and whose  trajectories have thus a tendency to turn around $(0,0)$).

The function $g$ is a ${\cal C}^1$-function but only a piecewise ${\cal C}^2$-function. However, since these functions still belong to the domain of $L_\rho$, the conclusions of Proposition \ref{prodynkin} still hold. Also note that if we now switch to cartesian coordinates, the counterpart of $g$ has the following (nice) form:
$$f(x,y)=xe^{-\frac{\sqrt{3}-1}{4}(x^2+y^2)}e^{-\frac{(x-y)^2}{4}1_{\{y\ge x\}}}.$$
\end{rem}
\begin{proof} (i) First, assume that $\rho=0$ and consider  the function $\beta$ stated in the former statement.
One checks that $\beta$ is a piecewise ${\cal C}^2$-function on $[-\pi/2,\pi/2]$. Note that we can use this function since Itô's formula is still available in this case. Furthermore, we check that 
\begin{equation}
\psi_1(\theta)=\begin{cases}1-\frac{\sqrt{3}}{2}&\textnormal{if $\theta\in[-\frac{\pi}{2},\frac{\pi}{4})$}\\
1+\cos(\frac{\pi}{3}+2\theta) &\textnormal{if $\theta\in[\frac{\pi}{4},\frac{\pi}{2}]$}
\end{cases}
\end{equation}
so that $\psi_1$ is non-negative on $[-\frac{\pi}{2},\frac{\pi}{2}]$.  As well, easy computations yield:
\begin{equation}
\psi_2(\theta)=\begin{cases}2-2\sqrt{3}-\tan \theta&\textnormal{if $\theta\in[-\frac{\pi}{2},\frac{\pi}{4})$}\\
-2\sqrt{3} &\textnormal{if $\theta\in[\frac{\pi}{4},\frac{\pi}{2}]$}
\end{cases}
\end{equation}
It follows that $\psi_2$ is lower-bounded by $-2\sqrt{3}$ and the result follows when $\rho=0$.
The extension to the case $\rho>0$ is obvious using that $-\rho\beta$ is a non-negative function. \par
\smallskip
(ii) This statement follows from  Proposition \ref{prodynkin}, since $g$ is ${\cal C}^1$ and  piecewise ${\cal C}^2$. \wwtbp
\end{proof}
\begin{rem}\label{rem:interpretation}
Figure \ref{fig:vector1} represents the partition of the state space $\RR^2$ (seen as $\ER_+ \times [0,2\pi)$ in the second picture) for the construction of the function $\beta$ (and  $g$) as well as the function $g(r,\theta)$ for several values of $r$. We should understand the function $g$ as follows: $g(r,\theta)$ must be large when the dynamical system is suspected to take 
long time to exit the set ${\cal{S}}_0=\ER_+^* \times \lt( -\pi/2,\pi/2\rt)$ from $(r,\theta)$. Conversely, it should be small in the region where the vector field of the underlying determinist dynamical system push the trajectories out of  ${\cal{S}}_0$.
As pointed out by Figure \ref{fig:vector1}, we do not need to consider sub-domain of ${\cal{S}}_0$: the action of the Brownian motion is elliptic and we can always build some trajectories starting from any point of ${\cal{S}}_0$ and staying an arbitrarily long time in ${\cal{S}}_0$. Note that when $r$ is small, the starting point is near the origin, whatever the value of $\theta$ is and hence, the function $g(r,\theta)$ is small (see the right side of Figure \ref{fig:vector1}).

\begin{figure}[h]
 \centering
\includegraphics[height=6cm]{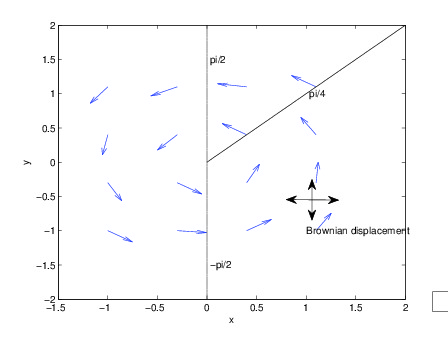}
\includegraphics[height=6cm]{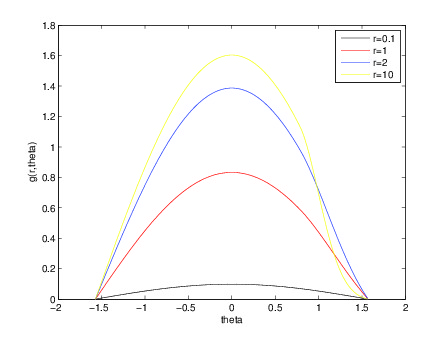}
\caption{\label{fig:vector1} Left: the domain to avoid is $\theta \in \lt[\pi/2,3\pi/2\rt]$. The elliptic situation is illustrated by the full rank black double arrow: the Brownian motion always move in all directions. In blue: rotation + homothety vector field. Right: function $\theta \mapsto g(r,\theta)$ for several values of $r$.}
\end{figure}

\end{rem}

\subsubsection{The hypoelliptic case}\label{subsec:hypo}
We now come back to the study of the lower-bound of Theorem \ref{T1}. This result is proved in Proposition \ref{pro:low2} stated below.
We know from Corollary \ref{cor:trmod} that up to  linear changes of variables in time  and space, the initial dynamic may be reduced to the simplified stochastic evolution  described by Equation  \eqref{eq:hypocoercif}.
Again let us  write down the corresponding infinitesimal generator ${\cal L}_\rho$ 
in 
polar coordinates (see Proposition \ref{prop:appendix1} given in the second appendix):

\begin{equation}\label{operhypo}
{\cal L}_\rho=-\rho r \partial_r  +\partial_\theta +\frac{\sin^2\theta}{2}\partial^2_{rr}-\frac{\sin\theta \cos\theta}{r^2}\partial_\theta
 +\frac{\sin\theta \cos\theta}{r}\partial^2_{r\theta}+\frac{\cos^2\theta}{2r}\partial_r+\frac{\cos^2\theta}{2r^2}\partial^2_\theta
 \end{equation}
 with $\rho=-\frac{a}{2\omega}$.
 As mentioned before, we would like to use a strategy similar to the one considered  in the elliptic case. However, the hypoelliptic problem is more involved. Roughly speaking, the degeneracy of the diffusive component implies that in the neighbourhood of $\pi/2$, the paths of the solutions to \eqref{eq:hypocoercif} can not be strongly slowed down by the action of the Brownian motion (see Remark \ref{rem:interpretation} and the study on Brownian bridges below). In other words, we are not able (and it seems indeed impossible) to build a function $\beta$ such that the function $\psi_2$ defined in the previous subsection is lower-bounded. Thus, the idea is to reduce the domain to a smaller angular sector ${\cal S}$ included in $\{(x,y),x>0\}$ where the diffusive action of the Brownian motion is more likely to keep the process in ${\cal S}$.
 
Consequently, we consider a more general class of functions $g$ (which must be calibrated in the sequel) and define
\begin{equation}\label{eq:ghyp} 
g(r,\theta)=r^n{\gamma}(\theta) e^{\beta(\theta)r^2}
\end{equation} 
where $n$ is a positive integer and ${\gamma}$ and $\beta$ are some sufficiently smooth functions. Now $\beta$ should be bounded above by a negative constant for $g$ to have a chance to be bounded. The new function ${\gamma}$ will be chosen in order that $g$ is positive in the interior of the angular sector and vanishes on the boundary of $\cal S$. We first describe the effect of ${\cal L}_{\rho}$ on such a function $g$ (the computations are deferred to the second appendix).

\begin{pro}\label{prop:lrhog_hypo}
For any $g \in \mathcal{C}^2\lt(\ER_+ \times \lt[-\frac{\pi}{2},\frac{\pi}{2}\rt] , \ER\rt)$ given by \eqref{eq:ghyp}, one has  
$$\forall (r,\theta)\in \ER_+^*\times \lt[-\frac{\pi}{2},\frac{\pi}{2}\rt]\qquad 
L_\rho g (r,\theta)= \lt[ \varphi_1(\theta) r^2+\varphi_2(\theta)+\frac{\varphi_3(\theta)}{r^2} \rt] g(r,\theta)$$
where 
\begin{align*}
&\varphi_1(\theta)=-2\rho\beta(\theta)+ \beta'(\theta)+\left(2\sin\theta{\beta(\theta)}+\cos\theta\beta'(\theta)\right)^2,\\
&\varphi_2(\theta)=-n\rho+\beta(\theta)\left((4n+2)\sin^2\theta+2\cos^2\theta\right)\\
&+(1+2\beta'(\theta) \cos^2\theta+4\beta(\theta)\sin\theta\cos\theta)\frac{{\gamma}'}{{\gamma}}(\theta)+\beta''(\theta) \cos^2(\theta)+
2(n+1)\cos\theta\sin\theta\beta'(\theta),\\
&\varphi_3(\theta)=(n^2-n)\sin^2\theta+\cos^2\theta(n+\frac{{\gamma}''(\theta)}{{\gamma}(\theta)})+2(n-1)\sin\theta\cos\theta\frac{{\gamma}'(\theta)}{{\gamma}(\theta)}.
\end{align*} 
\end{pro}
We now need to find an (open) angular sector ${\cal S}=\{(r\cos\theta,r\sin\theta),\theta_1<\theta<\theta_2\}$, a positive integer $n$, some functions ${\gamma}$ and $\beta$ such that
\begin{enumerate}
\item{} $\gamma(\theta)>0$ on $(\theta_1,\theta_2)$, $\gamma(\theta_1)=\gamma(\theta_2)=0$, $\beta(\theta)\le 0$ on $[\theta_1,\theta_2]$,
\item{} $\varphi_1$ and $\varphi_3$ are non-negative on  ${\cal S}$,
\item{} $\varphi_2$ is lower-bounded.
\item{} $\beta$ is bounded above by a negative constant.
\end{enumerate}
This is the purpose of the next proposition.
\begin{pro}\label{prophyplow} Let $\rho\ge 0$.\\
(i) Let $g$ be defined by \eqref{eq:ghyp} with $n=2$,
\begin{equation}
{\gamma}(\theta)=\begin{cases}-\sin(2\theta)&\textnormal{if $\theta\in [-\frac{\pi}{2},-\frac{\pi}{4}]$}\\
\cos^2(\pi/4+\theta)&\textnormal{if $\theta\in [-\frac{\pi}{4},\frac{\pi}{4}]$}
\end{cases}
\end{equation}
and $\beta(\theta)=-\frac{1}{2}$. Then, for every $r>0$ and $\theta\in]-\frac{\pi}{2},\frac{\pi}{4}[$ with $\theta\neq -\frac{\pi}{4}$,
\bq {\cal L}_\rho g(r,\theta)&\ge& -(3+2\rho) g(r,\theta).\eq
(ii) As a consequence, for any open half-plane $H$ such that ${\cal S}:=\{(r\cos \theta,r\sin \theta),r>0,\theta\in]-\frac{\pi}{2},\frac{\pi}{4}[\}\subset H$, for any probability measure $m_0$ on $\ER^2$ such that $m_0(H)=1$, we have
\bq \limsup_{t\rightarrow+\infty}-\frac{1}{t}\log (\PE_{m_0}(\tau_{H}\ge t))&\leq &3+2\rho.\eq
\end{pro}
\begin{rem} The vector field corresponding to the drift part of the stochastic evolution under study, as well as the most favorable positions  (which are expected to be the points where $g$ is large) for the starting point in order to keep the process in ${\cal S}$ for large times are illustrated in Figure~\ref{fig:vector2}. As pointed out above, the angular sector $[\pi/4,\pi/2]$ is now avoided to keep the process in the half-plane $x>0$. Moreover, the right side of Figure \ref{fig:vector2} shows that  excessive values of $r$ (too large or too small ones) are also prohibited: small values are unfavourable since it corresponds to starting positions very close to the origin (and naturally close to the axis $x=0$). Large values of $r$ are also disadvantageous owing to the large norm of the drift vector field against which the Brownian motion has to fight to keep the process in ${\cal S}$.

\begin{figure}[h]
 \centering
\includegraphics[height=6cm]{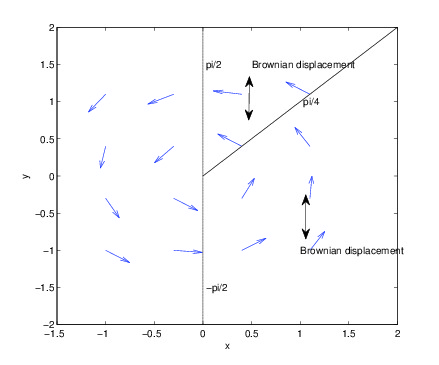}
\includegraphics[height=6cm]{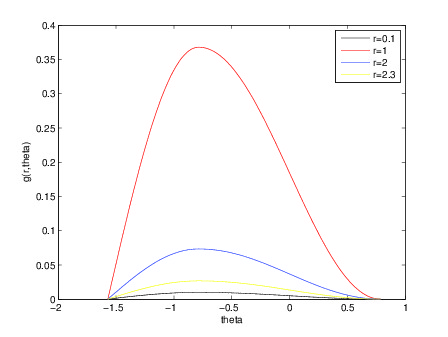}
\caption{\label{fig:vector2} Left: the domain to avoid is $\theta \in \lt[\pi/4,3\pi/2\rt]$. The hypo-elliptic situation is illustrated by the rank 1 double arrow: the Brownian motion can only move in vertical directions. In blue: rotation + homothety vector field. Right: function $\theta \mapsto g(r,\theta)$ for several values of $r$.}
\end{figure}

\end{rem}  
\begin{proof} With the proposed choices of $n$ and ${\gamma}$, one checks that 
$$\varphi_3(\theta)=\begin{cases} 0&\textnormal{if }\theta\in[-\frac{\pi}{2},-\frac{\pi}{4})\\
\frac{2}{1-\sin(2\theta)}&\textnormal{if }\theta\in(-\frac{\pi}{4},\frac{\pi}{4})
\end{cases}
$$
so that $\varphi_3$ is non-negative. Since $\beta$ is constant and $\rho$ is non-negative, the fact that 
$\varphi_1$ is non-negative is obvious. Thus, it remains to focus on $\varphi_2$.
In fact, easy computations show that $\varphi_2(\theta)=-(3+2\rho)$ on $[-\frac{\pi}{4},\frac{\pi}{4})$ whereas
$$\forall \theta\in(-\frac{\pi}{2},-\frac{\pi}{4}],\quad\varphi_2(\theta)=-(3+2\rho)+\frac{2}{\tan(2\theta)}.$$
The conclusion of the first assertion follows.

\smallskip

\noindent (ii) {By Proposition \ref{prodynkin} and what precedes, for any probability measure ${m_{\cal S}}$ on $\ER^2$ such that ${m_{\cal S}}({\cal S})=1$,
\begin{equation}\label{fopedr}
\limsup_{t\rightarrow+\infty}-\frac{1}{t}\log\left(\PE_{m_{\cal S}}(\tau_{\cal S}\ge t)\right)\le 3+2\rho.
\end{equation}
Now, consider the general case. Let ${m_0}$ be a probability such that ${m_0}(H)=1$. 
Then, for every $t>0$, for every $a.s.$ finite stopping time $T$,
$$\PE_{m_0}(\tau_H\ge t)\ge \PE_{m_0}(\tau_H\ge T+t)\ge  \PE_{m_0}( \tau_H> T, Z_s\in {\cal S} \;\forall s\in[T,T+t] ).$$
Thus,
$$\PE_{{m_0}}(\tau_H\ge t)\ge \ES_{m_0}\left[1_{\{\tau_H> T,Z_T\in {\cal S}\}}\PE(Z_{s+T}\in {\cal S},\forall s\in[0,t]|{\cal F}_T)\right].$$
and  it follows from the Markov property that
$$\PE_{m_0}(\tau_H\ge t)\ge \ES_{m_0} [1_{\{\tau_H> T,Z_T\in {\cal S}}\PE_{Z_T}(\tau_{\cal S}\ge t)].$$
If we assume for a moment that $T$ is such that
\begin{equation}\label{eq:mupos}
\PE_{m_0}(\tau_H> T,Z_T\in {\cal S})>0,
\end{equation}
then,
$$-\frac{1}{t}\log(\PE_{m_0}(\tau_H\ge t))\le -\frac{1}{t}\log(\PE_{m_0}(\tau_H> T,Z_T\in {\cal S}))-
 \frac{1}{t}\log\left(\PE_{{m_{\cal S}}}(\tau_{\cal S}\ge t)\right),$$
where ${m_{\cal S}}$ is the probability measure defined for every bounded measurable function $h:\ER^2\rightarrow\ER$ by
 $${m_{\cal S}}(h)=\frac{1}{\PE_{m_0}(\tau_H> T,Z_T\in {\cal S})}\ES_{m_0}[h(Z_T)1_{\{\tau_H> T,Z_T\in {\cal S}\}}].$$
By \eqref{fopedr} and the (strict) positivity of $\PE_{m_0}(\tau_H> T,Z_T\in {\cal S})$, we obtain that
$$\limsup_{t\rightarrow+\infty}-\frac{1}{t}\log(\PE_{m_0}(\tau_H\ge t))\le 3+2\rho.$$
Thus, it remains to prove \eqref{eq:mupos}. It is certainly enough to show that for every $(x_0,y_0)\in H$, there exists a deterministic positive $T(x,y)$ such
that
\begin{equation*}
\PE_{(x_0,y_0)}(\tau_H> T(x_0,y_0),Z_{T(x_0,y_0)}\in {\cal S})>0.
\end{equation*}
  
\smallskip
The idea is to build some ``good'' controlled trajectories: let $\varphi\in L^{2,{\rm loc}}(\ER_+,\ER)$ and denote by 
$(z_\varphi(t))_{t\ge0}$ the solution of the controlled system
\begin{equation*}
\begin{cases} \dot{x}(t)=-\rho x(t)-y(t)\\
\dot{y}(t)=-\rho y(t)+x(t)+\varphi(t)
\end{cases}
\end{equation*}
starting from $z_0=(x_0,y_0)\in{ H}$. The classical Support Theorem (see \cite{StroockVaradhan_support}) can be applied since the coefficients of the diffusion are Lipschitz continuous. This implies that \eqref{eq:mupos} is true as soon as there exists such a $\varphi$ for which the solution $(z_\varphi(t))_{t\in[0,T(x_0,y_0)]}$ belongs to $H$ and such that $z_\varphi(T(x_0,y_0))$ belongs to ${\cal S}$. Such a controlled trajectory can be built through the following lemma.}

 \wwtbp
\end{proof}

\begin{lem}\label{lem:trajcontrol} Let $\kappa\in(0,+\infty]$ and set $H_\kappa=\{(x,y),y<\kappa x\}$ and $H_\infty=D$ ($=\{(x,y),x>0\}$).

\smallskip 
(i) Let $(x_0,y_0)\in H_\kappa$ with $y_0\ge 0$. Then, for every $v\in(-\infty,y_0]$, there exists a controlled trajectory $(x_\varphi(t),y_\varphi(t))_{t\ge0}$ starting from $(x_0,y_0)$ and a positive $T_v$ such that $\{z_\varphi(t)\st t\ge0\}\subset H_\kappa \cap H_\infty$,   $x_\varphi(T_v)>0$ and $y_\varphi(T_v)=v$.

\smallskip
(ii) Let $(x_0,y_0)\in H_\kappa$ with $y_0\le 0$ and consider $(x(t),y(t))_{t\ge0}$ the solution to the free dynamical system ($i.e.$ the controlled trajectory with $\varphi\equiv0$) starting from $(x_0,y_0)$.
Then, there exists $T>0$ such that $(x(t),y(t))_{t\in[0,T]}\subset H_\kappa$ and such that $(x(T),y(T))=(a_T,0)$ with $a_T>0$. Furthermore,
writing $(x_0,y_0)=(r_0\cos(-\theta_0),r_0\sin (-\theta_0))$ (with $r_0>0$ and $\theta_0\in(\pi-{\rm{Arctan}}(\kappa),0]$), this property holds with $T=\theta_0$ and $a_T=r_0e^{-\rho \theta_0}$.
 
\end{lem}
\begin{rem} Note that this lemma will be also used in the proof of Proposition \ref{lowerboundprob} (see Step 3). This is the reason why its statements are a little sharper than what we need for the proof of the previous proposition.
\end{rem}
\begin{proof}
(i) Without loss of generality, we only prove the result when $\kappa<+\infty$. The idea is to build $\varphi$ such that the derivative of the second component is large enough. More precisely, for every $M>0$, 
\begin{equation*}
\begin{cases}
\dot{x}_M(t)=-\rho x_M(t)-y_M(t)\\
\dot{y}_M(t)=-M
\end{cases}
\end{equation*}
is certainly an equation of a controlled trajectory (by setting $\varphi(t)=-M+\rho y_M(t)+x_M(t)$). Furthermore, denoting by $z_0=(x_0,y_0)$ its starting point, we have
$$ y_M(t)=-Mt+y_0\quad\textnormal{and}\quad x_M(t)=\left(x_0+\frac{M}{\rho^2}+\frac{y_0}{\rho}\right) e^{-\rho t}+\frac{M}{\rho}t-\frac{M}{\rho^2}-\frac{y_0}{\rho}.$$
First, let us choose $M$ large enough in order that for all $t\ge0$, $x_M(t)>0$  and $(x_M(t),y_M(t))\in H_\kappa$, $i.e.$ such that 
$x_M(t)>0$ and $\kappa x_M(t)-y_M(t)\ge0$ for all $t\ge0$. A simple study of the derivative of $t\rightarrow x_M(t)$ yields
\begin{align*}
&\forall t\ge 0,\quad x_M(t)\ge x_M(t_M^*)\quad\textnormal{with}\quad t_M^*=\frac{1}{\rho}\log\lt(1+ \frac{\rho}{M}(y_0+{\rho x_0})\rt)\quad\textnormal{and}\\
& x_M(t_M^*)=\frac{M}{\rho^2}\log\lt(1+\frac{\rho}{M}(y_0+{\rho x_0})\rt)-\frac{y_0}{\rho}\xrightarrow{M\rightarrow+\infty} x_0.
 \end{align*}
Thus, for every $\varepsilon>0$, there exists $M_\varepsilon$ large enough such that $x(t^*_{M_\varepsilon})\ge \kappa x_0-\varepsilon$. Using that for any  $M>0$ and $t\ge0$, $y_M(t)\le y_0$ and setting $\varepsilon=\frac{\kappa x_0 -y_0}{2}$, we obtain that 
 $$\forall t\ge0,\quad x_{M_\varepsilon}(t)>0\quad\textnormal{and}\quad \kappa x_{M_\varepsilon}(t)-y_{M_\varepsilon}(t) >0.$$
Since $y_{M_\varepsilon}$ is a continuous function such that $y_{M_\varepsilon}(t)\rightarrow-\infty$ as $t\rightarrow+\infty$, it  follows that for every $v\in(-\infty,y_0]$, there exists $T_v>0$ such that $y_{M_\varepsilon}(T_v)=v$.

\noindent
\smallskip
(ii)  The result is obvious since the solution to the free dynamical system satisfies 
$$(x(t),y(t))=r_0e^{-\rho t}(\cos(t-\theta_0),\sin(t-\theta_0)),\quad t\ge0.$$
%
 \wwtbp
\end{proof}

\smallskip
We are now able to prove the lower-bound of Theorem \ref{T1}.
\begin{pro}\label{pro:low2} Let $(Z_t)_{t\ge0}$ be a solution of \eqref{eds2} with {$(1+\frac{1}{\sqrt{2}})a\le b$} and let $\tau=\inf\{t>0, X_t=0\}$. Then, for every probability measure $m_0$ on $\ER^2$ such that $m_0(\{(x,y),x>0\})=1$,
$$\limsup_{t\rightarrow+\infty}-\frac{1}{t}\log(\PE_{m_0}(\tau\ge t))\le \lt(3 +\frac{a}{\omega}\rt)\omega.$$
\end{pro}
\begin{rem}Since $\frac{a}{\omega}=(\frac{b}{a}-\frac{1}{4})^{-\frac{1}{2}}$, 
$$\sup_{(a,b),0<(1+\frac{1}{\sqrt{2}})a\le b} \lt(3 +\frac{a}{\omega}\rt)= 3+(\frac{3}{4}+\frac{1}{\sqrt{2}})^{-\frac{1}{2}}\le 4.$$
This corresponds to the bound given in Theorem \ref{T1}. However, the reader can remark that the above result yields some sharper bounds. In particular,  when $a$ tends to $0$, $3+a/\omega$ tends to $3$. 
\end{rem}
\begin{proof} 
Let $z_0=(x_0,y_0)\in\ER^2$ such that $x_0>0$. Owing to the symmetry of the Brownian motion, one can check that 
$$\PE_{z_0}(\tau\ge t)=\PE_{-z_0}(\tau_{{D}_{-}}\ge t)$$
where $z_0=(x_0,y_0)^*$, ${D}_{-}=\{(x,y),x<0\}$ and $\tau_{{D}_{-}}=\inf\{t\ge0, Z_t\in {D}_{-}^c\}$.\\
Second, set  $v=(\frac{1}{b^2}(\frac{a}{2}-b),1)^*$ and $P_v=(v,B v)$ with $B=\frac{1}{\omega}(A+\frac{a}{2} I_2)$.
By Corollary \ref{cor:trmod}, there exists $\alpha>0$ such that $(\tilde{Z}_t)_{t\ge0}:=({\sqrt{\omega}\alpha} P_v^{-1}Z_{\frac{t}{\omega}})_{t\ge0}$ is a solution of 
\eqref{eq:hypocoercif}.  Denote respectively by $(x,y)$ and by $(\tilde{x},\tilde{y})$, the coordinates in the canonical basis and in the basis $\tilde{\cal B}=(v,B v)$. Computing $P_v (\tilde{x},\tilde{y})^*$, one checks that in the new basis, the set $D_{-}$ corresponds to the half-plane $H_\kappa$ defined by
$$H_\kappa=\{(\tilde{x},\tilde{y}), \tilde{y}<\kappa \tilde{x}\}\quad\textnormal{with}\quad \kappa=\frac{1}{\omega}\lt(b-\frac{a}{2}\rt).$$
Furthermore, from the very definition of $(\tilde{Z}_t)_{t\ge0}$, we have $$\tau_{D_{-}}^{-z_0}=\tilde{\tau}_{H_\kappa}^{-\tilde{z_0}}\quad\textnormal{with $\tilde{\tau}_{H_\kappa}=\inf\{t\ge 0, \tilde{Z}_t\in H_\kappa^c\}$ and $\tilde{z}_0={\sqrt{\omega}\alpha} P_v^{-1}z_0$.}$$
In particular, $\PE_{-z_0}(\tau_{{D}_{-}}\ge t)=\PE_{-\tilde{z}_0}(\tilde{\tau}_{H_\kappa}\ge \omega t)$ so that for any probability $m_0$ on $\ER^2$ such that $m_0(\{(x,y),x>0\})=1$,
 $$\PE_{m_0}(\tau_{{D}_{-}}\ge t)=\PE_{\tilde{m}_0}(\tilde{\tau}_{H_\kappa}\ge \omega t)$$
where $\tilde{m}_0:=m_0\circ (z\mapsto -{\sqrt{\omega}\alpha} P_v^{-1}z) $ satisfies $\tilde{m}_0 (H_\kappa)=1$.
\noindent Now, when  $(1+\frac{1}{\sqrt{2}})a\le b$, one checks that $\kappa\ge 1$ so that $H_\kappa$ contains the set   ${\cal S}=\{((\tilde{x},\tilde{y}), \tilde{x}>0,\tilde{y}<\tilde{x}\}$ of Proposition \ref{prophyplow} (written in polar coordinates). Applying the second item of this proposition  with $\rho=a/(2\omega)$, we finally obtain 
$$\limsup_{t\rightarrow+\infty}-\frac{1}{t}\log(\PE_{m_0}(\tau\ge t))=\omega \limsup_{t\rightarrow+\infty}-\frac{1}{\omega t}\log(\PE_{\tilde{m}_0}(\tilde{\tau}_{H_\kappa}\ge \omega t))\le \omega\lt(3+\frac{a}{\omega}\rt).$$
\wwtbp
\end{proof}
%

\section{Bridges at small times and persistence rate}\label{Eobast}
In this section, we study  the diffusion bridge associated to our dynamical system. We then use it to establish  some lower-bounds for $\PE_{(x_0,y_0)}(\tau\ge t)$ (where $\tau:=\inf\{t\ge0,Z_t \in\{(x,y),x<0\}\}$).
\subsection{Explosion of bridges at small times}  
Our objective here is to prove Theorem \ref{T2} and to discuss some related results.\par\me
Since we are mainly to  take advantage of the Gaussian features of the problem, we could have worked directly with
 the process $Z$ whose evolution
is given by (\ref{eds2}). 
Nevertheless the
computations presented in Subsection \ref{Gc} suggest that it is more advisable to first consider the simplifications made in Subsection \ref{sitvoT}. So we begin by considering the two-dimensional Ornstein-Uhlenbeck process $(Z_t)_{t\geq 0}\df (X_t,Y_t)_{t\geq 0}$
whose evolution is dictated by 
\bqn{catourne}
\lt\{
\begin{array}{rcl}
d X_t&=&(-\rho X_t-Y_t) \,dt\\
dY_t&= &(-\rho Y_t+X_t)\,dt+\sqrt{2}dW_t
\end{array}\rt.
\eqn
where $\rho\in\RR$ and $(W_t)_{t\geq 0}$ is a standard real Brownian motion.
Let us assume furthermore that the initial condition of $Z$ is a deterministic point $z_0=(x_0,y_0)^*\in\RR^2$.
The arguments of Subsection \ref{Gc} show that $Z$ is Gaussian and more precisely we have:
\begin{lem}\label{mtSt}
For any $t\geq 0$, $Z_t$ is distributed as a Gaussian law of mean $m_t(z_0)$
and variance $\Sigma_t$, with
\bq
m_t(z_0)&\df& 
\exp(-\rho t)\lt( \begin{array}{c}
x_0\cos(t)-y_0\sin(t)\\
x_0\sin(t)+ y_0\cos(t)\end{array}\rt)\\
\Sigma_t(1,1)& \df& \frac{1-e^{-2\rho t}}{2\rho}- \frac{e^{-2\rho t}}{2(1+\rho^2)}(\sin(2t)-\rho\cos(2t))-\frac{\rho}{2(1+\rho^2)} \\
\Sigma_t(1,2)\ =\ \Sigma_t(2,1)& \df&\frac{e^{-2\rho t}}{2(1+\rho^2)}(\cos(2t)+\rho\sin(2t))-\frac{1}{2(1+\rho^2)}\\
\Sigma_t(2,2)& \df& \frac{1-e^{-2\rho t}}{2\rho}+ \frac{e^{-2\rho t}}{2(1+\rho^2)}(\sin(2t)-\rho\cos(2t))+\frac{\rho}{2(1+\rho^2)}
\eq
\end{lem}
\proof
Let us denote
\bq 
A \df\ \lt( \begin{array}{cc}
-\rho&-1\\
1& -\rho \end{array}\rt)
&\hbox{ and }&
C\ \df\ \lt( \begin{array}{c}
0\\
\sqrt{2} \end{array}\rt)
\eq
From the beginning of  Subsection \ref{Gc}, we get that for any $t\geq 0$, on one hand
\bq
m_t(z_0)&=&\exp(At)z_0\\
&=&\exp(-\rho t)\lt( \begin{array}{cc}
\cos(t)&-\sin(t)\\
\sin(t)& \cos(t) \end{array}\rt)\lt( \begin{array}{c}
x_0\\
y_0 \end{array}\rt)\eq
and on the other hand,
the validity of (\ref{Sigmat}).
We compute that for any $s\geq 0$,
\bq
\exp(As)CC^*\exp(A^*s)&=&2
\exp(-2\rho s)\lt( \begin{array}{cc}
\cos(s)&-\sin(s)\\
\sin(s)& \cos(s) \end{array}\rt)
\lt( \begin{array}{cc}
0&0\\
0&1 \end{array}\rt)
\lt( \begin{array}{cc}
\cos(s)&\sin(s)\\
-\sin(s)& \cos(s) \end{array}\rt)\\
&=&2\exp(-2\rho s)\lt( \begin{array}{cc}
\sin^2(s)&-\cos(s)\sin(s)\\
-\cos(s)\sin(s)& \cos^2(s) \end{array}\rt)\\
&=&
\exp(-2\rho s)\lt( \begin{array}{cc}
1-\cos(2s)&-\sin(2s)\\
-\sin(2s)& 1+\cos(2s) \end{array}\rt)\eq
The announced expressions for the entries of $\Sigma_t$ follow from
 immediate integrations. For instance for $\Sigma_t(1,1)$, we have
\bq
 \Sigma_t(1,1)&=&\int_0^t\exp(-2\rho s)(1-\cos(2s))\,ds\\
 &=&\frac{1-\exp(-2\rho t)}{2\rho}-\Re\lt(\int_0^t\exp(2(i-\rho)s)\,ds\rt)\\
 &=&\frac{1-\exp(-2\rho t)}{2\rho}-\Re\lt(\frac{\exp(2(i-\rho ) t)-1}{2(i-\rho)}\rt)\\
  &=&\frac{1-\exp(-2\rho t)}{2\rho}+\frac{1}{2(1+\rho^2)}\Re\lt((\rho+i)(\exp(2(i-\rho ) t)-1)\rt)\\
&=&\frac{1-\exp(-2\rho t)}{2\rho}+\frac{1}{2(1+\rho^2)}(\exp(-2\rho t)(\rho\cos(2t)-\sin(2t))-\rho)\eq
\wwtbp
For $t>0$, let us denote by $p_t(z_0,z)\, dz$ the law of $Z_t$ knowing that $Z_0=z_0$. With the notations of the above lemma we have
\bq
\fo t> 0,\,\fo z_0,z\in\RR^2,\qquad p_t(z_0,z)&=&\frac{1}{2\pi\det(\Sigma_t)}
\exp(-(z-m_t(z_0))^*(2\Sigma_t)^{-1}(z-m_t(z_0)))\eq
Using Bayes' formula, we get that for $0<t <T$ and $z_0, z_T\in\RR^2$, the law of $Z_t$ conditioned by $Z_T=z_T$ (and still by $Z_0=z_0$)
admits a density proportional to $z\mapsto p_t(z_0,z)p_{T-t}(z,z_T)$ where $z\in \mathbb{R}^2$.
It is a non-degenerate Gaussian law, let $\eta^{(T)}_t(z_0,z_T)$ (resp. $\sigma^{(T)}_t$) be its mean vector (resp. its covariance matrix), formula (\ref{sigmau}) below will show that the covariance matrix does not depend on $z_0$ and $z_T$.
We furthermore define 
\bq
\fo u\in[0,1],\qquad
\varphi_{z_0,z_T}(u)&\df&\lt( \begin{array}{c}
0\\
6u(1-u)(x_0-x_T) \end{array}\rt)
\eq
where $z_0\df (x_0,y_0)^*$ and $z_T\df (x_T,y_T)^*$.
The next result contains all the required technicalities we will need.
\begin{pro}\label{uuu}
For all $z_0,z_T\in\RR^2$ and $u\in(0,1)$, we have
\bq
\lim_{T\ri 0_+}T \eta^{(T)}_{uT}(z_0,z_T)&=&\varphi_{z_0,z_T}(u)\\
\lim_{T\ri 0_+} \sigma^{(T)}_{uT}(z_0,z_T)&=& 0\eq
\end{pro}
\proof
To simplify notations, for $u\in(0,1)$, we denote 
$v\df 1-u$, $\eta_u\df \eta^{(T)}_{uT}(z_0,z_T)$
and $\sigma_u\df \sigma^{(T)}_{uT}(z_0,z_T)$.
With the notations of Lemma \ref{mtSt}, the vector $\eta_u$ and the matrix $\sigma_u$
are such that for any $z\in\RR^2$,
\begin{align*}
(z-\eta_u)^*\sigma_u^{-1}(z-\eta_u)
=&(z-m_{uT}(z_0))^*\Sigma_{uT}^{-1}(z-m_{uT}(z_0))
+(z_T-m_{vT}(z))^*\Sigma_{vT}^{-1}(z_T-m_{vT}(z))\\
&+C(z_0,z_T),\end{align*}
where $C(z_0,z_T)$ is a normalizing term which is independent of $z$.
It follows that
\bqn{etau}
\eta_u&=& \sigma_u(e^{-\rho uT}S_uB_{u}z_0+e^{-\rho vT}B_{v}^*S_vz_T)\\
\label{sigmau}\sigma_u&=&\lt(S_u+e^{-2\rho vT}B_{v}^*S_vB_{v}\rt)^{-1}\eqn
where for any $ w\geq 0$,
\bq
B_w&\df&
\lt( \begin{array}{cc}
\cos(wT)&-\sin(wT)\\
\sin(wT)& \cos(wT) \end{array}\rt)\\
S_w&\df& \Sigma_{wT}^{-1}\\
&=& \frac{1}{D_w}\lt( \begin{array}{cc}
\Sigma_{wT}(2,2)&-\Sigma_{wT}(1,2)\\
-\Sigma_{wT}(1,2)& \Sigma_{wT}(1,1) \end{array}\rt)\\
D_{w}&\df& \det( \Sigma_{wT})\\
&=& \Sigma_{wT}(1,1) \Sigma_{wT}(2,2)-( \Sigma_{wT}(1,2))^2
\eq
All these expressions depend on $T>0$ and the announced convergences will be obtained by expanding
them for small $T>0$. Indeed, simple computations show that for $w\in(0,1)$, as $T\ri 0_+$,
\bq
\lt( \begin{array}{cc}
\Sigma_{wT}(1,1)&\Sigma_{wT}(1,2)\\
\Sigma_{wT}(1,2)& \Sigma_{wT}(2,2) \end{array}\rt)
&=& \lt( \begin{array}{cc}
\frac{2(wT)^3}{3}+\cO((wT)^4)&-(wT)^2+\cO((wT)^3)\\
-(wT)^2+\cO((wT)^3)&2wT+\cO((wT)^2) \end{array}\rt)
\eq
where $\cO((wT)^p)$, for $p\in\RR$, stands for a quantity bounded by $A(wT)^p$, uniformly over  $\rho\in[-1,1]$
and for $wT$ small enough.
It follows that
\bq
D_{w}&=&\frac{(wT)^4}{3}+\cO((wT)^4)\\
S_w&=&  \lt( \begin{array}{cc}
\frac{6}{(wT)^3}+\cO((wT)^{-2})&\frac{3}{(wT)^2}+\cO((wT)^{-1})\\
\frac{3}{(wT)^2}+\cO((wT)^{-1})&\frac{2}{wT}+\cO(1)\end{array}\rt)
\eq
Using furthermore that for $v\in(0,1)$, we have $e^{-2\rho v T}= 1+\cO(vT)$ and that
\bq
B_v&=& \lt( \begin{array}{cc}
1+\cO((vT)^{2})&-vT+\cO((vT)^{3})\\
vT+\cO((vT)^{3})&1+\cO((vT)^{2})\end{array}\rt)
\eq
we deduce that
\bq
e^{-2\rho vT}B_{v}^*S_vB_{v}&=& \lt( \begin{array}{cc}
\frac{6}{(vT)^3}+\cO((vT)^{-2})&-\frac{3}{(vT)^2}+\cO((vT)^{-1})\\
-\frac{3}{(vT)^2}+\cO((vT)^{-1})&\frac{2}{vT}+\cO(1)\end{array}\rt)
\eq
Thus from (\ref{sigmau}), we get that for $u\in(0,1)$,
\bq
\sigma_u&=&\frac1{d(u,v)}
 \lt( \begin{array}{cc}
2(u+v)(uv)^3T^3+\cO(T^{4})&3(u^2-v^2)(uv)^2T^2+\cO(T^3)\\
3(u^2-v^2)(uv)^2T^2+\cO(T^3)&6(u^3+v^3)uvT+\cO(T^2)\end{array}\rt)
\eq
with 
\bq
d(u,v)&=&12(u^3+v^3)(u+v)-9(u^2-v^2)^2+\cO(T^{-1})\eq
Recalling that $v=1-u$, it appears that $d(u,v)=12+\cO(T^{-1})$, so
we obtain
\bqn{equivsiu}
\sigma_u&=&
 \lt( \begin{array}{cc}
\frac23(uv)^3T^3+\cO(T^{4})&(u-v)(uv)^2T^2+\cO(T^3)\\
(u-v)(uv)^2T^2+\cO(T^3)&2(u^3+v^3)uvT+\cO(T^2)\end{array}\rt)
\eqn
The second convergence announced in the proposition follows at once.
To deduce the first one, we begin by checking that for $u\in(0,1)$,
\bq
e^{-\rho uT}S_uB_{u}&=& \lt( \begin{array}{cc}
\frac{6}{(uT)^3}+\cO((uT)^{-2})&-\frac{3}{(uT)^2}+\cO((uT)^{-1})\\
\frac{3}{(uT)^2}+\cO((uT)^{-1})&-\frac{1}{uT}+\cO(1)\end{array}\rt)
\\
e^{-\rho vT}B_{v}^*S_v&=& \lt( \begin{array}{cc}
\frac{6}{(vT)^3}+\cO((vT)^{-2})&\frac{3}{(vT)^2}+\cO((vT)^{-1})\\
-\frac{3}{(vT)^2}+\cO((vT)^{-1})&-\frac{1}{vT}+\cO(1)\end{array}\rt)
\eq
In conjunction with (\ref{equivsiu}), we get
\bq
\sigma_ue^{-\rho uT}S_uB_{u}&=&  \lt( \begin{array}{cc}
1-3u^2+2u^3+\cO(T)&-u(1-u)^2T+\cO(T^2)\\
\frac{6u(1-u)}{T}+\cO(1)&1-4u+3u^2+\cO(T)\end{array}\rt)\\
\sigma_u
e^{-\rho vT}B_{v}^*S_v&=&
\lt( \begin{array}{cc}
3u^2-2u^3+\cO(T)&u^2(1-u)T+\cO(T^2)\\
-\frac{6u(1-u)}{T}+\cO(1)&-2u+3u^2+\cO(T)\end{array}\rt)\eq
In these expression, the $(2,1)$-entries explode as $T\ri0_+$,
it explains the renormalisation by $T$ considered in the above proposition
for $\eta^{(T)}_{uT}(z_0,z_T)$ and resulting convergence.\wwtbp
\begin{rem}\label{pasnormal}
Note that when $x_0=x_T$ (namely if $z_0$ and $z_T$
are on the same vertical line), it is simpler for the underlying vertical Brownian motion to put them in relation. Hence, the second component of $\varphi_{z_0,z_T}$ is equal to $0$. In this case, we thus expect the second component of $(\eta^{(T)}_{uT})$ to be convergent when $T\rightarrow0$. Pushing further the previous developments yields
\begin{equation}\label{morpdl}
(\sigma_ue^{-\rho uT}S_uB_{u})_{2,1}=\frac{6u(1-u)}{T}-2\rho u (1-u)(2-u)+\cO(T)
\end{equation}
and 
\begin{equation}\label{morpdl2}
(\sigma_u
e^{-\rho vT}B_{v}^*S_v)_{2,1}=-\frac{6u(1-u)}{T}-2\rho u(1-u^2)+\cO(T).
\end{equation}
Combined with the computations of the end of the previous proof, we deduce that if $x_0=x_T$,
then no renormalisation is needed for the mean vector and we get
\bq
\lim_{T\ri 0_+} \eta^{(T)}_{uT}(z_0,z_T)&=&
\lt( \begin{array}{c}
x_0\\
(1-4u+3u^2)y_0-(2u-3u^2)y_T-6\rho u (1-u)x_0 \end{array}\rt)
\eq
In particular even in the case when $z_0=z_T$, the asymptotical bridge doesn't stay still (except if $y_0=0$),
since
\bq
\lim_{T\ri 0_+} \eta^{(T)}_{uT}(z_0,z_0)&=&
\lt( \begin{array}{c}
x_0\\
(1-6u+6u^2)y_0-6\rho u(1-u)x_0) \end{array}\rt).
\eq
\end{rem}
\par
Similarly to the notational conventions endorsed in the introduction,
for $T>0$ and $z,z'\in\RR^2$, let $\PP^{(T)}_{z,z'}$ be the law of the process $Z$ evolving according to (\ref{catourne}),
conditioned by the event $\{Z_0=z,\ Z_T=z'\}$ and consider the process
$\xi^{(T)}\df(\xi_u^{(T)})_{u\in[0,1]}$ defined by
\bq
\fo u\in[0,1],\qquad 
\xi_u^{(T)}&\df& TZ_{Tu}\eq
Under $\PP^{(T)}_{z,z'}$ this process is Gaussian and Proposition \ref{uuu} enables to see that for
 fixed $z,z'\in\RR^2$, as $T$ goes to $0_+$, $\xi^{(T)}$ converges in probability (under  $\PP^{(T)}_{z,z'}$) toward the deterministic trajectory 
$\varphi_{z,z'}$, with respect to the uniform norm on $\cC([0,1],\RR^2)$).
Indeed, $\lim_{T\ri 0_+} T^2\sigma^{(T)}_{uT}(z_0,z_T)=0$ would even have been sufficient for this behavior.
Using the linear space-time transformation described in Subsection \ref{sitvoT}, this result can be retranscripted under the form
of Theorem \ref{T2}.
\par\begin{rem}
Following Remark \ref{pasnormal}, if $z=(x,y)$ and $z'=(x',y')$ are such that
$x=x'$, then  
 the process
$\wi\xi^{(T)}\df(\wi\xi_u^{(T)})_{u\in[0,1]}$, defined by
\bq
\fo u\in[0,1],\qquad 
\wi\xi_u^{(T)}&\df& Z_{Tu}\eq
converges in probability (under  $\PP^{(T)}_{z,z'}$) toward the deterministic trajectory 
$\wi\varphi_{z,z'}$, with respect to the uniform norm on $\cC([0,1],\RR^2)$), where
\bq
\fo u\in[0,1],\qquad\wi\varphi_{z,z'}(u)&\df&
\lt( \begin{array}{c}
x\\
(1-4u+3u^2)y-(2u-3u^2)y'-6\rho u(1-u)x \end{array}\rt)
\eq
Using the linear space-time transformation described in Subsection \ref{sitvoT}, this result can also rewritten
in the original setting of the Introduction.
\end{rem}

\subsection{A probabilistic proof of a persistence rate upper-bound}

The previous developments on the diffusion bridge associated to \eqref{catourne} enable us to retrieve a lower-bound of
$\PE(\tau\ge t)$, for $\tau$ defined in (\ref{tau}).
\begin{pro}\label{lowerboundprob}
(i) Let $(Z_t)_{t\geq 0}$ be a solution of \eqref{catourne} with $\rho\ge0$. For $\kappa\geq 1$, let $H_\kappa=\{(x,y),y<\kappa x\}$. Then, for any positive $\rho_0$, 
there exists a constant $\tilde{\lambda}>0$ such that for any $\rho\in[0,\rho_0]$ satisfying $\kappa\geq 3\rho$ and any $z_0\in H_\kappa$, one can find a constant $C$ (which depends on $z_0$ as well as on the parameters $\rho_0$ and $\kappa$) such that
$$\PE_{z_0}(\tau_{H_\kappa}\ge t)\ge C \exp(-\tilde{\lambda} t),\quad t>0.$$
(ii) Let  $(Z_t)$ be a solution of  \eqref{eds2}. There
 exists $\tilde{\lambda}>0$ such that if $0<2a\leq b$,  we have for every $z_0\in D=\{(x,y),x>0\}$
$$\PE_{z_0}(\tau\ge t)\ge C \exp(-\tilde{\lambda} \omega t),\quad t>0,$$
where $\omega=\sqrt{ab-a^2/4}$ and $C$ is a constant which depends on $z_0$, $a$, $b$ and $c$.
\end{pro}
\begin{rem} 
It is possible to be more precise,  the same proof showing that for all $\varepsilon,\,\varepsilon'>0$, one can find a corresponding $\wi\lambda(\varepsilon,\varepsilon')>0$
such that (i) is satisfied if $\kappa\geq 3(1+\varepsilon)\rho/2$, $\kappa\geq \varepsilon'$ and if  $\wi \lambda$ is replaced by $\wi\lambda(\varepsilon,\varepsilon')$
(the constant $C$ has then also to depend on $\varepsilon$ and $\varepsilon'$).
It follows that in (ii) the condition $b\geq 2a$ can be replaced by $b\geq (5/4+\varepsilon)a$ with the price that $\wi\lambda$ (and $C$)
must depend on $\varepsilon>0$. But we believe that even these results are not optimal (e.g.\ one could hope for the condition
$b>(1/4+\varepsilon)a$ in (ii)), so we won't detail them.
Proposition \ref{lowerboundprob} is just an illustration of how the bridges could be used further.
\end{rem}
\begin{proof} (i) The proof is divided into three steps. In the first one, we show that  we can build a subset 
${\cal S}$ of $H_\kappa$
for which any  bridge associated to (\ref{catourne}), starting  and ending in ${\cal S}$ (at a time $T$ which will be chosen small) stays in $H_\kappa$ with a high probability. Then, in the second one, we use a Markov-type argument close to the one used in the proof of Proposition \ref{upper1erresult} to obtain the announced result when the starting point of $(Z_t)_{t\geq 0}$ is in ${\cal S}$. Finally,
we extend the result to any initial point in $ H_\kappa$.\\

\noindent \textbf{Step 1.} \underline{\textit{Lower-bound for $\inf_{z,z'\in{\cal S}}\PE_{z,z'}^{(T)}(\tau_{H_\kappa}>T)$ for a particular $T>0$.}} Let $z_0=(x_0,y_0)$ and $z_T=(x_T,y_T)$ belong to $\ER^2$. Denote respectively by $\eta_{uT}^1(z_0,z_T)$ and $\eta_{uT}^2(z_0,z_T)$ the first and second coordinate of
$\eta_{uT}^{(T)}(z_0,z_T)$. First, owing to Proposition \ref{uuu} (and to the more precise developments stated in its proof) and to \eqref{morpdl} and \eqref{morpdl2}, one checks  that 
\begin{equation}\label{eq:etau1}
\eta_{uT}^1(z_0,z_T)= \left(x_0+(x_T-x_0)(3u^2-2u^3)\right)+\gamma_1(z_0,z_T,T)
 \end{equation}
  and

 \begin{equation}\label{eq:etau2}
\eta_{uT}^2(z_0,z_T)= \frac{6u(1-u)(x_0-x_T)}{T}+2\rho u (1-u)\left((2-u)x_0+(1+u)x_T\right)+\gamma_2(z_0,z_T,T)
\end{equation}
where $\gamma_1$ and $\gamma_2$ satisfy: there exists $T_0>0$ and a positive constant $C$ such that for every $T\in(0,T_0]$, 
for any $z_0$ and $z_T$ and for any $\rho\in[0,\rho_0]$ (due to the uniformity of $\cO(T)$ in \eqref{morpdl} and \eqref{morpdl2}
with respect to $\rho$ in a compact set of $\RR_+$).
$$|\gamma_1(z_0,z_T,T)|  \le C(|z_0|+|z_T|) T\quad\textnormal{and}\quad |\gamma_2(z_0,z_T,T)|\le C(|y_0|+|y_T|+(|z_0|+|z_T|)T).$$
 Second, by Theorem V.5.3 of \cite{adler} (applied with $\alpha=1$ and $K=T_0$), there exists a universal constant $C$ such that 
$$\forall T\ge0,\;\forall h \ge1,\quad \PE_{z_0,z_T}^{(T)}\left(\sup_{u\in[0,T]}|Z_{uT}-\eta_{uT}^{(T)}(z_0,z_T)|> h \right)\le C h\exp\left(-\frac{h ^2}{2\bar{\sigma}_T}\right)$$
where $\bar{\sigma}_T=\sup_{u\in[0,1]}|(\sigma_{uT}^{(T)})_{1,1}|+\sup_{u\in[0,1]}|(\sigma_{uT}^{(T)})_{2,2}|.$ By Proposition \ref{uuu}, for every $u\in[0,1]$,
${\sigma}_{uT}^{(T)}(z_0,z_T)\rightarrow0$ as $T\rightarrow0$. Note that this convergence is 
uniform in $z_0$ and $z_T$ since 
the covariance matrix does not depend of them. By \eqref{equivsiu}, it appears that the convergence is also uniform in $u$ and a little sharper study of the dependence of $\sigma_{(uT)}^{(T)}$ in $\rho$ yields in fact that for every $\rho_0>0$,
$$\sup_{\rho\in[0,\rho_0],z_0,z_T\in\ER^2 }\bar{\sigma}_T\xrightarrow{T\rightarrow0}0.$$
Applying the previous inequality with $h=\sqrt{\sigma_T}$, we deduce that there exists $T_1\in(0, T_0]$ such that for every $T\in(0, T_1]$, for every $z_0,z_T\in\ER^2$ and every $\rho\in[0,\rho_0]$, 
$$\PE_{z_0,z_T}^{(T)}\left(\sup_{u\in[0,T]}|Z_{uT}-\eta_{uT}^{(T)}(z_0,z_T)|\le \sqrt{\sigma_T}\right)\ge \frac{1}{2}.$$
We shall now build a box ${\cal S}=[1,1+h_1]\times [-h_2,h_2]$ (where $h_1$ and $h_2$ are positive numbers) for which there exists a positive $T$ such that when $(z_0,z_T) \in {\cal{S}}^2$, the mean $\eta_{uT}^{(T)}(z_0,z_T)$ 
stays at a distance greater than $1$ of the boundary of $H_\kappa$. Checking that for a point $(x,y)$ of $\ER^2$, the distance from  $(x,y)$ to the boundary $\pa H_{\kappa}$ is equal to $|\kappa x -y|/\sqrt{\kappa^2+1}$, we thus need to find  $h_1$, $h_2$ and $T$
in order that 
\bqn{eq:cccd}
\nonumber  \inf_{z_0,z_T\in{\cal S}}\inf_{u\in[0,1]}{\kappa\eta_{uT}^1-\eta_{uT}^2}&>&{\sqrt{\kappa^2+1}}\sqrt{\sigma_T}\\
&\geq & \sqrt{2}\kappa \sqrt{\sigma_T}.
\eqn
Note that $T$, $h_1$ and $h_2$ will depend on $\rho_0$ but not on $\rho\in[0,\rho_0]$ and $\kappa\geq 1$ 
satisfying $\kappa\ge3\rho$.
Using \eqref{eq:etau1} and \eqref{eq:etau2}, we obtain that there exists $C>0$ such that  for every $\rho\in[0,\rho_0]$, for every $T\in(0, T_1]$ and every $(z_0,z_T)\in{\cal S}$
\begin{equation}\label{eq:cdcd}
\eta_{uT}^1\ge {1}-C(1+h_1+h_2)T\quad\textnormal{and}\quad \eta_{uT}^2\le \frac{3}{2T}h_1+\frac{3\rho}{2}+C h_2
+C(1+h_1+h_2)T.\end{equation}
Equation \eqref{eq:cdcd} shows that  \eqref{eq:cccd} is fulfilled as soon as 
\bq
\kappa-\frac{3\rho}{2}-{C}(h_1+h_2)(1+\kappa)T-\frac{3}{2T}h_1
-Ch_2&> &\sqrt{2}\kappa \sqrt{\sigma_T}.\eq
Taking into account that $\kappa\geq 1$ and $\kappa\geq 3\rho$,
the above inequality is satisfied if
\bqn{juste}
\frac12&> &2{C}(h_1+h_2)T+\frac{3}{2T}h_1 +Ch_2+ \sqrt{2} \sqrt{\sigma_T}.\eqn
relation which no longer depends on $\kappa$.
We can now set for instance $h_1\df T^2$ and $h_2\df T$ and choose $T\in(0,T_1]$ small enough so that
(\ref{juste}) is satisfied. As a consequence, 
the subset ${\cal S}=[1,1+h_1]\times[-h_2,h_2]$ of $H_\kappa$  is such that for every $\rho\in[0,\rho_0]$ and $\kappa\geq 1$ verifying $\kappa\ge 3\rho$, we have
\bqn{eq:nudez}
\inf_{z,z'\in{\cal S}}\PE_{z,z'}^{(T)}(\tau_{H_\kappa}>T)&\ge& \frac{1}{2}.
\eqn

\noindent \textbf{Step 2.} \underline{\textit{Lower-bound for $\PE_{z_0}(\tau_{H_\kappa}>t)$ when $z_0\in{\cal S}$}.} We consider a time $T>0$  and a subset ${\cal S}=[1,1+h_1]\times [-h_2,h_2]$ of $H_\kappa$ (depending only on $\rho_0$) for which \eqref{eq:nudez} holds. For every $\ell\ge1$, we have
\begin{align*}
\PE_{z_0}(\tau_{H_\kappa}>\ell T,&Z_{\ell T}\in{\cal S})\ge\\
&\PE_{z_0}(\tau_{H_\kappa}>\ell T,Z_{\ell T}\in{\cal S}|\tau_{H_\kappa}>(\ell -1)T,Z_{(\ell -1)T}\in{\cal S})\PE_{z_0}(\tau_{H_\kappa}>(\ell -1)T,Z_{(\ell -1)T}\in{\cal S}).\end{align*}
By the Markov property,
\bq
\PE_{z_0}(\tau_{H_\kappa}>\ell T,Z_{\ell T}\in{\cal S}|\tau_{H_\kappa}>(\ell -1)T,Z_{(\ell -1)T}\in{\cal S})
&=&
\int H(z)\, \mu_{z_0,(\ell-1)T}(dz)
\eq
where 
\bq
\fo z\in \cS,\qquad H(z)&\df&\PE_z(Z_T\in{\cal S})\int_{\cal S}\PE_{z,z'}^{(T)}(\tau_{H_\kappa}>T)\,\mu_{z,T}(dz').\eq
and $\mu_{z_0,(\ell-1)T}$ is the conditional law (under $\PP_{z_0}$) of $Z_{(\ell -1)T}$ on $\{\tau_{H_\kappa}>(\ell -1)T,Z_{(\ell -1)T}\in{\cal S}\}$.
Owing to \eqref{eq:nudez} and to the fact that the support of $\mu_{z_0,\ell -1)T}$ is included in $\cS$, we get
$$\PE(\tau_{H_\kappa}>\ell T,Z_{\ell T}\in{\cal S})\ge\varsigma\PE(\tau_{H_\kappa}>(\ell -1)T,Z_{(\ell -1)T}\in{\cal S})\quad\textnormal{with}\quad \varsigma:=\frac{1}{2}\inf_{z\in\cS}\PE_z(Z_T\in{\cal S}).$$
 Note that for $T>0$, the transition density $\RR^2\ni z'\mapsto f_{z,T}(z')\df\frac{dP_{z}(Z_T\in dz')}{dz'}$  is positive and continuous with respect to
 $(z,z')\in(\RR^2)^2$. It follows that the coefficient $\varsigma$ is  positive by compactness of $\cS$. Furthermore it is uniform over $\kappa,\rho$ satisfying the conditions of Proposition \ref{lowerboundprob} (but a priori $\varsigma$ depends on $\rho_0$, as $T$ and $\cS$ do).
Then, since for all $t>0$,
$$\PE_z(\tau _{H_\kappa}>t)\ge \PE_z(\tau _{H_\kappa}>k_T T,Z_{k_T T}\in{\cal S}),$$
where $k_T=\lfloor t/T\rfloor+1$, we deduce from an induction that for every $z\in{\cal S}$,
$$\PE_z(\tau _{H_\kappa}>t)\ge \varsigma^{k_T}\ge C\exp(-\tilde{\lambda}t),$$
where $\tilde{\lambda}=-\log(\varsigma)/{T}$ (depending only on  $\rho_0$).\\

\noindent \textbf{Step 3.} \underline{\textit{Lower-bound for $\PE_{z}(\tau_{H_\kappa}>t)$ when $z\in H_\kappa$}.} The idea of this step is identical to that of the proof of Proposition \ref{prophyplow}\textit{(ii)}. More precisely,  to extend the lower-bound  obtained above to any $z_0=(x_0,y_0)\in H_\kappa$ (up to a constant $C$ which depends to $z_0$), it is enough to build a controlled trajectory $(z_{\varphi}(t))_{t\ge0}$ such that 
$z_{\varphi}(0)=z_0,$ $z_{\varphi}(t_0)$ belongs to  ${\cal S}$ and such that $ z_{\varphi}(t)\in H_\kappa$ for every $t\in[0,t_0]$.

\smallskip
Owing to Lemma \ref{lem:trajcontrol}$(ii)$, it is enough to consider the case $y_0=0$. We treat successively the cases $x_0\le 1$ and $x_0\ge 1+h_1$.
If $x_0\le 1$, the idea is to join a point of a path of the free dynamical system which passes through $z_{\cal S}=(1+h_1/2,0)$. More precisely,  by Lemma \ref{lem:trajcontrol}$(ii)$, we know that the solution $(x(t),y(t))_{t\ge0}$ to the free dynamical system starting from $(0,-(1+h_1/2)e^{\frac{\rho\pi}{2}})$ passes through  $z_{\cal S}$ at time $\pi/2$ and that ${\cal C}=\{(x(t),y(t)),0<t<\pi/2\}$ is a curve included in $(0,+\infty)\times(-(1+h_1/2)e^{\frac{\rho\pi}{2}},0)$. It remains to join a point of ${\cal C}$ (without leaving out $H_\kappa$).  This can be done applying Lemma \ref{lem:trajcontrol}$(i)$ with $v=-(1+h_1/2)e^{\frac{\rho\pi}{2}}$.\\
Suppose now that $x_0\ge 1+h_1$. The construction is slightly different in the cases $\rho>0$ and $\rho=0$ (by the assumptions of Proposition 
\ref{lowerboundprob},
the 
situation $\rho<0$ is excluded). If $\rho>0$, we join  $z_{\cal S}$ by crossing the segment $[z_0,z_{\cal S}]$ with a controlled trajectory. Set $(x(t),y(t))_{t\ge0}=(x_0e^{-\rho t},0)_{t\ge0}$.  This trajectory can be viewed as the solution starting from $z_0$ of 
\begin{equation*}
\begin{cases}
\dot{x}(t)=-\rho x(t)-y(t)\\
\dot{y}(t)=-\rho y(t)+x(t)+\varphi(t)
\end{cases}
\end{equation*}
with $\varphi(t)=-x(t)$. Hence, this is a controlled trajectory which clearly crosses  $[z_0,z_{\cal S}]$ in a finite time.
Finally, if $\rho=0$, we join a point ${z}_1=(x_1,-h_2/2)$ using Lemma \ref{lem:trajcontrol}$(i)$. If $x_1\le 1$, we are reduced to the first considered case. So  we can assume that $x_1\ge 1+h_1$ and we join the point   $\tilde{z}_{\cal S}=(1+h_1/2,-h_2/2)$ by crossing the segment $[{z}_1,\tilde{z}_{\cal S}]$ with a controlled trajectory (more precisely, $(x_1-h_2t/2,-h_2/2)_{t\ge0}$ is a controlled trajectory). This ends the proof of $(i)$.

\smallskip
\textit{$(ii)$} By Corollary \ref{cor:trmod}, there exists $\alpha$ such that $(\hat{Z}_t)_{t\ge0}\df\lt({\sqrt{\omega}\alpha} P_v^{-1}Z_{\frac{t}{\omega}}\rt)_{t\ge0}$ is a solution of 
\eqref{eq:hypocoercif}. Remind that $\tau=\inf\{t>0, \Re(Z_t)<0\}$ and for a subset $C$ of $\ER^2$, set $\tau^{\hat{Z}}_C=\inf\{t>0,\hat{Z}_t\in C^c\}$. Similarly to the proof of Proposition \ref{pro:low2}, one checks that for every $z\in\{(u,v),u>0\}$,
$$\PE_z(\tau>t)=\PE_{\hat{z}}(\tau_{H_\kappa}^{\hat{Z}}>\omega t)$$
where $\omega=\sqrt{ab-a^2/4}$, $\kappa=\frac{1}{\omega}(b-\frac{a}{2})$, $H_\kappa=\{(u,v),v<\kappa u\}$ and $\hat{z}$ belongs to $H_\kappa$. The result  follows by applying the first part of this proposition with $\rho=a/(2\omega)$. 
Indeed, the assumption $\kappa\geq 3\rho$ amounts to $b\geq 2a$. For the other condition, $\kappa\geq 1$, taking into account that $\omega< \sqrt{ab}$,
it is sufficient that $b-a/2\geq \sqrt{ab}$, namely $\frac{b}a\geq \lt(\frac{1+\sqrt{3}}{2}\rt)^2$ and note that $\lt(\frac{1+\sqrt{3}}{2}\rt)^2\leq 2$.
\wwtbp
\end{proof}
\section{Simulations and statistical considerations}\label{sec:stat}

In this section, we shortly focus on some statistical problems related to the estimation of the real parameters which govern the trajectories  solutions of \eqref{eds}. We denote the unknown underlying parameters $(a^*,b^*,c^*)$ and aim at developing some statistical estimation methods of these parameters.
We are also interested in the average time for the process $(X_t)_{t \geq 0}$ to return to the equilibrium price.
When $c^*$ is small, it can be shown that such an average time is close to $T^*/2$
 (see the results of \cite{MR1652127} and \cite{MR590626} for small noise asymptotics of random dynamical system)
where $T^*$ is the period of the deterministic process  associated to the model:
\begin{equation}\label{eq:periode}
{T}^* : = \frac{2\pi}{\omega^*} \quad \textnormal{where}\quad {\omega^*:=\sqrt{a^*(b^*-\frac{a^*}{4})}}.
\end{equation}  
 With a slight abuse of language, 
 ${T}^*$ will be then called \textit{pseudo-period} of the process.


It is natural to wonder if it is efficient to first estimate $(a^*, b^*)$ and then, to plug these estimates $(\hat{a},\hat{b})$ in the analytical formula given above. 
We describe in the next paragraph how one can use the maximum likelihood estimator to approximate  $(a^*,b^*)$. Then,  a short simulation study exhibit rather different behaviours of the estimator of the pseudo-period $\hat{T}$ derived from the values $(\hat{a}_{ML},\hat{b}_{ML})$ plugged into \eqref{eq:periode}.
When $c$ is small, we compare such an estimation with a more natural estimator derived from hitting times of the level $X=0$ and show that in some cases, this last estimation can perform better for the recovery of $T^*$. In our short study, we will assume that the process starts from its equilibrium, that is $X_0=0$. This assumption slightly simplifies the MLE derived below.

\subsection{Maximum likelihood estimator for parameters $a^*$ and $b^*$}

\paragraph{Statistical settings}
In this paragraph, we first detail the computation of the MLE of the real parameters denoted $(a^*,b^*)$ when one observes the whole trajectory of the price $X$ between $0$ and $t$. This is of course an idealization and a dramatical simplification of the true statistical problem since in practical situations, we can only handle some values of $X$ in a discreted observation grid $(k \Delta)_{0 \leq k \leq t/\Delta}$. 

Even if the relative size of $\Delta$ compared to  the time length of observation $t$ is of first interest for some real statistical applications, we simplify this short study and consider only continuous observation times. We  leave this important question of the statistical balance between $\Delta$ and $t$ to a future work. In such a situation, it is easy to recover the parameter $c^*$ by considering the normalized quadratic variation of the trajectory $(X_s)_{0 \leq s \leq t}$. 
$$
c^*=\frac{\langle X \rangle_t}{t}
$$
Hence, in the sequel, we only consider  the problem of the estimation of $a^*$ and $b^*$ and we assume the knowledge of $c^*$ (we fix it to $1$  for sake of convenience).

\paragraph{Change of measure formula}
In order to estimate $(a^*,b^*)$, we can only handle the process $X$ since $Y$ depends on the unobserved parameter $b^*$ through the relation
$$
\fo t \geq 0, \qquad Y_t= b^* \int_{0}^t \exp (b^*(s-t)) dX_s - b^* X_t.
$$
For any choice of $(a,b) \in \RR_+^2$, we consider the two processes defined by
\begin{equation}\label{eq:Ytb}
\fo t\geq 0, \qquad 
Y_t^b = b \int_{0}^t \exp (b(s-t)) dX_s - b X_t,
\end{equation}
and 
$$
H_t^{a,b} = (b-a) X_t + Y_t^b.
$$
In Equation \eqref{eq:Ytb}, $Y^b$  depends on the increments $dX_s$. A simple integration by part yields the equivalent expression:
$$
    \fo s \geq 0, \qquad     Y_s^b=e^{- b t} \left( -b^2 \int_{0}^t e^{b s} X_s ds\right).
$$

The fact that $Y = Y^{b^*}$ follows from the definition of the process $Y$. Thus,  $(X_t)_{t \geq 0}$ satisfies
$$
dX_t=  H_t^{a^*,b^*} dt + dB_t.
$$
Now, we can apply the Girsanov formula: 
if we denote by $\PP_{a^*,b^*}$ the law of the process, we then obtain the change of measure formula:
$$
\frac{d\PP_{a^*,b^*}}{d\QQ_0}(X)= \exp \left(\int_{0}^t H_s^{a^*,b^*}(X_s) d X_s - \frac{1}{2} \int_0^t H_s^{a^*,b^*}(X_s)^2 ds \right).
$$
\paragraph{Maximum likelihood}
Given any trajectory $X$, we can then define $L_t$, the log-likelihood of the parameters $(a,b)$ as follows:
$$
L_t(a,b)=\int_{0}^t H_s^{a,b}(X_s) dX_s - \frac{1}{2} \int_{0}^t H_s^{a,b}(X_s)^2 ds.
$$
The expression above can be modified using an integration by part. We then obtain the  "robust" formulation:
\begin{eqnarray}\label{eq:mle}
L_t(a,b)&=&\frac{b-a}{2} \left[ X_t^2-t\right] + X_t Y_t^b  \nonumber\\ 
& & + \int_0^t b^2 X_s ^2 + b X_s Y_s^b  - \frac{1}{2} \left[ (b-a) X_s + Y_s^b\right]^2 ds.
\end{eqnarray}

The maximum likelihood  estimator is then formally defined by
$$
(\hat{a}^{ML}_t,\hat{b}^{ML}_t) := \arg\max_{(a,b) \in \RR_+^2} L_t(a,b).
$$
For any $b \geq 0$, $a \longmapsto L_t(a,b)$ is a concave function and  thus the optimal value of $a$ given any $b$ is
$$
a_b = b+ \frac{\int_{0}^t X_s Y_s^{ b} ds + \frac{t-X_t^2}{2}}{\int_{0}^t X_s^2 ds}.
$$
Hence,  $\hat{b}_t^{ML}$ is obtained by maximizing $b \longmapsto L_t(a_b,b)$. Unfortunately, we did not find any explicit formula regarding the relation $b \longmapsto Y^b$. Hence, to estimate $b^*$, we  use an exhaustive numerical search of the optimal value of $b$ and we obtain $\hat{b}_t^{ML}$.

\subsection{Estimation of the mean pseudo-period $T^*$ with hitting times strategy}

It may be possible to estimate $T^*$ using the maximum likelihood estimators $(\hat{a}_t^{ML},\hat{b}_t^{ML})$ defined above with and plug them in the relation \eqref{eq:periode} (which is supposed to be true only for $\mathcal{T}^*$ with a vanishing noise level)
$$
\hat{T}^{ML}_t := \frac{2 \pi}{\sqrt{\hat{a}^{ML}_t \left( \hat{b}^{ML}_t - \hat{a}^{ML}_t/4\right)}}.
$$
Of course, the ability of this estimator to well approximate $T^*$ highly depends on the asymptotic behaviour of $(\hat{a}^{ML}_t,\hat{b}^{ML}_t)$. We will discuss on several statistical questions related to this study in the next section.

We can also compare $\hat{T}^{ML}_t$ with a more nature way to estimate the mean return time all along the trajectories by  considering the sequence of crossing times of level $0$ of $(X_s)_{0 \leq s \leq t}$. In this view, let us consider $\epsilon>0$, and define a skeleton chain associated to the trajectory $(X_s)_{0 \leq s \leq t}$. The sequences $(\tau_k)_{k \geq 0}$ and  $(r_k)_{k \geq 0}$ are initialized with: 
$$
\tau_0 := 0 \qquad \text{and} \qquad r_0 := \inf \left\{s \geq \tau_0 \vert |X_s| \geq \epsilon \right\}
$$ and recursively built as follows:
$$
\forall k \geq 0 \qquad \tau_{k+1} := \inf \left\{s \geq r_k \vert X_s = 0 \right\} \qquad \text{and} \qquad r_{k+1} := \inf \left\{s \geq \tau_k \vert |X_s| \geq \epsilon \right\}.
$$ 
 When $c^*$ is small, we can now define a quite natural estimator of the mean pseudo-period using this construction.
If we set $N_t := \sup \left\{k \geq 0 \vert \tau_k \leq t\right\}$, we set
$$
\hat{T}^{\epsilon}_t := 2 \frac{\sum_{k = 1}^{N_t} (\tau_{k+1}-\tau_k)}{N_t} = 2 \frac{\tau_{N_t}- \tau_0}{N_t}.
$$

\subsection{Statistical performances and open problems}

To establish the numerical performances of $\hat{T}^{\epsilon}_t $  and $\hat{T}^{ML}_t$, we use the following statistical setting: 
several trajectories defined on $[0,T]$ are observed on some discrete times $(t_k)_{0 \leq k \leq K}$. The observation time are equally sampled with a constant step size $\Delta$ such that  $t_k - t_{k-1} = \Delta$.
We are interested in the behaviour of our two estimators in the two distinct asymptotic settings:
\begin{itemize}
\item  High resolution sampling scheme $\Delta \longrightarrow 0$ 
\item Long time observation $T \longrightarrow + \infty$.
\end{itemize}

For our purpose, the threshold is defined empirically after several runs of the estimator. It should be carefully chosen since $\epsilon$ is the parameter which enables to distinguish  real crossings of $X=0$ from the natural volatility of the model carried out by the brownian noise. Thus, the calibration of $\epsilon$ should be related to the noise level contained in $c^*$. In our simulation, we have chosen  $\epsilon=c^* \sqrt{\Delta}$.
We use for each estimator $N=10^{3}$ Monte-Carlo simulations to obtain the repartition of $\hat{T}^{\epsilon}_t $  and $\hat{T}^{ML}_t$ around $T^*$.

Moreover, we use a discretized version of the stochastic differential equation with several step size. We show in Figure \ref{fig:MLE_estimation} the performances of  $\hat{T}^{ML}_t$ for several size of discretization step $\Delta$ as well as the performances of $\hat{T}^{\epsilon}_t$  in Figure \ref{fig:traj_estimation}.

\begin{figure}[h]
 \centering
\includegraphics[height=5cm]{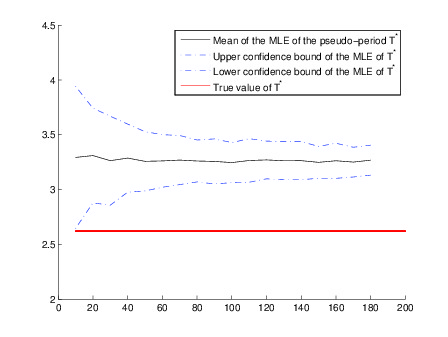}
\includegraphics[height=5cm]{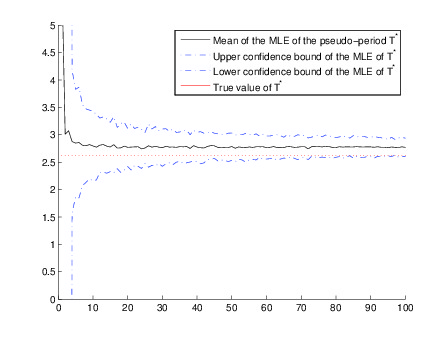}
\includegraphics[height=5cm]{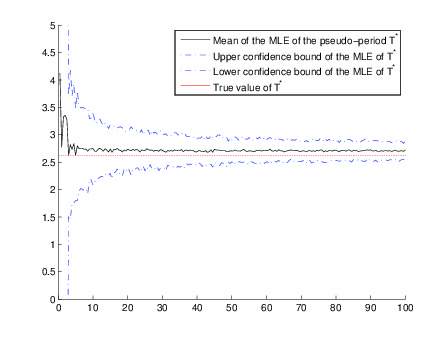}
\caption{\label{fig:MLE_estimation} Estimation of $T^*$ using $\hat{T}^{ML}_t$ with respect to the observation time $t$ for $a=1, b=6, c=1$. (top left: $\Delta=10^{-2 }, c=1$, top right: $\Delta=5 . 10^{-3}$, bottom left: $\Delta=10^{-3}$, bottom right: 
$\Delta=5 . 10^{-4}$).}
\end{figure}

\begin{figure}[h]
 \centering
\includegraphics[height=5cm]{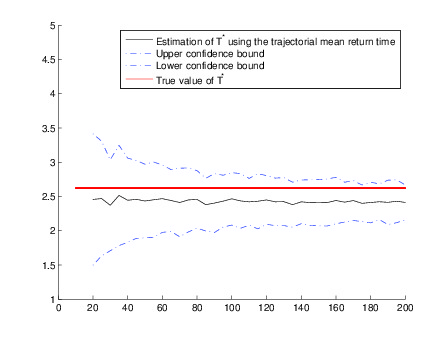}
\includegraphics[height=5cm]{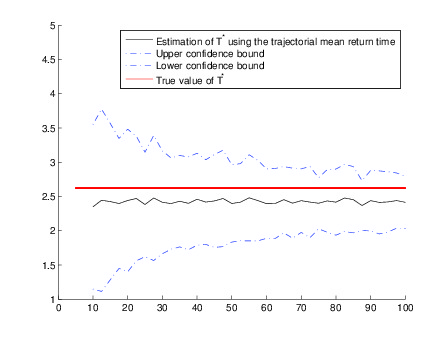}
\includegraphics[height=5cm]{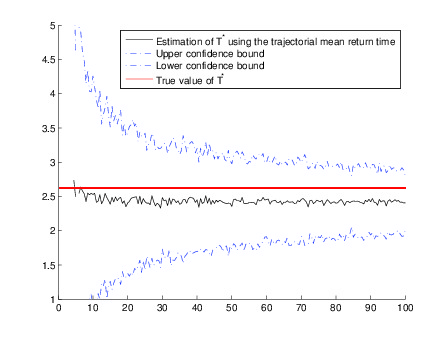}
\caption{\label{fig:traj_estimation} Estimation of $T^*$ using $\hat{T}^{\epsilon}_t$ with respect to the observation time $t$ for $a=1, b=6, c=1$. (top left: $\Delta=10^{-2 }, c=1$, top right: $\Delta=5 . 10^{-3}$, bottom left: $\Delta=10^{-3}$, bottom right: 
$\Delta=5 . 10^{-4}$).}
\end{figure}

One may instantaneously remark that the step size $\Delta$ has an important influence on the ability of $\hat{T}^{ML}_t$ to recover $T^*$ although this parameter does not seem so important for the estimator $\hat{T}^{\epsilon}_t$.
Simulations show that the smaller $\Delta$, the smaller the bias of $\hat{T}^{ML}_t$.
Moreover, the variance of the estimator is mainly determined by the length of simulation (time $t$). Hence, for a fixed step size of simulation, Figures \ref{fig:MLE_estimation} and \ref{fig:traj_estimation} demonstrate that it seems better to use $\hat{T}^{\epsilon}_t$ to infer $T^*$. When one may let $\Delta  \longmapsto 0$, the maximum likelihood estimator  $\hat{T}^{ML}_t$ seems more convenient.

We shortly describe several problems of interest which concern the estimation of $T^*$.
First, the influence of $\Delta$ as well as the influence of  $t$ needs to be understood for the estimation of $T^*$ using the MLE. This question may be faced with a careful understanding of the natural score function defined by the log-likelihood $L_t$. Our simulations tend to show that a special  asymptotic behaviour of ($\Delta,t) \longrightarrow (0,+\infty)$ should be considered to obtain optimal estimations.

Second, the size of the threshold $\epsilon$ in the definition of $\hat{T}^{\epsilon}_t$ has not been theoretically investigated, although it has a great influence on the ability of $\hat{T}^{\epsilon}_t$ to nicely recover $T^*$. Hence, there should also exist a precise asymptotic regime of $(\epsilon,t)\longrightarrow (0,+\infty)$ which may permit to obtain statistical reconstruction properties. Such a last result should be derived from a careful inspection of the local times spent by $(X_t)_{t \geq 0}$ around the level $0$ (to fix $\epsilon$) as well as the concentration rate of the hitting times which may be obtained using Theorem \ref{T1}.

At last, the link between $T^*$ and the expected time needed for  $(X_t)_{t \geq 0}$ to return to its equilibrium price is still mysterious. We only identify this link in the small noise asymptotics and even though such a relation seems to be true in more general situations for $T^*$, a theoretical proof is missing.

These three questions are far beyond the scope of this study, and we let them open for future works.

\section{Conclusion}

In this paper a model of speculative bubble evolution was proposed. 
The dynamics has to be at least of second order, to have a chance to display a weak periodic behavior typical of this kind of phenomena.
 This second order is induced by the way the process under consideration weights its past evolution to infer its future behavior (increase/decrease
 in the close past favoring an immediate tendency to follow the same trend). Dynamics of all orders (including non-integer ones) could be obtained in the same fashion,
 by modifying the weights. At the ``microscopic level", the latter are related  to the distribution of the backward time windows used by a multitude of agents 
 in order to speculate on the future evolution. \par
 But we restricted ourselves to  second order dynamics: it is the simplest one and in some sense it mimics Newtonian mecanics, which
 are also of second order, the forces impacting directly on the acceleration.
 A main difference is the noise entering our modeling, which is required to maintain some stability of the system under study.
 Nevertheless  and informally, this analogy with the physics law of motion enables to unmask some misleading arguments used by the
 real estate agencies and the mass media:  they mainly explain the evolution of prices by making an inventory of the forces in the housing market, such as loan interest rates, the growth of population, etc.\ (all these factors, as well as opposite leverages, are summed up in our parameter $a$), forgetting the order of the evolution equation (induced by the parameter $b$). Transposed in the astronomy field, it would amount
 to make the observation that the main interaction between the Sun and the Earth goes through gravitation, so that we should conclude
 that our planet would soon end up in the Sun. Luckily, we are essentially saved by the second order of the 
kinetics law which enables the Earth to turn around the Sun!

\par We have also assumed that the repelling force to the equilibrium level $X=0$ is linear and traduced in \eqref{eq:model} by the dritf term $- a X_t$ at any time $t$. This drift term mimics an economic repelling force which may be non-linear. Nevertheless, our linearization may be considered as a first order valid approximation at least near the equilibrium state $X=0$.

\par
From the mathematical point of view, our main interest was in the return time to the equilibrium ``price" and we have shown that it is
more concentrated than the relaxation to the equilibrium distribution of the prices. This feature explains the bubble/almost periodic aspect of the typical trajectories. 
We  have obtained some lower and upper bounds of this concentration rate in Theorem \ref{T1}. Even if there is still a gap of order around ten between our lower and upper bounds, numerous simulations (not shown in this paper) using the Fleming-Viot's type algorithm described in \cite{MR1956078} lead to the conjecture that $\lambda_0(D) = \frac{\log(2)}{\pi} \omega$.

\par

\begin{figure}[h]
 \centering
\includegraphics[height=7cm]{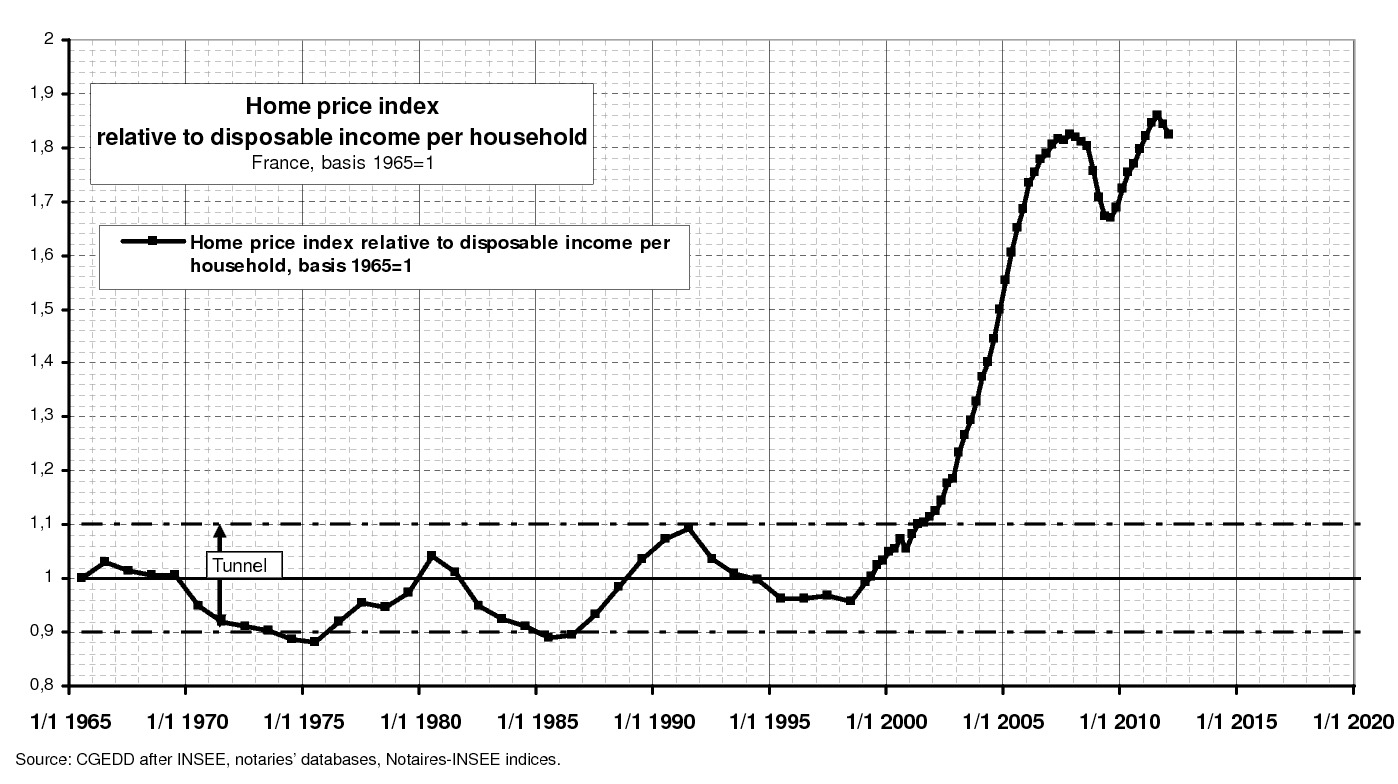}
\caption{Friggit's  curve \cite{Friggit} of the index of asset price relatively to the disposable income in France. \label{fig:friggit}}
\end{figure}

One failing of our framework is that the parameters $a$, $b$ and $c$ were assumed to be time-independent, hypothesis
which is certainly wrong in practice. 
Our model should apply only to a few periods and a finer modeling would take into account the time-inhomogeneity of $a$, $b$ and $c$.
Such an extension
remains Gaussian, but its
investigation is out of the scope of this paper.  \par
Nevertheless and heuristically, let us just consider the example of the
home price index
relative to disposable income per household in France from 1995 to 2013 shown in Figure \ref{fig:friggit}.
Consider for $X$ the logarithm of the quantity displayed in this picture, since it is a ratio
and not a difference as wanted in the introduction.
If we assume that the modeling of speculation presented in Section \ref{mos} can be applied to this case,
it seems that the coefficients $a$, $b$ and $c$ valid for the epoch 1965-1998 are not the same as those for 1998-2013.
For simplicity, let us make the hypothesis that $c$ did not change (the resolution is too coarse to check that)
and call $a_1$, $b_1$ and $a_2$, $b_2$  the respective values of $a$ and $b$.
That $a_2<a_1$ can be explained by the fact that the conditions were more favorable for buying in the second epoch, especially due
to low interest rates. Furthermore one can argue that the advent of Internet has probably modified the extend to
which the agents have access to the past evolution of the prices, both in short and long terms.
So its effect on $b$ is not a priori clear (recall that 
$1/b$ should be proportional to the
mean length of the backward time window). To get a rough idea, we can proceed as follows.
Let $t_0\df1965$, $t_1\df 1999$ and $t_2\df2008$.
Figure \ref{fig:friggit} suggests that between $t_0$ and $t_1$ there were 3 periods under the coefficients $a_1$, $b_1$ and $c$
and that between $t_1$ and $t_2$ there was one quarter of a period under the coefficients $a_2$, $b_2$ and $c$ (except if the equilibrium price has itself changed and that the evolution between 2006 and 2012 is interpreted as one period and half of a new epoch).
It follows that 
$1/\omega_1$ and $1/\omega_2$ should respectively be proportional to $(t_1-t_0)/3$ and $4(t_2-t_1)$.
Since we do not plan to be very precise, let us make the assumption that $a_1\ll b_1$ and that $a_2\ll b_2$,
so that
$\omega_1\approx \sqrt{a_1b_1}$, $\omega_2\approx \sqrt{a_2b_2}$  and 
\bq
a_1b_1&\approx&\chi a_2b_2\eq
where 
\bq
\chi&\df&\lt( \frac{12(t_2-t_1)}{t_1-t_0}\rt)^2\eq
To deduce another equation, let us believe that an ergodic theorem takes place very rapidly (permitted by the non reversibility of the process).
Thus according to Remark \ref{R0}, we would get
\bq
\frac1{t_1-t_0}\int_{t_0}^{t_1} X_s^2\,ds\ \approx\ c^2\frac{b_1+a_1}{2a_1^2}
\ \approx\ c^2\frac{b_1}{2a_1^2}\\
\frac1{t_2-t_1}\int_{t_1}^{t_2} X_s^2\,ds\ \approx\ c^2\frac{b_2+a_2}{2a_2^2}
\ \approx\ c^2\frac{b_2}{2a_2^2}\eq
(more carefully, an empirical variance should be computed), and it follows that
\bq
\frac{b_1}{a_1^2}&\approx& \wi\chi\frac{b_2}{a_2^2}\eq
where 
\bq
\wi\chi&\df&\frac{(t_2-t_1)\int_{t_0}^{t_1} X_s^2\,ds}{(t_1-t_0)\int_{t_1}^{t_2} X_s^2\,ds}
\eq
can be computed numerically on Figure \ref{fig:friggit}.
We deduce that
\bq
a_1&\approx& \chi^{5/3}\wi\chi^{1/3}a_2\\
b_1&\approx& \chi^{2/3}\wi\chi^{1/3}b_2\eq
Numerically, we obtain that $ \chi^{2/3}\wi\chi^{1/3}\approx 3.73 >1$, so it would seem that the advent of Internet has led people to rather use more recent trend of the housing market to make their speculation.

At last, using the lower bound obtained in Theorem \ref{T1}, we may postulate that the probability that the index price 
$X$ hits the equilibrium level $1$  (see  Figure \ref{fig:friggit}) before year 2017 (which corresponds to an average annual loss of around $13\%$) is at least $50\%$. We can thus wonder if the famous \textit{kiss landing} generally announced by estate agents may not more probably end in a crash \ldots


\appendix
\section{On the persistence rate}
Our goal here is to prove the existence of the quasi-stationary distribution and its persistence rate, as alluded to in Remark \ref{persistence}. It is based on general considerations relegated in this appendix because they do
 not lead to  explicit estimates such as (\ref{persistence2}), which are more important from a practical point of view than the mere existence of $\lambda_0(D)$. Furthermore, these a priori bounds will be useful in the development to follow.\par\me
Recall that $D\df\{ (x,y)\in\RR^2\st x> 0\}$ and let $\pa D$ be its boundary.
 We are interested in $L_D$, the realization on $D$ of  the differential operator $L$ given by (\ref{L}) with Dirichlet boundary condition
 on $\pa D$.
 From a probabilist point of view, it is constructed in the following way.
 For any $z\in\RR^2$, let $(Z^z_t)_{t\geq 0}$ be a diffusion process whose evolution is dictated by $L$ and whose initial condition is $Z_0^z=z$.
Starting from $z$, $(Z^z_t)_{t\geq 0}$ can be obtained by solving the stochastic differential equation (\ref{eds2}) with coefficients
given by (\ref{AC}). Let $\tau$ be the stopping time defined by (\ref{tau}), namely
\bq
\tau&\df& \inf\{t\geq 0\st Z_t^z\in \pa D\}\eq
For any $t\geq 0$, any $z\in D$ and any measurable and bounded function $f$ defined on $D$, consider
\bqn{PtD}
P^D_t[f](z)&\df& \EE[f(Z_t^z)\un_{t<\tau}]\eqn
Recall that $\mu$ is the invariant Gaussian probability measure of $L$ and denote by $\mu^D$ its restriction to $D$.
Then $P_t^D$ can be extended into a contraction operator on $\LL^2(\mu^D)$.
Indeed, let $P_t$ be the full operator associated to $L$:  any $z\in D$ and any measurable and bounded function $f$ defined on $\RR^2$,
we have  
\bqn{Pt} P_t[f](z)&\df& \EE[f(Z_t^z)]\eqn
Since $\mu$ is invariant for $P_t$, for any measurable and bounded function $f$ defined on $D$ (which can also be seen as a function on $\RR^2$ by assuming that it vanishes outside $D$), we get by Cauchy-Schwarz inequality,
\bq
\mu^D[(P_t^D[f])^2]&\leq & \mu^D[P_t^D[f^2 ]]\\
&\leq &  \mu^D[P_t[ f^2] ]\\
&\leq & \mu[P_t[ f^2] ]\\
&=&\mu[f^2]\\
&=&\mu^D[f^2]\eq
This bound enables to extend  $P_t^D$ as a contraction on $\LL^2(\mu^D)$.
The Markov property implies that $(P_t^D)_{t\geq 0}$ is a semi-group, which is easily seen to be continuous in $\LL^2(\mu^D)$.
The operator $L^D$ is then defined as the generator of this semi-group (in the Hille-Yoshida sense):
its domain $\cD(L_T^D)$ is the dense subspace  of $\LL^2(\mu^D)$ consisting of functions $f$  such that
$(P_t^D[f)-f)/t$ converges in $\LL^2(\mu^D)$ as $t$ goes to $0_+$ and the limit is $L^D[f]$ by definition.\par
The spectrum of $-L^D$ admits a smallest element (in modulus) $\lambda_0(D)$. It is a positive real number
and the main objective of this appendix is to 
justify the assertions made in Remark \ref{persistence}.
We begin by being more precise about the existence of $\lambda_0(D)$:
\begin{pro}\label{KR}
There exists a number $\lambda_0(D)>0$ and two functions $\varphi,\varphi^*\in\cD(L^D)\cap\bigcap_{r\geq 1} \LL^r(\mu^D)\setminus\{0\}$,
which are positive on $D$, such that
\bq
L^D[\varphi]&=&-\lambda_0(D)\varphi\\
L^{D*}[\varphi^*]&=&-\lambda_0(D)\varphi^*\eq
where $L^{D*}$ is  operator adjoint of $L^D$ in $\LL^2(\mu^D)$.
\end{pro}
Essentially, this result is a consequence of the Krein-Rutman theorem (which is an infinite version of the Perron-Frobenius theorem, see for instance the paper \cite{ MR2205529} of Du)
and the fact that the eigenfunctions belong to  $\LL^p(\mu^D)$ instead of $\LL^2(\mu^D)$ comes
from the hyperboundedness  of the underlying Dirichlet semi-group.
\par
The rigorous proof relies on a simple technical lemma about the kernels of the operators $P^D_t$ for $t>0$.
To check their existence, we first come back to $P_t$ for a given $t>0$:
from the computations of Section \ref{pas}, this operator is indeed given by a kernel
\bq
\fo z\in \RR^2,\, \fo f\in\LL^2(\mu),\qquad P_t[f](z)&=&\int p_t(z,z') f(z')\,\mu(dz')\eq
where
\bqn{pt}
\fo z,\,z'\in \RR^2,\quad p_t(z,z')&\df& \sqrt{\frac{\det(\Sigma)}{\det(\Sigma_t)}}\exp\lt(-(z'-z_t)^*\Sigma_t(z'-z_t)+(z')^*\Sigma z'\rt)\eqn
with
\bq
z_t&\df& \exp(At)z\eq
It follows easily from (\ref{PtD}) and (\ref{Pt}) that the same is true for $P_t^D$: there exists a function
$
D^2\ni(z,z')\mapsto  p^D_t(z,z')\geq 0$ such that
\bq
\fo z\in D,\, \fo f\in\LL^2(\mu^D),\qquad P_t^D[f](z)&=&\int p_t^D(z,z') f(z')\,\mu(dz')\eq
and satisfying 
\bqn{pDp}\fo z,z'\in D,\qquad p_t^D(z,z')&\leq & p_t(z,z')\eqn
More refined arguments based on the hypoellipticity of $L^D$  enable to see that the mapping $p_t^D$ is continuous and positive on $D^2$.
We can now state a simple but crucial observation:
\begin{lem}\label{hyperalamain}
For any $r>1$, there exists a time $T_r>0$ such that
\bq
\fo t\geq T_r,\qquad\int (p_t^D(z,z'))^r\, \mu^D(dz)\,\mu^D(dz')&<&+\iy\eq
\end{lem}
\proof
From (\ref{pDp}), it is sufficient to prove that
\bq
\int (p_t(z,z'))^r\, \mu(dz)\,\mu(dz')&<&+\iy\eq
and this can be obtained without difficulty from (\ref{pt}) and from the explicit computations of $\exp(tA)$
of $\Sigma_t$ and of $\Sigma$ presented in Section \ref{pas}.\wwtbp
We can now come to the
\prooff{Proof of Proposition \ref{KR}}
We begin by applying Lemma \ref{hyperalamain} with $r=2$ to find some $T_2>0$ such that
for $t\geq T_2$ we have
\bq
\int (p_t^D(z,z'))^2\, \mu^D(dz)\,\mu^D(dz')&<&+\iy\eq
which implies that $P^D_t$ is of Hilbert-Schmidt class and thus a compact operator.
Note furthermore that the spectral radius of $P_t^D$ is positive for all $t\geq 0$. 
Indeed, this feature can be deduced from the second bound of Theorem \ref{T1}, which implies that for all $z\in D$,
$P_t^D[\un_D](z)=\PP_z[\tau>t]>0$. Thus
we are in position to apply Krein-Rutman theorem
(see  Theorems 1.1 and 1.2 of Du \cite{ MR2205529}, where the abstract Banach $X$ space should be $\LL^2(\mu^D)$
and the cone $K$ should consist of the nonnegative elements of $\LL^2(\mu^D)$):
if $\theta_t>0$ is the spectrum radius of $P_t^D$, then there exists a positive function $\varphi_t\in\LL^2(\mu^D)\setminus\{0\}$ such that
$P_t[\varphi_t]=\theta_t\varphi_t$. This property characterizes $\theta_t$ and $\varphi_t$ (up to a constant factor): if $\theta$ is a positive real and if $\varphi\in\LL^2(\mu^D)$ is a 
positive function such that
$P_t[\varphi]=\theta\varphi$ then is $\theta=\theta_t$ and $\varphi$ is proportional to $\varphi_t$.
This suggests to consider the renormalization $\mu^D[\varphi_t^2]=1$, so that $\varphi_t$ is uniquely determined (being positive).
From the previous property, we deduce that for all $t\geq T_2$ and
 all $n\in\NN$, $\varphi_{nt}=\varphi_t$ and $\theta_{nt}=\theta^n_t$. Indeed, it is sufficient to note that
 \bq
 P_{nt}^D[\varphi_t]&=&(P_t^D)^n[\varphi_t]\\
 &=&\theta_t^n\varphi_{t}\eq
 We deduce that for any $r\in \QQ\cap[1,+\iy)$, $\varphi_{T_2r}=\varphi_{T_2}$ and $\theta_{{T_2}r}=\theta_{T_2}^r$:
 write $r=p/q$ with $p,q\in\NN$ and note that $\varphi_{T_2}=\varphi_{p{T_2}}=\varphi_{qr{T_2}}=\varphi_{r{T_2}}$
 and similarly  $\theta_{T_2}^p=\theta_{r{T_2}}^q=\theta_{pT_2}$.
Let us define $\varphi\df P_{T_2}^D\varphi_{T_2}=\theta_{T_2}\varphi_{T_2}$. Since ${T_2}>0$ and $P_{T_2}^D(\LL^2(\mu^D))$ is included in the domain of $L^D$,
we have $\varphi\in \cD(L^D)$. Furthermore from the general Hille-Yoshida theory we have in $\LL^2(\mu^D)$,
\bq
\lim_{t\ri 0_+} \frac{P_{{T_2}+t}^D[\varphi_{T_2}]-P_{{T_2}}^D[\varphi_{T_2}]}{t}&=&L^D[P_{T_2}[\varphi_{T_2}]]\eq
Thus considering $t$ of the form $q{T_2}$ with $q\in\QQ_+$ going to zero, we deduce
that
\bq
L^D[\varphi]&=&\lim_{q\in\QQ,\,q\ri 0_+}\frac{\theta_{T_2}^{q+1}-\theta_{T_2}}{{T_2}q}\varphi_{T_2}\\
&=&\theta_{T_2}\frac{\ln(\theta_{T_2})}{{T_2}}\varphi_{T_2}\\
&=&\frac{\ln(\theta_{T_2})}{{T_2}}\varphi\eq
It remains to set $\lambda_0(D)=-\ln(\theta_{T_2})/{{T_2}}$. Since $\theta_{T_2}$ is the spectral norm of the contraction operator $P_{T_2}$,
it appears that $\lambda_0(D)\geq 0$.
 The first bound of Theorem \ref{T1} enables to check that $\lambda_0(D)> 0$: from Cauchy-Schwarz inequality, we get that 
 for all $f\in\LL^2(\mu^D)$ and all $z\in D$, 
 \bq (P^D_{T_2}[f])^2(z)&\leq& P^D_{T_2}[f^2](z)P_{T_2}^D[\un_D](z)\\
 &\leq &  P^D_{T_2}[f^2](z)\sup_{z'\in D}P_{T_2}^D[\un_D](z')\eq
it follows that 
 \bq
 \mu^D[(P^D_{T_2}[f])^2]& \leq&\sup_{z\in D}P_{T_2}^D[\un_D](z)\mu^D[P^D_{T_2}[f^2]]\\
& \leq& \sup_{z\in D}P_{T_2}^D[\un_D](z)\mu^D[f^2]\eq
 So the norm operator of $P^D_{T_2}$ satisfies
 \bqn{normPD}
\theta_{T_2}\ =\  \lVe P^D_{T_2}\rVe_{\LL^2(\mu^D)\ri\LL^2(\mu^D)}\ \leq \  \sup_{z\in D}P_{T_2}^D[\un_D](z)\ =\ \sup_{z\in D}\PP_z[\tau>{T_2}]\eqn
 which itself is strictly less than 1 for ${T_2}$ large enough. Up to the choice of such a ${T_2}$ in the above arguments,
 we conclude that $\lambda_0(D)> 0$.\par\sm
Let us now check that $\varphi\in\bigcap_{r\geq 1}\LL^r(\mu^D)$, since a priori we only know that
$\varphi\in \LL^2(\mu^D)=\bigcap_{r\in[1,2]}\LL^r(\mu^D)$.
This is due to the hyperboundedness of $(P_t^D)_{t\geq 0}$.
Let $r>2$ be given and a corresponding $T_r>0$ such that the conclusion of Lemma \ref{hyperalamain} is satisfied.
Let $f\in\LL^2(\mu^D)$ be given. Cauchy-Schwarz and H\"older inequalities imply that for all $z\in D$ and all $t\geq T_r$,
\bq
(P_t^D[f](z))^r&=& \lt(\int f(z')p_t^D(z,z')\,\mu^D(dz')\rt)^r\\
&\leq &  \lt(\int f^2(z')\,\mu^D(dz')\rt)^\frac{r}{2} \lt(\int (p_t^D(z,z'))^2\,\mu^D(dz')\rt)^\frac{r}{2} \\
&\leq &  \lt(\int f^2(z')\,\mu^D(dz')\rt)^\frac{r}{2} \lt(\int (p_t^D(z,z'))^r\,\mu^D(dz')\rt)\eq
Integrating this bound with respect to $\mu^D(dz)$, it follows that
\bq
\lt(\int (P_t^D[f])^r\,d\mu^D\rt)^{\frac1r}&\leq &  \lt(\int (p_t^D(z,z'))^r\,\mu^D(dz)\mu^D(dz')\rt)^{\frac1r} \lt(\int f^2(z')\,\mu^Ddz')\rt)^\frac{1}{2}\eq
namely $P_t^D$ send continuously $\LL^2(\mu^D)$ into $\LL^r(\mu^D)$.
If furthermore $t$ is of the form $T_2q$ with $q\in\QQ\cap [1,+\iy)$, we get from $\varphi_{T_2q}=P_{T_2q}^D[\varphi_{T_2q}]/\theta_{T_2q}$ that $\varphi=\varphi_{T_2q}$ belongs to $\LL^r(\mu^D)$. \par
The same arguments are also valid for the adjoint semigroup $(P_t^{D*})_{t\geq 0}$. Its elements for $t>0$ admit the kernels $p^{D*}_t$
where
\bq
\fo t>0,\,\fo z,z'\in D,\qquad
p^{D*}_t(z,z')\ \df\  \frac{\mu^D(z')p_t^{D}(z',z)}{\mu^D(z)}\ =\ \frac{\mu(z')p_t^{D}(z',z)}{\mu(z)}\eq
We end up with the same quantity $\lambda_0(D)$, since for any $t>0$ the operators $P_t^{D}$ and $P_t^{D^*}$ have the same
spectral radius.
\wwtbp
\par
Let $\nu^D$ be the probability measure on $D$ which admits $\varphi^*/\mu^D[\varphi^*]$ as density with respect to $\mu^D$. 
Next result shows the validity of (\ref{persistence1}):
\begin{pro}
The probability measure $\nu^D$ is a quasi-stationary distribution for $L^D$ and
under $\PP_{\nu^D}$, $\tau$ is distributed as an exponential law of parameter $\lambda_0(D)$.
\end{pro}
\proof
Let a test function $f\in\cD(L^D)$ be given. 
We compute that for all $t\geq 0$,
\bq
\pa_t \nu^D[P_t^D[f]]&=&\nu^D[L^DP_t^D[f]]\\
&=&\mu^D[\varphi^*L^DP_t^D[f]]/\mu^D[\varphi^*]\\
&=&\mu^D[L^{D*}[\varphi^*]P_t^D[f]]/\mu^D[\varphi^*]\\
&=&-\lambda_0(D)\mu^D[[\varphi^*]P_t^D[f]]/\mu^D[\varphi^*]\\
&=&-\lambda_0(D) \nu^D[P_t^D[f]]\eq
By integration it follows that 
\bq
\nu^D[P_t^D[f]]&=&\exp(-\lambda_0(D)t)\nu^D[f]\eq
at least for $f\in\cD(L^D)$, but by usual approximation procedures, this can be extended
to any $f$ which is measurable and bounded (or nonnegative). 
This means that $\nu^D$ is a quasi-stationary distribution for $L^D$ with rate $\lambda_0(D)$.
In particular with $f=\un_D$, we get
\bq
\PP_{\nu^D}[\tau>t]&=&\nu^D[P_t^D[\un_D]]\\
&=&\exp(-\lambda_0(D)t)\nu^D[\un_D]\\
&=&\exp(-\lambda_0(D)t)\eq
which amounts to $\tau$ being distributed as an exponential law of parameter $\lambda_0(D)$ under $\PP_{\nu^D}$.\wwtbp
\par
The bounds (\ref{persistence2}) are now easy to deduce. Indeed recalling the definition of $\lambda_0(D)$ in terms of $\theta_{T_2}$
given in the proof of Proposition \ref{KR} (and the fact that $T_2$ can be chosen arbitrary large), we get  from the first bound 
of Theorem \ref{T1} that $\lambda_0(D)\geq{\ln(2)}\omega/\pi$.


The second bound of Theorem \ref{T1} applied with $m_0=\nu^D$ gives that $\lambda_0(D)\leq 4$.\par\me
\begin{rem}
Is $\nu^D$ the unique quasi-stationary probability measure associated to $L^D$?
A priori one has to be careful since this is wrong for the usual one-dimensional Ornstein-Uhlenbeck process with respect to a half-line.
Nevertheless we believe there is uniqueness in our situation, because it is easy for the underlying process to get out of $D$ uniformly over the starting point
(as shown by the first bound of Theorem \ref{T1}) and this should be a sufficient condition (in the spirit of Section 7.7 of  the book \cite{MR2986807} of Collet, Mart{\'{\i}}nez and San
              Mart{\'{\i}}n, which unfortunately only treat the case of one-dimensional diffusions).
              At least from the uniqueness statement included in Krein-Rutman theorem
(cf.\ again Theorem 1.2 of Du \cite{ MR2205529}), we deduce that $\nu^D$ is the unique quasi-stationary measure admitting
a density with respect to $\mu^D$ which is in $\LL^2(\mu^D)$.
By hyperboundedness of $(P^D_t)_{t\geq 0}$, the latter condition can be relaxed by only requiring that the density belongs to $\bigcap_{p>1}\LL^p(\mu^D)$.
\end{rem}

\section{Computations in polar coordinates}

For the sake of completeness, we give below a series of elementary but tedious computations which are omitted in Section \ref{Dee}. We start with the proof of \eqref{opellippol} and \eqref{operhypo}
\smallskip
\begin{pro}\label{prop:appendix1}
In the usual polar coordinates $(r,\theta)$,
the infinitesimal generator $L_{\rho}$ of the elliptic diffusion whose evolution is described by (\ref{eq:xitt}) is given by
$$
\forall g \in {\cal C}^2(\ER_+^*\times\ER), \qquad 
L_\rho g(r,\theta)=-\rho r\partial_r g(r,\theta) +\partial_\theta g(r,\theta)+\partial ^2_r g(r,\theta)+\frac{1}{r}\partial_r g(r,\theta)+\frac{\partial^2_\theta}{r^2}g(r,\theta).
$$
In a similar way, the action infinitesimal generator $\cal{L}_{\rho}$ of the hypo-elliptic diffusion described by (\ref{eq:hypocoercif}) is given by
$$
{\cal L}_\rho =-\rho r \partial_r  +\partial_\theta +\frac{\sin^2\theta}{2}\partial^2_{rr}-\frac{\sin\theta \cos\theta}{r^2}\partial_\theta
 +\frac{\sin\theta \cos\theta}{r}\partial^2_{r\theta}+\frac{\cos^2\theta}{2r}\partial_r+\frac{\cos^2\theta}{2r^2}\partial^2_\theta.
$$

\end{pro}
\begin{proof}
 First, write $g(r,\theta)=f(r\cos\theta,r\sin\theta)$. Using that
\begin{align*}
\partial_x f=\cos \theta \partial_r g-\frac{\sin\theta}{r}\partial_\theta g,
\quad \partial_y f=\sin \theta\partial_r g+\frac{\cos\theta}{r}\partial_\theta g,
\end{align*}
one checks that,
$$-(\rho x+y)\partial_x f+(x-\rho y) \partial_y f=-\rho r\partial_r g+\partial_\theta g.$$
Second,
$$\partial^2_{x}f={\cos^2\theta}\partial^2_{r} g+2\frac{\sin\theta \cos\theta}{r^2}\partial_\theta g
 -2\frac{\sin\theta \cos\theta}{r}\partial^2_{r\theta}g+\frac{\sin^2\theta}{r}\partial_rg+\frac{\sin^2\theta}{r^2}\partial^2_\theta g,$$
 
$$\partial^2_{y}f={\sin^2\theta}\partial^2_{r}g-2\frac{\sin\theta \cos\theta}{r^2}\partial_\theta g
 +2\frac{\sin\theta \cos\theta}{r}\partial^2_{r\theta} g+\frac{\cos^2\theta}{r}\partial_r g +\frac{\cos^2\theta}{r^2}\partial^2_\theta g.$$
The expressions \eqref{opellippol} and \eqref{operhypo} follow. \wwtbp
\end{proof}
\smallskip

%


\noindent \textbf{Proofs of Proposition \ref{prop:lrhog_elliptic} and Proposition \ref{prop:lrhog_hypo}} 
In Subsections \ref{subsec:ellip} and \ref{subsec:hypo}, we need to compute $\frac{L_\rho g}{g}$ and $\frac{{\cal L}_\rho g}{g}$ where $g$ has the following form
$$g(r,\theta)=r^n{\gamma}(\theta) e^{\beta(\theta)r^2}.$$
Note that in Subsection \ref{subsec:ellip}, $n=1$ and $\gamma(\theta)=\cos\theta$. Then,  $\frac{L_\rho g}{g}$ and $\frac{{\cal L}_\rho g}{g}$ are expressed in terms of some functions denoted by $\psi_1$, $\psi_2$, $\varphi_i$, $i=1,2,3$. The computation of these functions follows from those of the derivatives of $g$ given below:
\begin{align*}
&\frac{\partial_r {g}}{{g}}(r,\theta)=\left(\frac{n}{r}+2\beta(\theta) r\right),\\
&\frac{\partial_{r}^2 {g}}{{g}}(r,\theta)=\frac{n^2-n}{r^2}+4r^2\beta^2(\theta)+(4n+2)\beta(\theta),\\
&\frac{\partial_\theta {g}}{{g}}(r,\theta)=\beta'(\theta) r^2+\frac{{\gamma}'(\theta)}{{\gamma}(\theta)},\\
&\frac{\partial^2_{r\theta} {g}}{r{g}}(r,\theta)=2\beta(\theta)\beta'(\theta)r^2+\left((2+n)\beta'(\theta)+2\frac{{\gamma}'(\theta)}{{\gamma}(\theta)}\beta(\theta)\right)+n\frac{{\gamma}'(\theta)}{r^2{\gamma}(\theta)}\\
&\frac{1}{r^2}\frac{\partial_\theta^2 {g}}{{g}}(r,\theta)=\beta'(\theta)^2 r^2 +\left(\beta''(\theta)+2 \beta'(\theta)\frac{{\gamma}'(\theta)}{{\gamma}(\theta)}\right)+\frac{1}{r^2}\frac{{\gamma}''(\theta)}{{\gamma}(\theta)}.
\end{align*}
We can now use carefully the expressions of the elliptic (resp. hypo-elliptic) generator $L_{\rho}$ (resp. $\cal{L}_\rho$)  given by \eqref{opellippol} (resp. \eqref{operhypo}). \wwtbp



 \bibliography{bubblebib}
 \bibliographystyle{plain}

\vskip2cm
\hskip70mm\box5

\end{document}